\documentclass[11pt]{article}

\usepackage{amsmath,amsthm,amssymb,mathtools,bm,mathrsfs,microtype}
\usepackage{graphicx}
\usepackage{enumitem}
\usepackage{stmaryrd}
\usepackage{aop_like_compact} 
\usepackage[colorlinks,citecolor=blue,linkcolor=blue,urlcolor=blue]{hyperref}
\theoremstyle{plain}
\newtheorem{theorem}{Theorem}[section]
\newtheorem{lemma}[theorem]{Lemma}
\newtheorem{prop}[theorem]{Proposition}
\newtheorem{coro}[theorem]{Corollary}
\theoremstyle{definition}
\newtheorem{definition}[theorem]{Definition}

\numberwithin{equation}{section}

\DeclareMathOperator{\dive}{div}

\DeclareMathOperator{\curl}{curl}

\DeclareMathOperator{\loc}{loc}

\DeclareMathOperator{\realpart}{Re}
\DeclareMathOperator{\imaginarypart}{Im}

\DeclareMathOperator{\Det}{det}
\DeclareMathOperator{\diag}{diag}
\DeclareMathOperator{\row}{Row}

\DeclareMathOperator{\rem}{rem}
\DeclareMathOperator{\eff}{eff}

\DeclareMathOperator{\ini}{in}
\DeclareMathOperator{\cut}{cut}
\DeclareMathOperator{\bd}{bd}

\DeclareMathOperator{\bulk}{bulk}
\DeclareMathOperator{\even}{even}

\newcommand{\fa}{{\mathfrak a}}
\newcommand{\fq}{{\mathfrak q}}
\newcommand{\pp}{{\mathfrak p}}
\newcommand{\PP}{{\mathbb P}}
\newcommand{\PPP}{{\mathrm P}}
\newcommand{\PPPP}{{\mathcal P}}
\newcommand{\Z}{{\mathbb Z}}
\newcommand{\z}{{\mathfrak z}}
\newcommand{\V}{\mathcal{V}}
\newcommand{\rr}{\mathfrak{r}}
\newcommand{\R}{{\mathbb R}}
\newcommand{\RR}{{\mathcal R}}
\newcommand{\fd}{{\mathfrak d}}
\newcommand{\ddd}{{\mathrm d}}
\newcommand{\DDD}{{\mathcal D}}
\newcommand{\DDDD}{{\mathscr D}}
\newcommand{\FD}{{\mathfrak D}}
\newcommand{\C}{{\mathbb C}}
\newcommand{\CC}{{\mathcal C}}
\newcommand{\N}{{\mathbb N}}
\newcommand{\NN}{{\mathcal N}}
\newcommand{\fn}{{\mathfrak n}}
\newcommand{\fm}{{\mathfrak m}}
\newcommand{\MM}{{\mathcal M}}
\newcommand{\HH}{{\mathcal H}}
\newcommand{\fh}{{\mathfrak h}}
\newcommand{\I}{{\mathbb I}}

\newcommand{\tX}{{\tilde X}}

\newcommand{\XX}{{\mathcal X}}
\newcommand{\XXX}{{\mathfrak X}}
\newcommand{\YYY}{{\mathcal Y}}
\newcommand{\UU}{{\mathcal U}}
\newcommand{\fu}{{\mathfrak u}}
\newcommand{\ttt}{{\mathfrak t}}
\newcommand{\T}{{\mathbb T}}
\newcommand{\TT}{{\mathcal T}}
\newcommand{\FT}{{\mathfrak T}}
\newcommand{\SSSSS}{{\mathscr S}}
\newcommand{\SSSS}{{\mathcal S}}
\newcommand{\FS}{{\mathfrak S}}

\newcommand{\s}{{\mathfrak s}}
\newcommand{\E}{{\mathbb E}}
\newcommand{\EE}{{\mathcal E}}
\newcommand{\F}{{\mathbb F}}

\newcommand{\FFF}{{\mathscr F}}
\newcommand{\hf}{{\hat f}}
\newcommand{\tf}{{\tilde f}}
\newcommand{\ff}{{\mathfrak f}}
\newcommand{\LL}{{\mathcal L}}
\newcommand{\LLL}{{\mathfrak L}}
\newcommand{\llll}{{\mathfrak l}}

\newcommand{\FJ}{{\mathfrak J}}

\newcommand{\bg}{{\bar g}}

\newcommand{\FG}{{\mathfrak G}}

\newcommand{\BB}{{\mathcal B}}
\newcommand{\fb}{{\mathfrak b}}
\newcommand{\A}{{\mathcal A}}
\newcommand{\FA}{{\mathfrak A}}
\newcommand{\tw}{{\tilde w}}
\newcommand{\bw}{{\bar w}}
\newcommand{\hw}{{\hat w}}

\newcommand{\htw}{{\hat{\tilde w}}}
\newcommand{\thw}{{\tilde{\hat w}}}

\newcommand{\TIH}{{\widetilde H}}

\newcommand{\fc}{{\mathfrak c}}
\newcommand{\hk}{{\hat k}}
\newcommand{\ssum}{\mathop{\textstyle\sum}\nolimits}

\title{Delayed blow-up by transport noise for the 3D Navier--Stokes equation with Navier-slip boundary conditions}

\begin{aug}
\author[A]{\fnms{Meng}~\snm{Zhao}\ead[label=e1]{mzhao1009@gmail.com}}
\address[A]{Department of Mathematics, CY Cergy Paris University, CNRS UMR 8088, 2 avenue Adolphe Chauvin, 95302 Cergy-Pontoise, France\printead[presep={,\ }]{e1}}
\end{aug}

\date{\today}

\begin{document}
\maketitle

\begin{abstract}
We study the vorticity formulation of the 3D Navier--Stokes equation driven by transport noise in a periodic channel with Navier-slip boundary conditions. We consider both non-degenerate transport noise and degenerate tangential transport noise. For any prescribed $T>0$ and $\epsilon>0$, we prove that, by choosing the noise intensity sufficiently large and concentrating the noise on sufficiently high modes, the solution exists up to $T$ with probability at least $1-\epsilon$.

A main contribution of this work is to identify and analyze the interaction between enhanced dissipation induced by transport noise and physical boundary effects. The no-flux condition breaks the isotropy of the noise and changes the scaling limit of the Itô--Stratonovich corrector. In the non-degenerate case, a boundary feedback term appears in the limiting effective operator; in the degenerate case, the limiting operator is a nonlocal anisotropic tangential dissipation. The proof is based on a combination of a boundary correction operator, a Meyers-type estimate, a scaling-limit analysis of the Itô--Stratonovich corrector, and resolvent estimates for the deterministic limiting equations.


\end{abstract}

\tableofcontents

\section{Introduction}

In this paper, we study the vorticity formulation of the 3D Navier--Stokes (NS, for short) equation driven by transport noise in the periodic channel $D:=\{(x,y)|x\in\T^2, y\in(0,1)\}$ with Navier-slip boundary conditions
    \begin{align}
    \label{INTRO1}
    \begin{cases}
    \partial_tw-\nu\Delta w+(u\cdot \nabla)w-(w\cdot\nabla)u=\sqrt{2\kappa}\Pi (\dot{W}(t,x,y)\circ \nabla)w,\\u=\curl^{-1}w,\\(\partial_y w_3-\alpha w_3)|_{y=0}=(w_h-\alpha 
    u_h^\perp)|_{y=0}=0,\\(\partial_y w_3+\alpha w_3)|_{y=1}=(w_h+\alpha 
    u_h^\perp)|_{y=1}=0,\\
    w|_{t=0}=w_{\ini}.
    \end{cases}
    \end{align} 
    Here, $\nu>0$ is the viscosity, $\alpha>0$ is the friction parameter of the boundary, $\curl^{-1}$ is the Biot--Savart operator, and $w,u$ denote the vorticity and the corresponding velocity field, respectively, with $w_h,u_h$ being the tangential components and $u^\perp_h:=(-u_2,u_1)$. The driving noise $W(t,x,y)$ is a divergence-free Gaussian random field, whose precise definition will be given below. The parameter $\kappa>0$ characterizes the noise intensity, the symbol $\circ$ denotes the Stratonovich integral, and $\Pi:L^2 \to H_{\curl}$ is a projection operator, where
    \begin{align}\label{INTRO1-1}
        H_{\curl}:=\bigg\{w\in L^2\bigg| \dive w=0,\ \int_{\T^2} w_3(x,y)\ddd x=0\mbox{ a.e. $y\in(0,1)$}\bigg\}
    \end{align}
    is the curl-admissible class endowed with the usual $L^2$-norm $\|\cdot\|$.
    
    \subsection{The noise structure}
    We consider two types of Gaussian random fields $W$. The first is the non-degenerate case, in which each component of $W$ is not identically zero and $W$ depends on both the tangential and vertical variables $(x,y)$. The second is the degenerate case, in which the vertical component vanishes and $W$ depends only on the tangential variables $x$, namely
    \[W:=(W_1(t,x),W_2(t,x),0).\]
    In this case, the stochastic transport reduces to
    \[(\dot{W}\circ \nabla)w=(\dot{W}_h\circ \nabla_x)w,\]
    and therefore acts only in the tangential directions. Such a stochastic transport is naturally relevant in situations where the dominant motion is tangential, for instance in oceanic and atmospheric circulations.
    \subsubsection*{The non-degenerate transport noise}
    Let
    \[k:=(k_h,m/2)=(k_1,k_2,m/2),\qquad  \mbox{ $k_h\in\mathbb Z^2$, $\ m\in\mathbb Z$}\]
    be a wave number. For $k\neq0$, define $a_{k,1}, a_{k,2}$ as unit vectors in $\R^3$ satisfying 
    \begin{align}
        \label{INTRO1-2}a_{k,1}\cdot a_{k,2}=0,\qquad a_{k,j}\cdot k=0,\qquad a_{k,j}=a_{-k,j},\qquad j=1,2.
    \end{align}
    In particular, for $m=0$, we define 
    \[a_{k,1}:=(a_{k,1,h},0),\qquad a_{k,2}:=(0,0,1),\]
    where $a_{k,1,h}\cdot k_h=0$ and $a_{k,1,h}=a_{-k,1,h}$. By definition, one has
    \begin{align}
        \label{INTRO2}
        \ssum_{j=1,2} a_{k,j}\otimes a_{k,j}= I-\frac{k\otimes k}{|k|^2},\qquad k\neq0,
    \end{align}
    where $ k\otimes k:= kk^T$ is the usual tensor product. Introduce the solenoidal Fourier sine-cosine modes
    \begin{align}
        \label{INTRO3}\sigma_{k,j}(x,y):=c_{k}e^{2\pi i k_h\cdot x}\left(\begin{matrix}
        a_{k,j,1}\cos m\pi y\\
        a_{k,j,2}\cos m\pi y\\
        i a_{k,j,3} \sin m\pi y
    \end{matrix}\right)
    \end{align}
    with $c_k=\sqrt{2}$ if $m\neq0$ and $c_k=1$ if $m=0$. Let $(\Omega,\FFF,\PP)$ be a probability space with a filtration $\F:=\{\FFF_t\}_{t\ge 0}$ satisfying the usual conditions. Suppose that $\{B^{k,j}_\cdot\}_{k,j}$ is a sequence of complex Brownian motions defined on $\Omega$, which satisfies
    \begin{align}
        \label{INTRO4-2}\overline{B_t^{-k,j}}=B_{t}^{k,j},\qquad \llbracket B_{\cdot}^{k,j},B_{\cdot}^{k',j'}\rrbracket_t=2t\delta_{k+k'=0}\delta_{j=j'}.
    \end{align} 
    Here, $\llbracket \cdot, \cdot \rrbracket$ denotes the quadratic covariation. We define the non-degenerate Gaussian random field by
    \begin{align}
        \label{INTRO4}W(t,x,y):=\ssum_{k\neq 0}\ssum_{j=1,2}\theta_k \sigma_{k,j}B_t^{k,j},
    \end{align}
    where $\{\theta_k\}_{k\neq 0}$ is a sequence of real numbers satisfying
    \begin{align}\label{INTRO4-1}
        \|\theta_\cdot\|_{\ell^2}=1\qquad\mbox{and}\qquad \theta_{k}=\theta_{k'}, \qquad\mbox{if $|k|=|k'|$}. 
    \end{align}
    Under the above notation, the transport noise takes the form
    \begin{align*}\sqrt{2\kappa}\Pi (\dot{W}\circ \nabla)w=\sqrt{2\kappa}\ssum_{k\neq0}\ssum_{j=1,2}\theta_{k}\Pi(\sigma_{k,j}\cdot\nabla)w\circ \dot{B}_{t}^{k,j}.
    \end{align*}
    
    \subsubsection*{The degenerate tangential transport noise}
    Next, we introduce the degenerate tangential transport noise. Let $\Z^2_0:= \Z^2-\{0\}$. Define the solenoidal tangential Fourier modes
    \begin{align}\label{INTRO5}
    \sigma_k(x):=\begin{cases}\frac{k^\perp}{|k|}e^{2\pi ik\cdot x},\qquad k\in \Z_{0,+}^2:=\{k\in\Z_0^2| \mbox{$k_2>0$ or $k_2=0, k_1>0$}\},\\
    -\frac{k^\perp}{|k|}e^{2\pi ik\cdot x},\qquad k\in \Z_{0,-}^2:=\Z_0^2-\Z_{0,+}^2.
    \end{cases}
    \end{align}
    Fix a filtered probability space $(\Omega,\FFF,\F,\PP)$ with $\F:=\{\FFF_t\}_{t\ge 0}$ satisfying the usual conditions. Suppose that $\{B^k_\cdot\}_{k}$ is a sequence of complex Brownian motions defined on $\Omega$, which satisfies
    \[\overline{B_t^{-k}}=B_{t}^{k},\qquad \llbracket B_{\cdot}^{k},B_{\cdot}^{k'}\rrbracket_t=2t\delta_{k+k'=0}.\]
    The degenerate random field $W$ is defined by
    \begin{align}
    \label{INTRO6}
    W(t,x,y):=(W_h(t,x),0
    ),\qquad W_h(t,x):=\ssum_{k\neq 0}\theta_k\sigma_k(x)B^k_t,
    \end{align}
    where $\{\theta_k\}_{k\neq 0}$ satisfies \eqref{INTRO4-1}. Therefore, the transport noise reduces to
    \begin{align}
        \label{INTRO6-1}\sqrt{2\kappa}\Pi (\dot{W}_h\circ \nabla_x)w=\sqrt{2\kappa}\ssum_{k\neq0}\theta_k\Pi(\sigma_k\cdot\nabla_x)w\circ \dot{B}_{t}^{k}.
    \end{align}
    \subsection{The main result and literature review}
    To present our main result, we first introduce the notion of solutions used in this paper.
    \begin{definition}\label{definition1}Let $\tau:\Omega\to [0,\infty)$ and $w:\Omega\times[0,\tau)\times D\to \R^3$ be a stopping time and a progressively measurable process, respectively. We call $(w,\tau)$ a solution of \eqref{INTRO1}, if the following conditions are satisfied.
    \begin{enumerate}
        \item $w(\cdot\wedge\tau)\in L^\infty(0,\infty;H_{\curl})\cap L^2(0,\infty;H^1\cap H_{\curl})$ $\PP$-almost surely.
        \item The boundary conditions of $w_h$ are satisfied in the sense of traces, namely,
        \[(w_h-\alpha 
    u_h^\perp)|_{y=0}=(w_h+\alpha 
    u_h^\perp)|_{y=1}=0\]
    in $L^2(0,\tau;H^{1/2}(\T^2))$ $\PP$-almost surely. Here, $u:=\curl^{-1}w$.
    \item For any test function 
    \begin{align}
        \label{INTRO4-3}
        \varphi:=(\varphi_h,\varphi_3)\in C^\infty\cap H_{\curl}\qquad \mbox{with}\qquad \varphi_h\in C_{c}^\infty(D),
    \end{align}
    the following identity holds $\PP$-almost surely for any $t\ge0$:
    \begin{align}
        \label{INTRO4-4}
        \langle w(t\wedge\tau),&\varphi\rangle+\int_0^{t\wedge\tau}\!\!\left(\nu\langle \nabla w,\nabla \varphi\rangle+\nu\alpha\langle w_3,\varphi_3\rangle_{L^2(\partial D)}+\langle (u\cdot\nabla)w-(w\cdot \nabla u),\varphi\rangle\right)\ddd s\notag\\&= \langle w_{\ini},\varphi\rangle+\sqrt{2\kappa}\ssum_{k\neq 0}\ssum_{j=1,2}\int_0^{t\wedge\tau}\theta_k\langle \Pi(\sigma_{k,j}\cdot\nabla)w,\varphi\rangle\circ \ddd B_{s}^{k,j}.
    \end{align}
    \end{enumerate}
    In the degenerate case, the solution is defined as above, with the noise term in \eqref{INTRO4-4} replaced by the right-hand side of \eqref{INTRO6-1}.
    \end{definition}
    Define the anisotropic Sobolev space 
    \begin{align}
        \label{INTRO7}\XX^m:=\Big\{f\in L^2\Big|\|f\|_{m,0}^2:=\ssum_{|l|\le m}\|\partial_x^l f\|^2<\infty\Big\},\qquad m\in\N.
    \end{align}
    The main result is presented below.
    \begin{theorem}\label{MT}
        Let $\alpha,\nu>0$, $m\ge 2$. Suppose that $W$ is either the non-degenerate random field defined in \eqref{INTRO4}, or the degenerate random field defined in \eqref{INTRO6}. Then, for any $\epsilon,T>0$ and $\FFF_0$-measurable initial data 
        \[w_{\ini}\in L^\infty(\Omega;H_{\curl}\cap H^1\cap \XX^m),\]
        there exist $\kappa>0$ independent of $T$ and $\theta\in \ell^2$ supported in sufficiently high modes and satisfying \eqref{INTRO4-1} such that the stochastic NS equation \eqref{INTRO1} has a solution $(w,\tau)$, which exists up to time $T$ with high probability:
        \[\PP\{\tau\ge T\}\ge 1-\epsilon.\]
    \end{theorem}
    To the best of our knowledge, this is the first result that addresses the interaction between the enhanced dissipation induced by transport noise and physical boundary effects in nonlinear fluid equations. The key new feature is that, due to the presence of the Navier boundary conditions, the stabilizing mechanism generated by transport noise does not merely add bulk dissipation; it also interacts with the boundary vorticity generation mechanism. Consequently, the structure of the effective dissipation operator arising in the scaling limit differs from its boundaryless counterpart. Understanding this interaction is one of the central themes of the present work.

    The stabilizing effect of transport noise has attracted considerable attention in recent years. A central mechanism behind this phenomenon is that, under a suitable high-frequency scaling, the Itô--Stratonovich corrector converges to an effective dissipation operator. In many boundaryless settings, this limiting operator acts as an additional viscosity and thus produces enhanced dissipation. This mechanism was first identified by Galeati \cite{Gal20} for transport equations with high-frequency transport noise, where solutions were shown to converge to solutions of a deterministic parabolic equation. It was later applied by Flandoli, Galeati, and Luo \cite{FGL21} to nonlinear PDEs with possible finite-time blow-up, leading to delayed blow-up and global existence with high probability. For fluid equations, this mechanism was first applied by Flandoli and Luo \cite{FL21} to the NS equation on $\T^3$. More recently, by using maximal regularity and Meyers-type estimates, Agresti further developed this mechanism, obtaining anomalous dissipation for the NS equation on $\T^2$, and a delayed blow-up result for the hyperviscous NS equation on $\T^3$; see \cite{Agr25} and \cite{Agr26}.

    Let us also mention a related research topic on Kraichnan-type noise. In the Kraichnan model, a passive scalar is transported by a Gaussian velocity field which is white-in-time and correlated-in-space. This model provides a classical paradigm for turbulent mixing and anomalous dissipation. Recent mathematical developments include anomalous dissipation and regularization for fluid equations; see, e.g., \cite{CM23,Row24,BGM25,GL25,GGM26} and the references therein.

    However, much less is known in the presence of physical boundaries. Existing results in this direction mainly concern linear passive scalar or heat-type equations, where transport noise still generates enhanced dissipation despite boundary effects; see \cite{FGL22,FL22}. To the best of our knowledge, delayed blow-up or global existence results by transport noise for nonlinear fluid equations with physical boundary conditions have not been available so far.

    The presence of a physical boundary creates several difficulties which are absent in the periodic setting. First, the vorticity formulation is coupled with a nonhomogeneous mixed Dirichlet--Robin boundary condition, which is nonlocal through the Biot--Savart law. Thus, the vorticity Laplacian is not directly self-adjoint, and the basic energy estimate has to be recovered through a boundary correction. Second, the projection $\Pi$ onto the curl-admissible class $H_{\curl}$ is more delicate than the usual Leray projection on the torus. More precisely, the gradient part in the Helmholtz decomposition \eqref{A0} solves an elliptic boundary value problem, and its boundary condition produces boundary layer terms in the scaling limit of the Itô--Stratonovich corrector. Third, the random field $W$ has to satisfy the no-flux condition at the boundary, and is therefore not isotropic as in the periodic case. This boundary-induced anisotropy affects the limiting operator in different ways depending on the structure of the noise. In the non-degenerate case, a nontrivial boundary layer contribution survives in the scaling limit and produces an additional boundary feedback term; see \eqref{LE6}. In the degenerate case, the effective operator becomes a nonlocal anisotropic tangential dissipation; see \eqref{LE7}. Together with the nonhomogeneous vorticity boundary conditions, these boundary effects also make the limiting deterministic equation harder to analyze. In particular, the enhanced dissipation cannot be read off directly from energy estimates, and we instead prove resolvent estimates for the corresponding effective dissipation operators and convert them into exponential decay by a semigroup stability criterion.
    \subsection{The scheme of the proof}
    We reformulate equation \eqref{INTRO1} into the following It\^o form
    \begin{align}
    \label{INTRO8}
    \begin{cases}
    \partial_tw-\nu\Delta w+(u\cdot \nabla)w-(w\cdot\nabla)u=S_\theta w+\sqrt{2\kappa}\Pi (\dot{W}\cdot \nabla)w,\\ u=\curl^{-1}w,\\(\partial_y w_3-\alpha w_3)|_{y=0}=(w_h-\alpha 
    u_h^\perp)|_{y=0}=0,\\(\partial_y w_3+\alpha w_3)|_{y=1}=(w_h+\alpha 
    u_h^\perp)|_{y=1}=0,\\
    w|_{t=0}=w_{\ini},
    \end{cases}
    \end{align} 
    where
    \begin{align}
    \label{INTRO9}
    S_{\theta}w:=
    \left\{
    \begin{array}{@{}l@{\quad}r@{}}
    2\kappa \ssum_{k\neq 0}\ssum_{j=1,2}\theta_k^2
    \Pi(\sigma_{k,j}\cdot\nabla)\Pi(\sigma_{-k,j}\cdot \nabla)w,&\text{in the non-degenerate case}
    \\[2mm]
    2\kappa \ssum_{k\neq 0}\theta_k^2
    \Pi(\sigma_k\cdot\nabla_x)\Pi(\sigma_{-k}\cdot \nabla_x)w,&\text{in the degenerate case}
    \end{array}\right. 
    \end{align}
    is the It\^o--Stratonovich corrector. It is important to note that $S_\theta$ does not yield a net dissipation at the stochastic level: in the energy estimate, its contribution is compensated by the quadratic variation of the transport noise. The actual stabilization appears only after passing to the high-frequency scaling limit, where $S_\theta$ converges to an effective deterministic dissipation operator.

    We now describe the main ideas of the proof. The argument consists of three steps, each of which is designed to overcome one of the boundary difficulties discussed above.

    The first step is a uniform-in-$\theta$ estimate for the stochastic linear Stokes equation \eqref{LS1}. This requires two ingredients. The first is a boundary correction. Due to the nonhomogeneous vorticity boundary conditions, the Laplacian $\Delta$ is not self-adjoint at the vorticity level. We therefore introduce a boundary correction operator $\BB$ and define the corrected vorticity $\tw:=(I-\BB)w$. Using the compactness of the operator $\BB$ combined with a Fredholm argument, we prove that $I-\BB$ is invertible. Thus, one may equivalently work with the corrected unknown $\tw$, which satisfies homogeneous boundary conditions and admits a uniform energy estimate.

    The second ingredient is a Meyers-type estimate in the periodic channel $D$. Here, a new difficulty appears because of the anisotropy induced by the physical boundary. In the periodic setting, the Itô--Stratonovich corrector can be decomposed into a principal Laplacian $\Delta$ plus a perturbative remainder. Then, one may absorb the remainder by a continuity argument combined with the maximal regularity of the Laplacian $\Delta$, as in the boundaryless theory; see \cite{Agr25}. In our case, this decomposition is no longer available in a useful form. Indeed, the leading part of the corrector behaves like a second-order operator with coefficients depending on $y$, rather than a constant multiple of the Laplacian $\Delta$. To avoid analyzing the maximal regularity of this variable-coefficient principal operator, we use the Sneiberg principle, see Theorem \ref{theoremLS4}, to directly lift the $L^2$-solvability to $L^p$-solvability with $p$ sufficiently close to $2$. This yields a uniform-in-$\theta$ Meyers estimate.

    In the second step, our goal is to pass to the high-frequency scaling limit of the stochastic NS equation \eqref{SL1} with cut-off. The tightness of the cut-off solutions follows directly from the uniform-in-$\theta$ estimate obtained in the first step, so the main task is to identify the limit of the Itô--Stratonovich corrector $S_\theta$. Since the random field $W$ has to satisfy the no-flux condition, the transport vector fields $\sigma$ entering the definition \eqref{INTRO9} of $S_\theta$ are tangent to the boundary. Consequently, the corrector does not by itself select any normal boundary condition. On the other hand, in the non-degenerate case, the bulk scaling limit of $S_\theta$ is a full Laplace operator $\Delta$, which has to be realized as a mixed Dirichlet--Robin Laplacian. This creates an unavoidable mismatch in boundary behavior between the corrector $S_\theta$ and its bulk scaling limit $\Delta$. The additional boundary feedback term in the limiting effective operator compensates precisely for this mismatch. By contrast, in the degenerate case, the limiting operator is a nonlocal anisotropic tangential dissipation. Since it involves no normal derivatives, the above boundary mismatch is absent, and no boundary feedback term appears.

    The third step is the analysis of the limiting deterministic equations \eqref{NDGWP1} and \eqref{DGWP1}. A direct energy estimate for the corrected unknown is not sufficient for our purpose, since the commutator between the boundary correction operator $\BB$ and the linear part may produce growth terms, which would lead to a smallness restriction on the friction parameter $\alpha$. To obtain a result for arbitrary $\alpha>0$, we instead use a resolvent approach. In the non-degenerate case, this requires studying the nonlocal resolvent equation~\eqref{NDGWP4}. A key observation is that the equation for the third component is closed, while the tangential components solve a local resolvent equation with a forcing term depending on the third component. We therefore first construct the third component by a superposition of two elementary one-dimensional boundary value problems, which yields the required estimate for $w_3$. Based on this estimate, we then prove the tangential resolvent bound by a contradiction argument. Finally, the resolvent estimate is converted into exponential decay of the linear semigroup by an exponential stability criterion; see Theorem~\ref{TheoremES}. In the degenerate case, since no nonlocal boundary feedback term appears, the resolvent estimate and the corresponding linear decay can be obtained more directly by contradiction.

    Combining the three steps, choosing $\kappa$ sufficiently large, and taking the noise coefficients $\theta$ supported on sufficiently high modes, the solution of the stochastic NS equation \eqref{INTRO1} remains close to that of the controlled deterministic limiting equations with high probability. This completes the proof of Theorem~\ref{MT}.

    \subsection{Further comments}
    We close the introduction with two comments. First, Theorem~\ref{MT} is stated on an arbitrary but fixed interval $[0,T]$, rather than on $[0,\infty)$. The obstruction is the lack of a small-data infinite-lifespan theory for the stochastic NS equation \eqref{INTRO1}. In the periodic case, since the Itô--Stratonovich corrector is exactly cancelled by the quadratic variation, and no boundary contribution appears, the $L^2$ estimate is essentially pathwise. Thus, the small-data global theory is close to the deterministic one. In our situation, however, as the boundary correction operator $\BB$ does not commute with the stochastic transport terms, the resulting commutators make the $L^2$-theory genuinely stochastic.

    Existing small-data theories do not seem to cover this regime. The results of Kukavica--Xu \cite{KX24} and Aydin--Kukavica--Xu \cite{AKX25} yield infinite-lifespan solutions with high probability, but the noise under consideration is of order $0$ and does not involve the gradient of the unknown. On the other hand, Agresti--Veraar \cite{AV24} allow transport-type noise, but the solution exists up to any fixed time $T$ with high probability. None of these results provides an infinite-lifespan theory applicable to our setting.
    
    Second, the present method is specific to the Navier boundary conditions and does not directly extend to the no-slip case. The reason is that at the level of estimates, transport noise behaves like a second-order operator. In the case of Navier boundary conditions, the boundary correction operator $\BB$ is of order $-1$, and therefore the commutators between $\BB$ and the stochastic transport terms are lower-order contributions. This is crucial for the uniform-in-$\theta$ energy estimate. For the no-slip boundary conditions, however, the corresponding boundary correction is expected to be of order $0$. Its commutator with the transport noise would then remain a leading-order term, and the cancellation between the Itô--Stratonovich corrector and the quadratic variation would no longer close the energy estimate uniformly in the noise coefficients. For this reason, the no-slip case would require a substantially different approach.

    \subsection*{Organization of the paper}
    The paper is organized as follows. Section~\ref{sec2} is devoted to the uniform-in-$\theta$ estimates for the stochastic linear Stokes equation. In Section~\ref{SLofNS}, we study the scaling limit of the stochastic NS equation with cut-off. Then, in Section~\ref{GWPofDNS}, we prove small-data global well-posedness for the limiting NS-type equations. Finally, in Section~\ref{sec5}, we combine all the results obtained in the previous sections to complete the proof of Theorem~\ref{MT}. The appendices collect the auxiliary results on the curl-admissible class, the Biot--Savart operator, the boundary correction operator, and the linear estimate used in Section~\ref{GWPofDNS}.

    \subsection*{Acknowledgements}
    The author was supported by the ANR project DYNACQUS through the grant ANR-24-CE40-5714-01, and by the National Natural Science Foundation of China under Grant Nos. 12331008, 12571156.
    \subsection*{Notation}
    We shall use the following standard notations.
	\begin{itemize}
    \item For $q\in [1,\infty]$, $L^q$ is the usual Lebesgue space endowed with the norm $\|\cdot\|_q$. For simplicity, we write $\|\cdot\|:=\|\cdot\|_2$. For $s\ge 0$, $H^{s,q}$ denotes the Bessel potential space with the conventions $H^{0,q}:=L^q$ and $H^{s,2}:=H^s$. It is known that $H^{m,q}=W^{m,q}$ for all $m\in\N$ and $q\in (1,\infty)$, where $W^{m,q}$ denotes the classical Sobolev space. For $p\in [1,\infty]$, $B^s_{q,p}$ denotes the Besov space.
    \item For $q\in(1,\infty)$ and integer $m\ge 1$, we define
    \begin{align}\label{notation0}H^{-m,q}:=(H^{m,q'})^*,\qquad H_0^{-m,q}:=(H^{m,q'}_0)^*,
    \end{align}
    where $\frac{1}{q}+\frac{1}{q'}=1$. We shall simply write $H^{-m}$ and $H^{-m}_0$, if $q=2$. Define the spaces
    \begin{align}
        \label{notation0-1}X^{1,q}:=\{f=(f_1,f_2,f_3)| f_1,f_2\in H^{1,q}_0, f_3\in H^{1,q}\}
    \end{align}
    and
    \begin{align}
        \label{notation1}X^{-1,q}:=(X^{1,q'})^*=\{f=(f_1,f_2,f_3)| f_1,f_2\in H^{-1,q}_0, f_3\in H^{-1,q}\}
    \end{align}
    endowed with the natural product norms, respectively.
    \item $H_{\curl}$ is the curl-admissible class defined in \eqref{INTRO1-1}. $\XX_m$ is the anisotropic Sobolev space given in \eqref{INTRO7}.
    \item By $e_j$, we denote the unit vector in the Euclidean space with the $j$-th component being $1$. For a three-dimensional vector $w$, we denote by $w_h$ its horizontal components, and by $w_3$ its vertical component. For a matrix $\MM$, $\row_j \MM$ denotes its $j$-th row.
    \item Define $\nabla_x:=(\partial_{x_1},\partial_{x_2})$ and $\Delta_x:=\partial_{x_1}^2+\partial_{x_2}^2$. For a vector field $f=(f_1,f_2,f_3)$, $f_h$ denotes its horizontal component $(f_1,f_2)$, and $\dive_hf_h:= \partial_{x_1}f_1+\partial_{x_2}f_2.$
    \item By $\I_\cdot$, we denote the usual indicator function.
    \item We use the notations $\lesssim$, $\lesssim_\alpha$, $\lesssim_\nu$, etc., to indicate inequalities that hold up to an inessential multiplicative constant, such as  $C$, $C_\alpha$, $C_\nu$, and so on. By $A\sim_\alpha B$, we denote the relation $B\lesssim_\alpha A\lesssim_\alpha B$.
    
    \item For any Banach spaces $X,Y$, we use the notation $\TT\in\LL(X,Y)$ to indicate that $\TT$ is a bounded linear operator from $X$ to $Y$. For any two operators $\TT_1,\TT_2$, we define the commutator $[\TT_1,\TT_2]:=\TT_1\TT_2-\TT_2\TT_1$.
    \item For stochastic processes $X(t),Y(t)$, we use the notation $\llbracket X,Y\rrbracket_t$ to denote their quadratic covariation. Define the quadratic variation of $X(t)$ by $\llbracket X\rrbracket_t:=\llbracket X,X\rrbracket_t$.
    \end{itemize}

    \section{Uniform-in-\texorpdfstring{$\theta$}{theta} estimate for the linear Stokes equation with transport noise}\label{sec2}
    In this section, we establish estimates for the linear Stokes equation \eqref{LS1} with transport noise, which are uniform with respect to the noise coefficients. These estimates will be used in Section \ref{SLofNS} to prove the tightness of the cut-off solutions under the high-frequency scaling. We consider
    \begin{align}
    \label{LS1}
    \begin{cases}
    \partial_tw-\nu\Delta w=\curl f+S_\theta w+\sqrt{2\kappa}\Pi (\dot{W}\cdot \nabla)w,\qquad u=\curl^{-1} w,\\(\partial_y w_3-\alpha w_3)|_{y=0}=(w_h-\alpha 
    u_h^\perp)|_{y=0}=0,\\(\partial_y w_3+\alpha w_3)|_{y=1}=(w_h+\alpha 
    u_h^\perp)|_{y=1}=0,\\
    w|_{t=0}=w_{\ini},
    \end{cases}
    \end{align} 
    where $W$ is given by either \eqref{INTRO4} or \eqref{INTRO6}, and $S_\theta$ is defined accordingly in \eqref{INTRO9}. The main difficulty is that the vorticity boundary conditions are nonhomogeneous and nonlocal through the Biot--Savart law. We remove this difficulty by introducing a boundary correction operator, which gives a uniform energy estimate. Then, by establishing a Meyers-type estimate, we upgrade this energy estimate to an estimate in $H^{s,p}_tH_x^{1-2s,q}$, with $p,q$ sufficiently close to $2$ and $s\in(0,\frac{1}{2})$.

    \subsection{Uniform energy estimate}
    The main result of this subsection is stated as follows.
    \begin{prop}\label{propLS1}Suppose that $w_{\ini}$ is $\FFF_0$-measurable, $f$ is progressively measurable, and 
    \[w_{\ini}\in L^2(\Omega;H_{\curl}),\qquad f\in L^2_{\loc}(\Omega\times (0,\infty);L^2).\]
    Then, there exists a unique global solution $w$ of \eqref{LS1}. Moreover, for any $T>0$, 
    \begin{align}
        \label{LS2}
        \E\sup_{t\le T}\|w(t)\|^2+\E\int_0^T\|\nabla w\|^2\ddd t\lesssim_{\alpha,\nu,\kappa,T}\E\left(\|w_{\ini}\|^2+\int_0^T\| f\|^2\ddd t\right).
    \end{align}
    \end{prop}
    Due to the inhomogeneous Dirichlet boundary conditions satisfied by the horizontal components $w_h$, the Laplacian is no longer self-adjoint, and hence direct energy estimates are not available. To overcome this difficulty, we introduce the boundary correction operator
    \begin{align}\label{LS3}
    \BB:\ &H_{\curl}\cap H^{s,q}\to H^{s,q},\qquad 
    w\mapsto 
    \begin{pmatrix}
        \alpha(1-2y)\big[(\curl^{-1}w)_h\big]^\perp\\
        0
    \end{pmatrix},
    \end{align}
    and the corrected vorticity 
     \[\tw:=(I-\BB)w.\]
    Basic properties of $\BB$ are gathered in Appendix \ref{BCO}.
    \begin{proof}[Proof of Proposition \ref{propLS1}]
    We prove the estimate in the non-degenerate case. The degenerate case follows from the same argument, with minor modifications indicated at the end of the proof.
    
    \noindent\textbf{Step I.} Applying $I-\BB$ to both sides of \eqref{LS1} and Proposition~\ref{propB2}, one has 
    \begin{align}
    \label{LS4}
    \begin{cases}
    \partial_t\tw-\nu\Delta \tw=\curl f+S_\theta w+h+\sqrt{2\kappa}\ssum_{k\neq 0}\ssum_{j=1,2}\theta_k(\Pi (\sigma_{k,j}\cdot\nabla)w +g_{k,j})\dot{B}_t^{k,j},\\(\partial_y \tw_3-\alpha \tw_3)|_{y=0}=(\partial_y \tw_3+\alpha\tw_3)|_{y=1}=\tw_h|_{y=0,1}=0,\\
    \tw|_{t=0}=\tw_{\ini},
    \end{cases}
    \end{align} 
    where 
    \[h:=-\nu[\BB,\partial_y^2]w-\BB S_\theta w-\BB\curl f,\qquad  g_{k,j}:=-\BB\Pi (\sigma_{k,j}\cdot\nabla)w.\]
    By Corollary \ref{coroA4} and Propositions \ref{propB1} and \ref{propB2}, it follows that 
    \begin{align}\label{LS4-2}
         \nu \|[\BB,\partial_y^2]w\|\lesssim_{\alpha,\nu}\|w\|,
         \qquad \|\BB S_\theta w\|\lesssim_{\alpha}\|S_\theta w\|_{H^{-1}}.
    \end{align}
    We estimate the It\^o--Stratonovich corrector $S_\theta w$ as follows. Let $\psi\in H^1$ be a test function. Since $\sigma_{k,j}$ is divergence-free and $\sigma_{k,j,3}|_{y=0,1}=0$, by integrating by parts and using \eqref{INTRO4-1} and Corollary \ref{coroA2}, one gets
    \begin{align*}
        \langle S_\theta w,\psi\rangle&= 2\kappa \ssum_{k\neq 0}\ssum_{j=1,2}\theta_k^2\langle
    \Pi(\sigma_{k,j}\cdot\nabla)\Pi(\sigma_{-k,j}\cdot \nabla)w,\psi\rangle\\&= -2\kappa \ssum_{k\neq 0}\ssum_{j=1,2}\theta_k^2\langle
    \Pi(\sigma_{-k,j}\cdot \nabla)w,(\sigma_{-k,j}\cdot\nabla)\Pi\psi\rangle\\&\lesssim_\kappa \|\nabla w\|\|\nabla \Pi\psi\|\lesssim_\kappa \|\nabla w\|\|\psi\|_{H^1},
    \end{align*}
    which gives 
    \begin{align}
        \label{LS5}
        \|\BB S_\theta w\|\lesssim_{\alpha}\|S_\theta w\|_{H^{-1}}\lesssim_{\alpha,\kappa} \|\nabla w\|.
    \end{align}
    Next, we turn to the estimate of $\BB\curl f$. By the definition \eqref{A13} of $\curl^{-1}$, one has
    \[\curl^{-1}\curl f= \PPP f-\left(\int_D(\PPP f)_1\ddd x\ddd y,\int_D(\PPP f)_2\ddd x\ddd y,0\right),\]
    where $\PPP$ is the usual Leray projection. This implies
    \begin{align}\label{LS4-1}
        \|\BB\curl f\|\lesssim_\alpha \|\curl^{-1}\curl f\|\lesssim_\alpha\|\PPP f\|\lesssim_\alpha\|f\|,
    \end{align}
    which, combined with \eqref{LS4-2} and \eqref{LS5}, yields
    \begin{align}
        \label{LS6}
        \|h\|\lesssim_{\alpha,\nu,\kappa}\|w\|_{H^1}+\|f\|.
    \end{align}
    Finally, we estimate $g_{k,j}$. Notice that
    \[\int_{D}(\sigma_{k,j}\cdot\nabla)w_3\ddd x\ddd y=0,\]
    by Corollaries \ref{coroA2} and \ref{coroA4}, one has 
    \begin{align}\label{LS7}
        \|g_{k,j}\|\lesssim_\alpha\|\Pi (\sigma_{k,j}\cdot\nabla)w\|_{H^{-1}}\lesssim_\alpha\|(\sigma_{k,j}\cdot\nabla)w\|_{H^{-1}}\lesssim_\alpha \|w\|.
    \end{align}

    \noindent\textbf{Step I\!I.} Throughout the energy estimates below, all $L^2$ pairings involving complex Fourier modes are understood as the real parts of the corresponding complex $L^2$ inner products. Applying the It\^o formula and using \eqref{INTRO4-2}, it follows that
    \begin{align}\label{LS8}
        \ddd \|\tw \|^2&=2\langle \tw, \nu\Delta\tw +\curl f+S_\theta w+h\rangle\ddd t\notag\\&\quad + 2\sqrt{2\kappa}\ssum_{k\neq 0}\ssum_{j=1,2}\theta_k\langle\tw, \Pi (\sigma_{k,j}\cdot\nabla)w +g_{k,j}\rangle \ddd B_t^{k,j}\notag\\&\quad+4\kappa\ssum_{k\neq 0}\ssum_{j=1,2}\theta_k^2\|\Pi (\sigma_{k,j}\cdot\nabla)w +g_{k,j}\|^2\ddd t.
    \end{align}
    The quadratic variation is estimated in the following way. Notice that
    \begin{align*}
        I_1:&=2\langle \tw,S_\theta w\rangle+4\kappa\ssum_{k\neq 0}\ssum_{j=1,2}\theta_k^2\|\Pi (\sigma_{k,j}\cdot\nabla)w +g_{k,j}\|^2\\
        &=2\langle w,S_\theta w\rangle-2\langle\BB w,S_\theta w\rangle+4\kappa\ssum_{k\neq 0}\ssum_{j=1,2}\theta_k^2\|\Pi (\sigma_{k,j}\cdot\nabla)w +g_{k,j}\|^2\\
        &=4\kappa\ssum_{k\neq 0}\ssum_{j=1,2}\theta_k^2\left(\|\Pi (\sigma_{k,j}\cdot\nabla)w +g_{k,j}\|^2-\|\Pi (\sigma_{k,j}\cdot\nabla)w\|^2\right)-2\langle\BB w,S_\theta w\rangle,
    \end{align*}
    where, by using \eqref{LS5}, \eqref{LS7}, and Corollary \ref{coroA4}, one has 
    \begin{align*}
        \|\Pi (\sigma_{k,j}\cdot\nabla)w +g_{k,j}\|^2-\|\Pi (\sigma_{k,j}\cdot\nabla)w\|^2&=2\langle \Pi (\sigma_{k,j}\cdot\nabla)w,g_{k,j}\rangle+\|g_{k,j}\|^2\\&\lesssim_\alpha \|w\|\|w\|_{H^1},
    \end{align*}
    and 
    \begin{align*}
        -2\langle\BB w,S_\theta w\rangle\lesssim \|\BB w\|_{H^1}\|S_\theta w\|_{H^{-1}}\lesssim_{\alpha,\kappa}\|w\|\|\nabla w\|.
    \end{align*}
    Therefore, it follows that
    \begin{align}
        \label{LS9}
        I_1\lesssim_{\alpha,\kappa} \|w\|\|w\|_{H^1}.
    \end{align}
    Now, by first integrating from $0$ to $T$ and taking expectations on both sides of \eqref{LS8}, and then using \eqref{LS6}, \eqref{LS9}, and Proposition \ref{propB1}, one gets 
    \begin{align*}
        \E\|\tw (T)\|^2-\E\|\tw _0\|^2&+2\nu\E\int_0^T\left(\|\nabla \tw\|^2+\alpha\|\tw_3\|^2_{L^2(\partial D)}\right)\ddd t
        \\&\lesssim_{\alpha,\nu,\kappa}\E \int_0^T\left(\|\tw\|_{H^1}\|f\|+\|\tw\|(\|w\|_{H^1}+\|f\|)+\|w\|\|w\|_{H^1}\right)\ddd t
         \\&\lesssim_{\alpha,\nu,\kappa} \E\int_0^T\left(\|\tw\|_{H^1}\|f\|+\|\tw\|\|\tw\|_{H^1}\right)\ddd t.
    \end{align*}
    Since the forcing terms of \eqref{LS1} belong to $H_{\curl}$, the horizontal average $\int_{\T^2}w_3\ddd x$ solves the homogeneous 1D heat equation with Robin boundary conditions and zero initial data:
    \[\int_{\T^2}w_{\ini,3}\ddd x=0,\qquad \mbox{ a.e. $y\in(0,1)$}.\]
    Therefore, one has
    \[\int_{\T^2}\tw_3(t,x,y)\ddd x=\int_{\T^2}w_3(t,x,y)\ddd x=0,\qquad \mbox{for any $t\ge 0$ and a.e. $y\in(0,1)$,}\]
    which, together with the homogeneous Dirichlet boundary conditions for the horizontal components, ensures the Poincar\'e inequality
    \begin{align}
        \label{LS10}\|\tw \|_{H^{1,q}}\lesssim \|\nabla \tw\|_q,\qquad q\in(1,\infty).
    \end{align}
    Combining this with the Gronwall inequality, one gets
    \begin{align}
        \label{LS11}\E\|\tw (T)\|^2+\nu\E\int_0^T\left(\|\nabla \tw\|^2+\alpha\|\tw_3\|^2_{L^2(\partial D)}\right)\ddd t\lesssim_{\alpha,\nu,\kappa,T} \E\|\tw _0\|^2+\E\int_0^T\|f\|^2\ddd t.
    \end{align} 

    \noindent\textbf{Step I\!I\!I.} From \eqref{INTRO4-1}, \eqref{LS6}--\eqref{LS9}, Proposition \ref{propB1}, and the Burkholder--Davis--Gundy inequality, it follows that 
    \begin{align*}
        \E\sup_{t\le T}\|\tw (t)\|^2-&\E\|\tw_{\ini}\|^2\lesssim_{\alpha,\nu,\kappa}\E\int_0^T\left(\|\tw\|_{H^1}\|f\|+\|\tw\|\|\tw\|_{H^1}\right)\ddd t\\&\qquad\qquad+\E\sup_{t\le T}\left|\int_0^t\ssum_{k\neq 0}\ssum_{j=1,2}\theta_k\langle\tw, \Pi (\sigma_{k,j}\cdot\nabla)w +g_{k,j}\rangle \ddd B_t^{k,j}\right|\\&\quad\lesssim _{\alpha,\nu,\kappa}\E\int_0^T\left(\|\tw\|_{H^1}\|f\|+\|\tw\|^2_{H^1}\right)\ddd t +\E\left(\int_0^T\|\tw\|^2\|\tw\|^2_{H^1}\ddd t\right)^{\frac{1}{2}},
    \end{align*}   
    where, by using the Cauchy--Schwarz inequality, the Poincar\'e inequality \eqref{LS10}, and \eqref{LS11}, one has 
    \begin{align*}
        \E\int_0^T\left(\|\tw\|_{H^1}\|f\|+\|\tw\|^2_{H^1}\right)\ddd t \lesssim_{\alpha,\nu,\kappa,T} \E\|\tw _0\|^2+\E\int_0^T\|f\|^2\ddd t,
    \end{align*}
    and 
    \begin{align*}
        \E\left(\int_0^T\|\tw\|^2\|\tw\|^2_{H^1}\ddd t\right)^{\frac{1}{2}}\le \epsilon\E\sup_{t\le T}\|\tw (t)\|^2+C_{\epsilon,\alpha,\nu,\kappa,T}\left( \E\|\tw _0\|^2+\E\int_0^T\|f\|^2\ddd t\right)
    \end{align*}
    with $\epsilon$ being arbitrarily small. Therefore, 
    \begin{align*}\E\sup_{t\le T}\|\tw (t)\|^2\lesssim_{\alpha,\nu,\kappa,T} \E\|\tw _0\|^2+\E\int_0^T\|f\|^2\ddd t,
    \end{align*} 
    which, together with \eqref{LS11} and Proposition \ref{propB1}, implies \eqref{LS2}.

    \noindent\textbf{Step I\!V.} The existence of global solutions to \eqref{LS1} follows from standard iteration arguments applied to the equivalent formulation \eqref{LS4}; we therefore omit the details. The degenerate case is treated in the same way. Indeed, the noise coefficients $\theta_\cdot$ satisfy \eqref{INTRO4-1}, and each solenoidal tangential Fourier mode $\sigma_k$, viewed as a three-dimensional vector field with zero vertical component, satisfies the divergence-free and no-flux conditions.
    \end{proof}

     \subsection{Improved regularity via a Meyers-type estimate}
    In this subsection, our goal is to improve the uniform energy estimate \eqref{LS2} to the fractional space-time regularity $H^{s,p}(0,T;H^{1-2s,q})$ with $p,q$ close to $2$ and $s\in[0,1/2)$.
    \begin{prop}\label{propLS2}
    Suppose that $w_{\ini}$ is $\FFF_0$-measurable, $f$ is progressively measurable, and 
    \[w_{\ini}\in L^p(\Omega;B^{1-2/p}_{q,p}\cap H_{\curl}),\qquad f\in L^p(\Omega\times (0,T);L^q).\]
    Then, there exists $p_0>2$ depending on $\alpha, \nu,\kappa$ such that for any $2\le q\le p< p_0$ and $s\in [0,1/2)$, the unique global solution $w$ of \eqref{LS1} satisfies
        \begin{align}\label{LS12}
            \|w\|_{L^p(\Omega;H^{s,p}(0,T;H^{1-2s,q}))}\lesssim_{s,p,q,\alpha,\nu,\kappa,T}\|f\|_{L^p(\Omega\times (0,T);L^q)}+\|w_{\ini}\|_{L^p(\Omega;B^{1-2/p}_{q,p})}.
        \end{align}
    \end{prop}
    As before, we focus on the non-degenerate case and indicate the minor modifications needed for the degenerate case at the end of this subsection. The key ingredient is the following Meyers-type estimate for the associated model problem
    \begin{equation}
    \label{LS13}
    \left\{
    \begin{aligned}
    &\begin{aligned}
    \partial_t\tw-\nu\Delta \tw+\lambda_0 \tw&= F+(I-\BB)S_\theta \tw  
    \\&\qquad\quad +\sqrt{2\kappa}\ssum_{k\neq 0}\ssum_{j=1,2}\big(\theta_k\Pi(\sigma_{k,j}\cdot\nabla)\tw +G_{k,j}\big)\dot{B}_t^{k,j},
    \end{aligned}
    \\&(\partial_y \tw_3-\alpha \tw_3)|_{y=0}=(\partial_y \tw_3+\alpha\tw_3)|_{y=1}=\tw_h|_{y=0,1}=0,
    \\&\tw|_{t=0}=\tw_{\ini},
    \end{aligned}
    \right.
    \end{equation}
    where $\lambda_0>0$ is a parameter to be chosen below, $S_\theta$ is defined in \eqref{INTRO9}$_1$, $\BB$ is defined in \eqref{LS3}, the coefficients $\theta_\cdot$ satisfy \eqref{INTRO4-1}, and
    \[F\in L^p(\Omega\times\R_+;X^{-1,q}),\qquad G:=\{G_{k,j}\}_{k,j}\in L^p(\Omega\times\R_+;L^q(\ell^2)),\]
    with $X^{-1,q}$ given in \eqref{notation1}.
    \begin{lemma} \label{lemmaLS3}Let $\tw$ be a solution of \eqref{LS13}. Then, there exist $p_0>2$ and $\lambda_0>0$ depending on $\alpha, \nu,\kappa$ such that for any $2\le q\le p< p_0$, 
    \begin{align}\label{LS14}
            \|\nabla\tw\|_{L^p(\Omega\times\R_+;L^{q})}&\lesssim_{p,q,\alpha,\nu,\kappa}\|\tw_{\ini}\|_{L^p(\Omega;B^{1-2/p}_{q,p})}\notag\\&\qquad\qquad+\|F\|_{L^p(\Omega\times\R_+;X^{-1,q})}+\|G\|_{L^p(\Omega\times\R_+;L^q(\ell^2))}.
        \end{align}
    \end{lemma}
    To establish the above result, we need the following Sneiberg theorem; see \cite{DtER17} and references therein for a proof.
    \begin{theorem}[Sneiberg]\label{theoremLS4}
        Let $(X_0,X_1)$ and $(Y_0,Y_1)$ be interpolation couples of Banach spaces. Assume $\TT\in\LL(X_0,Y_0)\cap \LL(X_1,Y_1)$ and put
        \[c_1:=\max\{\|\TT\|_{\LL(X_0,Y_0)},\|\TT\|_{\LL(X_1,Y_1)}\}.\]
        Suppose that for a given $\gamma_0\in(0,1)$, $\TT$ is an isomorphism from $[X_0,X_1]_{\gamma_0}$ to $[Y_0,Y_1]_{\gamma_0}$. Let $c_2$ satisfy
        \[c_2\ge \|\TT^{-1}\|_{\LL([Y_0,Y_1]_{\gamma_0},[X_0,X_1]_{\gamma_0})}.\]
        Then, if  
        \[\gamma\in(0,1),\qquad |\gamma-\gamma_0|\le \frac{\min\{\gamma_0,1-\gamma_0\}}{6(1+2c_1c_2)},\]
        $\TT$ remains an isomorphism from $[X_0,X_1]_{\gamma}$ to $[Y_0,Y_1]_{\gamma}$, and
        \[\|\TT^{-1}\|_{\LL([Y_0,Y_1]_{\gamma},[X_0,X_1]_{\gamma})}\le 8c_2.\]
    \end{theorem}
    \begin{proof}[Proof of Lemma \ref{lemmaLS3}]
        \textbf{Step I.} In this step, we assume that both the random forcing and the initial data vanish, namely, $\tw$ satisfies 
        \begin{align}
            \label{LS15}
            \partial_t\tw-\nu\Delta \tw-(I-\BB)S_\theta \tw+\lambda_0 \tw&= F
        \end{align}
        with zero initial data and boundary conditions \eqref{LS13}$_2$. Define the operator $T:\XXX_{p,q}\to \YYY_{p,q}$ by
        \[\TT\tw:=\partial_t\tw-\nu\Delta \tw-(I-\BB)S_\theta \tw+\lambda_0 \tw,\]
        and the spaces
        \[\XXX_{p,q}:=\{\tw\in L^p(\R_+;X^{1,q})|\partial_t\tw\in L^p(\R_+;X^{-1,q}),\ \tw(0)=0\},\]
        \[\YYY_{p,q}:=L^p(\R_+;X^{-1,q}),\]
        where $X^{1,q}$ is given in \eqref{notation0-1}. We claim that for any $p,q\in (1,\infty)$, $\TT\in \LL(\XXX_{p,q},\YYY_{p,q})$. Indeed, for any test function $\psi\in X^{1,q'}$, by integrating by parts, it follows that 
        \begin{align*}
            \langle\Delta \tw,\psi\rangle\le \alpha|\langle \tw_3,\psi_3\rangle_{L^2(\partial D)}|+\|\nabla \tw\|_q\|\nabla\psi\|_{q'}\lesssim_\alpha \|\tw\|_{X^{1,q}}\|\psi\|_{X^{1,q'}}.
        \end{align*}
        Similarly, by integrating by parts again and applying Corollary \ref{coroA2}, for any $\varphi\in H^{1,q'}$, one has
        \begin{align*}
                \langle S_\theta \tw,\varphi\rangle\lesssim_{q,\kappa}\ssum_{k\neq 0}\ssum_{j=1,2}\theta_k^2\|(\sigma_{-k,j}\cdot\nabla)\tw\|_q\|(\sigma_{-k,j}\cdot\nabla)\Pi\varphi\|_{q'}\lesssim_{q,\kappa}\|\nabla\tw\|_{q}\|\varphi\|_{H^{1,q'}}.
        \end{align*}
        These estimates, together with Corollary \ref{coroA4}, imply 
        \begin{align}
            \label{LS16}
            \|\Delta \tw\|_{X^{-1,q}}\lesssim_\alpha \|\tw\|_{X^{1,q}},\quad \|(I-\BB)S_\theta \tw\|_{X^{-1,q}}\lesssim_{\alpha,q} \|S_\theta \tw\|_{H^{-1,q}}\lesssim_{q,\alpha,\kappa} \|\nabla \tw\|_{q},
        \end{align}
        and thus $\TT\in \LL(\XXX_{p,q},\YYY_{p,q})$ as claimed.

        Next, we show that $\TT$ is an isomorphism when $p=q=2$ and $\lambda_0$ is sufficiently large. Indeed, for any $F\in \YYY_{2,2}$, if $\tw\in \XXX_{2,2}$ satisfies $\TT\tw=F$, then a direct energy estimate gives
        \begin{align*}
           \frac{1}{2}\frac{\ddd}{\ddd t}\|\tw\|^2+\nu&(\|\nabla\tw\|^2+\alpha\|\tw_3\|^2_{L^2(\partial D)})+\lambda_0\|\tw\|^2\\&\qquad\qquad\quad+2\kappa\ssum_{k\neq 0}\ssum_{j=1,2}\theta_k^2\|\Pi(\sigma_{k,j}\cdot \nabla )\tw\|^2=\langle F-\BB S_\theta\tw,\tw\rangle,
        \end{align*}
        where, by using Corollary \ref{coroA4}, \eqref{LS5}, and the Poincar\'e-type inequality 
        \[\|\tw\|\lesssim\|\nabla \tw\|+\|w_3\|_{L^2(\partial D)},\]
        one has
        \begin{align*}
           \langle F-\BB S_\theta\tw,\tw\rangle\lesssim_{\alpha,\kappa} \|F\|_{X^{-1,2}}(\|\nabla\tw\|^2+\alpha\|\tw_3\|^2_{L^2(\partial D)})^{\frac{1}{2}}+\|\nabla \tw\|\|\tw\|.
        \end{align*} 
        This yields
        \begin{align*}
            \frac{1}{2}\frac{\ddd}{\ddd t}\|\tw\|^2+\frac{\nu}{2}&(\|\nabla\tw\|^2+\alpha\|\tw_3\|^2_{L^2(\partial D)})+(\lambda_0-C_{\alpha,\nu,\kappa})\|\tw\|^2\\&\qquad\qquad\quad+2\kappa\ssum_{k\neq 0}\ssum_{j=1,2}\theta_k^2\|\Pi(\sigma_{k,j}\cdot \nabla )\tw\|^2\lesssim_{\alpha,\nu,\kappa}\|F\|_{X^{-1,2}}^2.
        \end{align*}
        Therefore, by choosing $\lambda_0>C_{\alpha,\nu,\kappa}$, it follows that
        \[\int_0^\infty\|\tw\|^2_{X^{1,2}}\ddd t\lesssim_{\alpha,\nu,\kappa}\|F\|^2_{\YYY_{2,2}},\]
        which, together with \eqref{LS16}, implies 
        \[\|\partial_t\tw\|_{L^2(\R_+;X^{-1,2})}=\|\nu\Delta \tw+(I-\BB)S_\theta \tw-\lambda_0\tw+F\|_{L^2(\R_+;X^{-1,2})}\lesssim_{\alpha,\nu,\kappa}\|F\|_{\YYY_{2,2}}.\]
        The above a priori estimates ensure the unique solvability of the equation $\TT\tw=F$. Therefore, $\TT$ is an isomorphism.

        Finally, we prove \eqref{LS14} for this special case by applying Theorem \ref{theoremLS4}. Notice that the families $\{\XXX_{p,q}\}_{p,q}$ and $\{\YYY_{p,q}\}_{p,q}$ form the required interpolation scales. Moreover, for any $p\in(1,\infty)$, $\TT\in \LL(\XXX_{p,2},\YYY_{p,2})$, and $\TT$ is an isomorphism when $p=2$, with relevant constants depending only on $\alpha,\nu,\kappa$. Then, there exists $p_1=p_1(\alpha,\nu,\kappa)>2$ such that for any $2\le p<p_1$, $\TT$ remains an isomorphism from $\XXX_{p,2}$ to $\YYY_{p,2}$. By repeating this argument on the index $q$ with fixed $p\in[2,p_1)$, we conclude that there is $p_0=p_0(\alpha,\nu,\kappa)>2$ such that for any $2\le q\le p<p_0$, $\TT$ is an isomorphism from $\XXX_{p,q}$ to $\YYY_{p,q}$. In particular, 
        \begin{align}
            \label{LS17}
            \|\tw\|_{L^p(\R_+;X^{1,q})}\lesssim_{p,q,\alpha,\nu,\kappa}\|F\|_{L^p(\R_+;X^{-1,q})}.
        \end{align}

        \noindent\textbf{Step I\!I.} In this step, we consider the case in which the multiplicative stochastic transport term is absent. More precisely, we assume that $\tw$ satisfies
        \begin{align}
            \label{LS18}\partial_t\tw-\nu\Delta \tw-(I-\BB)S_\theta \tw+\lambda_0 \tw=F+\sqrt{2\kappa}\ssum_{k\neq 0}\ssum_{j=1,2}G_{k,j}\dot{B}_t^{k,j}
        \end{align}
        with the initial data $\tw_{\ini}$ and boundary conditions \eqref{LS13}$_2$. Let $z$ be the solution of 
        \begin{align*}
            \begin{cases}
                \partial_t z-\nu\Delta z+\lambda_0z=\sqrt{2\kappa}\ssum_{k\neq 0}\ssum_{j=1,2}G_{k,j}\dot{B}_t^{k,j}\\
                (\partial_y z_3-\alpha z_3)|_{y=0}=(\partial_y z_3+\alpha z_3)|_{y=1}=z_h|_{y=0,1}=0,\\z|_{t=0}=\tw_{\ini}.
             \end{cases}
        \end{align*}
        It is known that the operator $-\Delta+\lambda_0$, endowed with the Dirichlet boundary conditions on the horizontal components and the Robin boundary conditions on the vertical component, admits a bounded $H^\infty$-calculus on $L^q$ with angle less than $\pi/2$, see Section 2 of \cite{VNVW12}. Therefore, Theorem 1.1 therein, together with the standard heat semigroup estimate for the initial-value part, yields
        \begin{align}
            \label{LS19}\|\nabla z\|_{L^p(\Omega\times\R_+;L^{q})}\lesssim_{p,q,\alpha,\nu,\kappa}\|G\|_{L^p(\Omega\times\R_+;L^q(\ell^2))}+\|\tw_{\ini}\|_{L^p(\Omega;B^{1-2/p}_{q,p})}.
        \end{align}
        Let $\bw:=\tw-z$. Notice that $\bw$ satisfies \eqref{LS15} with $F$ replaced by $F+(I-\BB)S_\theta z$. Then, an application of \eqref{LS16} and \eqref{LS17} gives 
        \begin{align*}
            \|\bw\|_{L^p(\Omega\times\R_+;X^{1,q})}&\lesssim_{p,q,\alpha,\nu,\kappa}\|F+(I-\BB)S_\theta z\|_{L^p(\Omega\times\R_+;X^{-1,q})}\\&\lesssim_{p,q,\alpha,\nu,\kappa}\|F\|_{L^p(\Omega\times\R_+;X^{-1,q})}+\|\nabla z\|_{L^p(\Omega\times\R_+;L^{q})},
        \end{align*}
        which, together with \eqref{LS19}, yields \eqref{LS14} in this case.

        \noindent\textbf{Step I\!I\!I.} We finally consider the general case, in which the stochastic transport term, the additive random forcing, and the initial data are all present. Introduce the following equivalent norm of $X^{1,q}$
        \[\|\tw\|_{ \tX^{1,q}_{\alpha,\nu}}:=\left(\nu\|\nabla\tw\|^2_{q}+\nu\alpha\|\tw_3\|^2_{L^q(\partial D)}\right)^{\frac{1}{2}}.\]
        Define $\|\cdot\|_{ \tX^{-1,q}_{\alpha,\nu}}$ as the dual norm of $\tX^{1,q}_{\alpha,\nu}$ and
        \[B_{p,q,\theta,\kappa}(\tw):=\sqrt{2\kappa}\left\|\left(\ssum_{k\neq 0}\ssum_{j=1,2}\theta_k^2|\Pi(\sigma_{k,j}\cdot\nabla )\tw|^2\right)^{\frac{1}{2}}\right\|_{L^p(\Omega\times \R_+;L^q)}.\]
        Since $q\ge2$, by the Minkowski inequality, it follows that 
        \begin{align}\label{LS20}
            B_{p,q,\theta,\kappa}(\tw)&\le \sqrt{2\kappa} \left\|\left(\ssum_{k\neq 0}\ssum_{j=1,2}\theta_k^2\|\Pi(\sigma_{k,j}\cdot\nabla)\tw\|^2_q\right)^{\frac{1}{2}}\right\|_{L^p(\Omega\times \R_+)}\notag\\&\le 4\sqrt{\kappa} \|\Pi\|_{\LL(L^q)}\|\nabla \tw\|_{L^p(\Omega\times\R_+;L^{q})},
        \end{align}
        which, together with the estimate obtained in Step I\!I for equation \eqref{LS18}, ensures the existence of an optimal constant $C_{p,q}$ depending only on $p,q,\alpha,\nu,\kappa$ such that  
        \begin{align}\label{LS21}
            \frac{\nu}{4}&\|\nabla\tw\|^2_{L^p(\Omega\times \R_+;L^q)}+B_{p,q,\theta,\kappa}^2(\tw)\notag\\&\ \le C^2_{p,q}\left(\frac{1}{2}\|F\|^2_{L^p(\Omega\times\R_+;\tX^{-1,q}_{\alpha,\nu})}+2\kappa\|G\|_{L^p(\Omega\times\R_+;L^q(\ell^2))}^2+\frac{1}{2}\|\tw_{\ini}\|_{L^p(\Omega;B^{1-2/p}_{q,p})}^2\right)
        \end{align}
        for all solutions $\tw$ of \eqref{LS18}. In particular, for $p=q=2$, by the It\^o formula, it follows that 
        \begin{align*}
            \E\int_0^\infty&\left(\nu\|\nabla\tw\|^2+\nu\alpha\|\tw_3\|^2_{L^2(\partial D)}\right)\ddd t+B_{2,2,\theta,\kappa}^2(\tw)+\lambda_0\E\int_0^\infty\|\tw\|^2 \ddd t\\&\qquad\qquad\le \E\int_0^\infty |\langle F-\BB S_\theta\tw,\tw\rangle|\ddd t+2\kappa\|G\|_{L^2(\Omega\times\R_+;L^2(\ell^2))}^2+\frac{1}{2}\|\tw_{\ini}\|_{L^2(\Omega;L^2)}^2,
        \end{align*}
        where, by using \eqref{LS5} and Corollary \ref{coroA4}, one has
        \[|\langle F-\BB S_\theta\tw,\tw\rangle|\le \frac{1}{2}\|F\|^2_{\tX^{-1,2}_{\alpha,\nu}}+\frac{3}{4}\|\tw\|^2_{\tX^{1,2}_{\alpha,\nu}}+C_{\alpha,\nu,\kappa}\|\tw\|^2.\]
        Therefore, by choosing $\lambda_0>C_{\alpha,\nu,\kappa}$, one gets $C_{2,2}\le 1$. Combining this with the same interpolation argument for the optimal constants as in Appendix~A of \cite{Agr25}, one has
        \[\limsup_{p\downarrow 2} C_p\le C_{2,2}\le 1,
        \qquad C_p:=\sup_{2\le q\le p}C_{p,q}.\]
        
        Now, let $\tw$ be a solution of \eqref{LS13}. Notice that $\tw$ solves \eqref{LS18} with $G_{k,j}$ replaced by $\theta_k\Pi (\sigma_{k,j}\cdot\nabla)\tw +G_{k,j}$. By applying \eqref{LS21}, it follows
        \begin{align*}
             \frac{\nu}{4}\|\nabla\tw&\|^2_{L^p(\Omega\times \R_+;L^q)}+B_{p,q,\theta,\kappa}^2(\tw)\\&\le C^2_{p,q}\left(\frac{1}{2}\|F\|^2_{L^p(\Omega\times\R_+;\tX^{-1,q}_{\alpha,\nu})}+\frac{1}{2}\|\tw_{\ini}\|_{L^p(\Omega;B^{1-2/p}_{q,p})}^2\right)\\
             &\quad +C^2_{p,q}\left(\sqrt{2\kappa}\|G\|_{L^p(\Omega\times\R_+;L^q(\ell^2))}+\sqrt{2\kappa}\|\{\theta_k\Pi (\sigma_{k,j}\cdot\nabla)\tw\}_{k,j}\|_{L^p(\Omega\times\R_+;L^q(\ell^2))}\right)^2 \\&\le C^2_{p,q}\left(\frac{1}{2}\|F\|^2_{L^p(\Omega\times\R_+;\tX^{-1,q}_{\alpha,\nu})}+\frac{1}{2}\|\tw_{\ini}\|_{L^p(\Omega;B^{1-2/p}_{q,p})}^2\right)\\
             &\quad +C^2_{p,q}(1+\delta^{-1})2\kappa\|G\|_{L^p(\Omega\times\R_+;L^q(\ell^2))}^2+C^2_{p,q}(1+\delta)B_{p,q,\theta,\kappa}^2(\tw),
        \end{align*}
        which, together with \eqref{LS20}, yields
        \begin{align*}
            \frac{\nu}{4}\|\nabla\tw\|^2_{L^p(\Omega\times \R_+;L^q)}&\le \frac{C^2_{p,q}}{2}\left(\|F\|^2_{L^p(\Omega\times\R_+;\tX^{-1,q}_{\alpha,\nu})}+\|\tw_{\ini}\|_{L^p(\Omega;B^{1-2/p}_{q,p})}^2\right)\\&\quad+2\kappa C^2_{p,q}(1+\delta^{-1})\|G\|_{L^p(\Omega\times\R_+;L^q(\ell^2))}^2\\&\quad+16(C^2_{p}(1+\delta)-1)\kappa\|\Pi\|_{\LL(L^q)}^2\|\nabla\tw\|^2_{L^p(\Omega\times \R_+;L^q)}.
        \end{align*}
        Since $\|\Pi\|_{\LL(L^2)}=1$ and by interpolation, $\|\Pi\|_{\LL(L^q)}\to 1$ as $q\to2$, one has
        \[\limsup_{p\to2} 16(C^2_{p}(1+\delta)-1)\kappa\|\Pi\|_{\LL(L^q)}^2\le 16\delta\kappa. \]
        By choosing $p_0$ sufficiently close to $2$ such that for any $2\le q\le p<p_0$, 
        \[16(C^2_{p}(1+\delta)-1)\kappa\|\Pi\|_{\LL(L^q)}^2\le 20\delta\kappa,\]
        and $\delta$ sufficiently small so that $\delta<\frac{\nu}{160\kappa}$, we obtain \eqref{LS14}.
    \end{proof}
    The following result follows immediately by applying Lemma \ref{lemmaLS3} to $\widehat w(t):=e^{-\lambda_0 t}\tw(t)$.
    \begin{coro}
         \label{coroLS5}Let $\tw$ be a solution of \eqref{LS13} with $\lambda_0=0$. Then, there is $p_0>2$ depending on $\alpha, \nu,\kappa$ such that for any $2\le q\le p< p_0$,   
    \begin{align}\label{LS22}
            \|\nabla\tw\|_{L^p(\Omega\times(0,T);L^{q})}&\lesssim_{p,q,\alpha,\nu,\kappa,T}\|\tw_{\ini}\|_{L^p(\Omega;B^{1-2/p}_{q,p})}\notag\\&\qquad+\|F\|_{L^p(\Omega\times(0,T);X^{-1,q})}+\|G\|_{L^p(\Omega\times(0,T);L^q(\ell^2))}.
        \end{align}
    \end{coro}
    Now, we are in a position to prove Proposition \ref{propLS2}.
    \begin{proof}[Proof of Proposition \ref{propLS2}]\textbf{Step I.} Let $w$ be a solution of \eqref{LS1} and set $\tw:=(I-\BB)w$. Notice that
    \begin{align*}
    \partial_t\tw-\nu\Delta \tw&=F+(I-\BB)S_\theta \tw+\sqrt{2\kappa}\ssum_{k\neq 0}\ssum_{j=1,2}(\theta_k\Pi (\sigma_{k,j}\cdot\nabla)\tw +G_{k,j})\dot{B}_t^{k,j},
    \end{align*} 
    where
    \[G_{k,j}:=\theta_k\big(\Pi(\sigma_{k,j}\cdot\nabla)\BB w-\BB\Pi (\sigma_{k,j}\cdot\nabla)w\big),\]
    and 
    \[F:=-\nu[\BB,\partial_y^2]w+(I-\BB) S_\theta\BB w+(I-\BB)\curl f.\]
    Therefore, by Corollary \ref{coroLS5}, there exists $p_0>2$ depending on $\alpha, \nu,\kappa$ such that for any $2\le q\le p< p_0$, it follows
     \begin{align}
            \label{LS23}\|\nabla\tw\|_{L^p(\Omega\times(0,T);L^{q})}&\lesssim_{p,q,\alpha,\nu,\kappa,T}\|\tw_{\ini}\|_{L^p(\Omega;B^{1-2/p}_{q,p})}\notag\\&\qquad+\|F\|_{L^p(\Omega\times(0,T);X^{-1,q})}+\|\{G_{k,j}\}_{k,j}\|_{L^p(\Omega\times(0,T);L^q(\ell^2))}.
        \end{align}
    Shrinking $p_0$ if necessary and applying the embedding $L^2\hookrightarrow H^{-1,q}$, Corollary~\ref{coroA4}, Proposition \ref{propB2}, and \eqref{LS16}, one has
    \begin{align}\label{LS25}
        \nu\|[\BB,\partial_y^2]w\|_{X^{-1,q}}\lesssim_\nu \|[\BB,\partial_y^2]w\|\lesssim_{\alpha,\nu} \|w\|\lesssim_{\alpha,\nu} \|w\|_q,
    \end{align}
    and 
    \begin{align}
        \label{LS26}
        \|(I-\BB) S_\theta\BB w\|_{X^{-1,q}}\lesssim_\alpha\|S_\theta\BB w\|_{H^{-1,q}}\lesssim_{q,\alpha,\kappa}\|\nabla \BB w\|_q\lesssim_{q,\alpha,\kappa}\| w\|_q.
    \end{align}
    Moreover, by repeating the derivation of \eqref{LS4-1}, one gets 
    \begin{align}\label{LS26-1}
        \|(I-\BB) \curl f\|_{X^{-1,q}}\le \|f\|_q+\|\BB\curl f\|_q\lesssim_{q,\alpha}\|f\|_q,
    \end{align}
    which, combined with \eqref{LS25}, \eqref{LS26} and Proposition \ref{propB1}, yields
    \begin{align}
        \label{LS27}
        \|F\|_{L^p(\Omega\times(0,T);X^{-1,q})}\lesssim_{q,\alpha,\nu,\kappa}\|\tw\|_{L^p(\Omega\times(0,T);L^q)}+\|f\|_{L^p(\Omega\times(0,T);L^q)}.
    \end{align}
    We turn to the estimate of $G:=\{G_{k,j}\}_{k,j}$. Indeed, as $q\ge 2$, by the Minkowski inequality, it follows that
    \begin{align*}
        \|G\|_{L^q(\ell^2)}\le \|G\|_{\ell^2(L^q)}=\left(\ssum_{k\neq 0}\ssum_{j=1,2}\theta_k^2\|\Pi(\sigma_{k,j}\cdot\nabla)\BB w-\BB\Pi (\sigma_{k,j}\cdot\nabla)w\|^2_{q}\right)^{\frac{1}{2}},
    \end{align*}
    where, by using Corollaries \ref{coroA2} and \ref{coroA4} combined with the fact $\int_D (\sigma_{k,j}\cdot\nabla)w_3\ddd x\ddd y=0$, one has
    \begin{align*}
        \|\Pi(\sigma_{k,j}\cdot\nabla)\BB w-\BB\Pi (\sigma_{k,j}\cdot\nabla)w\|_{q}&\lesssim_{q,\alpha}\|\nabla\BB w\|_q+\|\curl^{-1}\Pi(\sigma_{k,j}\cdot\nabla)w\|_q\\&\lesssim_{q,\alpha}\|w\|_q+\|(\sigma_{k,j}\cdot\nabla)w\|_{H^{-1,q}}\lesssim_{q,\alpha}\|w\|_q. 
    \end{align*}
    This, together with Proposition \ref{propB1}, yields
    \begin{align*}
        \|G\|_{L^p(\Omega\times(0,T);L^q(\ell^2))}\lesssim_{q,\alpha}\|\tw\|_{L^p(\Omega\times(0,T);L^q)}.
    \end{align*}
    Plugging the above estimate and \eqref{LS27} into \eqref{LS23}, one gets
    \begin{align*}\|\nabla\tw\|_{L^p(\Omega\times(0,T);L^{q})}&\lesssim_{p,q,\alpha,\nu,\kappa,T}\|\tw_{\ini}\|_{L^p(\Omega;B^{1-2/p}_{q,p})}\notag\\&\qquad\qquad\qquad+\|\tw\|_{L^p(\Omega\times(0,T);L^q)}+\|f\|_{L^p(\Omega\times(0,T);L^q)},
    \end{align*}    
    which, combined with the Gagliardo--Nirenberg interpolation inequality
    \begin{align}\label{LS27-1}
        \|\tw\|_{L^p(\Omega\times (0,T);L^{q})}\lesssim\|\tw\|^{1-\delta}_{L^p(\Omega\times (0,T);L^{2})}\|\nabla\tw\|^{\delta}_{L^p(\Omega\times (0,T);L^{q})},\qquad \delta=\frac{3(q-2)}{5q-6},
    \end{align}
    yields
    \begin{align}
        \label{LS28}
        \|\nabla\tw\|_{L^p(\Omega\times(0,T);L^{q})}&\lesssim_{p,q,\alpha,\nu,\kappa,T}\|\tw_{\ini}\|_{L^p(\Omega;B^{1-2/p}_{q,p})}\notag\\&\qquad\qquad\qquad+\|\tw\|_{L^p(\Omega\times(0,T);L^2)}+\|f\|_{L^p(\Omega\times(0,T);L^q)}.
    \end{align}
    
    \noindent\textbf{Step I\!I.} We estimate $\|\tw\|_{L^p(\Omega\times(0,T);L^2)}$ through the energy method. By the It\^o formula and \eqref{LS8}, it follows that 
    \begin{align}\label{LS29}
        \ddd \|\tw(t)\|^p&=\ddd M(t)+p \|\tw\|^{p-2}\langle \tw, \nu\Delta\tw +\curl f+S_\theta w+h\rangle\ddd t\notag\\&\ +2\kappa p\|\tw\|^{p-2}\ssum_{k\neq 0}\ssum_{j=1,2}\theta_k^2\|\Pi (\sigma_{k,j}\cdot\nabla)w +g_{k,j}\|^2\ddd t\notag\\&\ +4\kappa p(p-2)\|\tw\|^{p-4} \ssum_{k\neq 0}\ssum_{j=1,2}\theta_k^2|\langle\tw, \Pi (\sigma_{k,j}\cdot\nabla)w +g_{k,j}\rangle|^2 \ddd t\notag\\&=:\ddd M(t)+ I_2\ddd t+I_3\ddd t+I_4\ddd t,
    \end{align}
    where 
    \[\ddd M(t)=\sqrt{2\kappa} p\|\tw\|^{p-2}\ssum_{k\neq 0}\ssum_{j=1,2}\theta_k\langle\tw, \Pi (\sigma_{k,j}\cdot\nabla)w +g_{k,j}\rangle \ddd B_t^{k,j}.\]
    Applying \eqref{LS6}, \eqref{LS7}, \eqref{LS9}, the Poincar\'e inequality \eqref{LS10}, and Proposition \ref{propB1}, one has 
    \begin{align*}
        I_2+I_3&=p \|\tw\|^{p-2}\langle \tw, \nu\Delta\tw +\curl f+h\rangle\\&\quad+p \|\tw\|^{p-2}\left( \langle \tw,S_\theta w\rangle+2\kappa\ssum_{k\neq 0}\ssum_{j=1,2}\theta_k^2\|\Pi (\sigma_{k,j}\cdot\nabla)w +g_{k,j}\|^2\right)\\&\le -\nu p  \|\tw\|^{p-2}\left(\|\nabla\tw\|^2+\alpha\|\tw_3\|_{L^2(\partial D)}^2\right)+p \|\tw\|^{p-2}\|\tw \|_{H^1}\|f\|\\&\quad+p C_{\alpha,\nu,\kappa}\|\tw\|^{p-2}\|\tw\|(\|w\|_{H^1}+\|f\|)\\&\le -\frac{\nu}{2}p \|\tw\|^{p-2}\left(\|\nabla\tw\|^2+\alpha\|\tw_3\|_{L^2(\partial D)}^2\right)+p C_{\alpha,\nu,\kappa}\|f\|^p+p C_{\alpha,\nu,\kappa}\|\tw\|^{p}
    \end{align*}
    and 
    \begin{align*}
        I_4&\lesssim_\kappa p(p-2)\|\tw\|^{p-4} \ssum_{k\neq 0}\ssum_{j=1,2}\theta_k^2\|\tw\|^2\|\Pi (\sigma_{k,j}\cdot\nabla)w +g_{k,j}\|^2\\&\lesssim_{\alpha,\kappa} p(p-2)\|\tw\|^{p-2} \|\tw\|^2_{H^1}\\&\lesssim_{\alpha,\kappa} (p-2)\left(p\|\tw\|^{p-2} \left(\|\nabla\tw\|^2+\alpha\|\tw_3\|_{L^2(\partial D)}^2\right)\right).
    \end{align*}
    By choosing $p_0=p_0(\alpha,\nu,\kappa)$ sufficiently close to $2$ so that 
    \[(p-2)C_{\alpha,\kappa}\le \frac{\nu}{4},\]
    one obtains
    \begin{align*}
        I_2+I_3+I_4\le-\frac{\nu}{4}p \|\tw\|^{p-2}\left(\|\nabla\tw\|^2+\alpha\|\tw_3\|_{L^2(\partial D)}^2\right)+p C_{\alpha,\nu,\kappa}\|f\|^p+p C_{\alpha,\nu,\kappa}\|\tw\|^{p}.
    \end{align*}
    This, together with first integrating from $0$ to $t$ and then taking expectations on both sides of \eqref{LS29}, ensures
    \begin{align*}
        \E\|\tw(t)\|^p-\E\|\tw_{\ini}\|^p\le_{p,\alpha,\nu,\kappa} \E\int_0^t\|f\|^p\ddd s+\int_0^t\E\|\tw\|^p\ddd s.
    \end{align*}
    An application of the Gronwall inequality yields
    \begin{align*}
        \E\int_0^T\|\tw\|^p\ddd t\lesssim_{p,\alpha,\nu,\kappa,T}\E\|\tw_{\ini}\|^p+\|f\|^p_{L^p(\Omega\times(0,T);L^2)}.
    \end{align*}
    Plugging the above estimate into \eqref{LS28} and using $q\ge2$, the embedding $B^{1-2/p}_{q,p}\hookrightarrow L^2$, and the Poincar\'e inequality \eqref{LS10}, one gets
    \begin{align}
        \label{LS30}\|\tw\|_{L^p(\Omega\times(0,T);H^{1,q})}&\lesssim_{p,q,\alpha,\nu,\kappa,T}\|f\|_{L^p(\Omega\times(0,T);L^q)}+\|\tw_{\ini}\|_{L^p(\Omega;B^{1-2/p}_{q,p})}.
    \end{align}

    \noindent\textbf{Step I\!I\!I.} To conclude the proof, we apply stochastic maximal regularity of the heat semigroup to equation \eqref{LS4}, see Theorem 1.2 in \cite{VNVW12}, to obtain
    \begin{align}\label{LS31}
        \|\tw\|_{L^p(\Omega;H^{s,p}(0,T;H^{1-2s,q}))}&\lesssim_{s,p,q,\alpha,\nu,\kappa,T}\|\tw_{\ini}\|_{L^p(\Omega;B^{1-2/p}_{q,p})}\notag\\&\ +\|\nu[\BB,\partial_y^2]w+(I-\BB) S_\theta w+(I-\BB)\curl f\|_{L^p(\Omega\times(0,T);X^{-1,q})}\notag\\&\ +\|\{\theta_k(I-\BB)\Pi(\sigma_{k,j}\cdot\nabla)w\}_{k,j}\|_{L^p(\Omega\times(0,T);L^q(\ell^2))}\notag\\&=:\|\tw_{\ini}\|_{L^p(\Omega;B^{1-2/p}_{q,p})}+I_5+I_6.
    \end{align}
    By using \eqref{LS16}, \eqref{LS25}, \eqref{LS26-1}, Corollary \ref{coroA4}, and Proposition \ref{propB1}, it follows that
    \begin{align*}
        \|\nu[\BB,\partial_y^2]w-(I-\BB) S_\theta w-(I-\BB)\curl f\|_{X^{-1,q}}\lesssim_{q,\alpha,\nu,\kappa}\|\tw\|_{H^{1,q}}+\|f\|_q,
    \end{align*}
    which, together with \eqref{LS30}, leads to
    \begin{align}
        \label{LS32}
        I_5\lesssim_{p,q,\alpha,\nu,\kappa,T}\|f\|_{L^p(\Omega\times(0,T);L^q)}+\|\tw_{\ini}\|_{L^p(\Omega;B^{1-2/p}_{q,p})}.
    \end{align}
    We turn to the estimate of $I_6$. By the Minkowski inequality and Proposition \ref{propB1}, one has  
    \begin{align*}
        \|\{\theta_k(I-\BB)\Pi(\sigma_{k,j}\cdot\nabla)w\}_{k,j}\|_{L^q(\ell^2)}&\le\|\{\theta_k(I-\BB)\Pi(\sigma_{k,j}\cdot\nabla)w\}_{k,j}\|_{\ell^2(L^q)}\\&\lesssim_{q,\alpha}\|\nabla w\|_q\lesssim_{q,\alpha} \|\tw\|_{H^{1,q}}.
    \end{align*}
    Therefore, one obtains
    \begin{align*}
        I_6\lesssim_{p,q,\alpha,\nu,\kappa,T}\|f\|_{L^p(\Omega\times(0,T);L^q)}+\|\tw_{\ini}\|_{L^p(\Omega;B^{1-2/p}_{q,p})}.
    \end{align*}
    Plugging the above estimate and \eqref{LS32} into \eqref{LS31} yields
    \begin{align*}
        \|\tw\|_{L^p(\Omega;H^{s,p}(0,T;H^{1-2s,q}))}\lesssim _{s,p,q,\alpha,\nu,\kappa,T}\|f\|_{L^p(\Omega\times(0,T);L^q)}+\|\tw_{\ini}\|_{L^p(\Omega;B^{1-2/p}_{q,p})}.
    \end{align*}
    This, combined with Proposition \ref{propB1}, ensures \eqref{LS12}. 

    \noindent\textbf{Step I\!V.} The degenerate case is handled in the same way, with the following notational modifications. Throughout the proof of Proposition \ref{propLS2} and Lemma \ref{lemmaLS3}, one replaces the transport term $\sigma_{k,j}\cdot\nabla$ by $\sigma_{k}\cdot\nabla_x$, and $S_\theta$ is defined in \eqref{INTRO9}$_2$. Since $\sigma_k$, considered as a three-dimensional vector field with zero vertical component, is divergence-free and tangent to the boundary $\partial D$, all integrations by parts used above produce no boundary contribution. Therefore, Proposition \ref{propLS2} and Lemma \ref{lemmaLS3} hold verbatim in the degenerate case.
    \end{proof}

    \section{Scaling limit of the Navier--Stokes equation with transport noise and cut-off}\label{SLofNS}
    In this section, we study the stochastic NS equation with cut-off
    \begin{align}
    \label{SL1}
    \begin{cases}
    \partial_tw-\nu\Delta w+\eta_R(\|w\|)[(u\cdot \nabla)w-(w\cdot\nabla)u]=S_\theta w+\sqrt{2\kappa}\Pi (\dot{W}\cdot \nabla)w,\\ u=\curl^{-1}w,\\(\partial_y w_3-\alpha w_3)|_{y=0}=(w_h-\alpha 
    u_h^\perp)|_{y=0}=0,\\(\partial_y w_3+\alpha w_3)|_{y=1}=(w_h+\alpha 
    u_h^\perp)|_{y=1}=0,\\
    w|_{t=0}=w_{\ini},
    \end{cases}
    \end{align}
    where $W$ is given in \eqref{INTRO4} or \eqref{INTRO6}, $S_\theta$ is defined accordingly in \eqref{INTRO9}, and $\eta:[0,\infty)\to[0,1]$ is a smooth function satisfying $\eta\equiv 1$ on $[0,R]$ and $\eta\equiv 0$ on $[2R,\infty)$. For $\gamma>0$, define the noise coefficients $\theta_\cdot$ by
    \begin{align}
        \label{SL2}\theta^N_k:=\frac{\CC_N}{|k|^\gamma} \I_{\{N\le |k|\le 2N\}},
    \end{align}
    where $\CC_N$ is chosen so that $\|\theta^N_\cdot\|_{\ell^2}=1$ for every $N$. We denote by $w^N_{\cut}$ the solution of \eqref{SL1} corresponding to the noise coefficient $\theta^N_\cdot$.
    
    The main goal of this section is to investigate the asymptotic behavior of $w^N_{\cut}$ as $N\to\infty$. The analysis is carried out in three steps. First, using the Meyers-type estimate obtained in Proposition \ref{propLS2}, we prove the tightness of the laws $\{\DDD(w^N_{\cut})\}_{N\ge1}$ in a suitable space-time topology. Second, under this topology, we identify the limit of the It\^o--Stratonovich corrector $S_\theta$. Finally, we show that $\{w^N_{\cut}\}_{N\ge1}$ converges to a process $w^{\Det}_{\cut}$, which solves a limiting deterministic equation.
    
    \subsection{Tightness of the cut-off solutions}
    The main result of this subsection is stated below.
    \begin{prop}\label{propSL1} Let $\alpha,\nu,\kappa, R>0$, and
    \[w_{\ini}\in L^\infty(\Omega;H_{\curl}\cap H^1).\]
    Then, for each $N\ge 1$, there exists a unique global solution $w^N_{\cut}$ of the cut-off equation. Moreover, there exists $\delta_0$ depending on $\alpha,\nu,\kappa$ such that for any $0<\delta<\delta_0$ and $T>0$, the laws $\{\DDD(w^N_{\cut})\}_{N\ge1}$ are tight in 
    \[\chi_{\delta,T}:=C([0,T];H^\delta)\cap L^2(0,T;H^{1-\delta}).\]
    \end{prop}
    \begin{proof}We only give the proof of the non-degenerate case, since the degenerate case follows from the same argument by replacing $\sigma_{k,j}\cdot\nabla$ by $\sigma_k\cdot\nabla_x$.
    
    \noindent\textbf{Step I.} In this step, we derive some uniform-in-$N$ a priori estimates. Utilizing the identity
    \[(u\cdot\nabla)w-(w\cdot\nabla)u=\curl((u\cdot \nabla) u),\]
    the cut-off equation \eqref{SL1} can be formulated in the form
    \begin{align}
        \label{SL3}\partial_t w-\nu\Delta w
    = \curl f_R(w)+S_{\theta}w
    +\sqrt{2\kappa}\Pi(\dot W\cdot\nabla)w,
    \end{align}
    where
    \[f_R(w):=-\eta_R(\|w\|)(u\cdot \nabla) u,
    \qquad u=\curl^{-1}w.\]
    By taking $q=2$ and $p\in[2,p_0)$ in Proposition \ref{propLS2}, for any $s\in[0,1/2)$, one has 
    \begin{align}
        \label{SL4}
        \|w^N_{\cut}\|_{L^p(\Omega;H^{s,p}(0,T;H^{1-2s}))}\lesssim_{s,p,\alpha,\nu,\kappa,T}\|f_R(w^N_{\cut})\|_{L^p(\Omega\times (0,T);L^2)}+\|w_{\ini}\|_{B^{1-2/p}_{2,p}},
    \end{align}
    where, by the H\"older inequality, the Sobolev embedding $H^1\hookrightarrow L^6$, Corollary \ref{coroA4}, and the Gagliardo--Nirenberg interpolation inequality
    \[\|v\|_3\lesssim \|v\|^{\frac{1}{2}}\|v\|^{\frac{1}{2}}_{H^1},\]
    it follows
    \begin{align*}
        \|(u\cdot\nabla)u\|\lesssim \|u\|_6\|\nabla u\|_3\lesssim \|u\|_{H^1}^{\frac{3}{2}}\|\nabla u\|_{H^1}^{\frac{1}{2}}\lesssim\|w\|^{\frac{3}{2}}\|w\|_{H^1}^{\frac{1}{2}}.
    \end{align*}
    This implies 
    \begin{align*}
    \|f_R(w^N_{\cut})\|\lesssim_R 1+\|\nabla w^N_{\cut}\|^{\frac{1}{2}},
    \end{align*}
    and thus 
    \begin{align}\label{SL5}
        \|f_R(w^N_{\cut})\|_{L^p(\Omega\times (0,T);L^2)}\le C_{T,R,\epsilon}+\epsilon\|\nabla w^N_{\cut}\|_{L^{p}(\Omega\times (0,T);L^2)}.
    \end{align}
    Plugging the above estimate into \eqref{SL4} and choosing $s=0$ and $\epsilon$ sufficiently small, one has  
    \begin{align*}\|w^N_{\cut}\|_{L^p(\Omega\times(0,T);H^{1})}\lesssim_{p,\alpha,\nu,\kappa,T,R}1+\|w_{\ini}\|_{B^{1-2/p}_{2,p}},
    \end{align*}
    which, together with \eqref{SL4}, \eqref{SL5}, and the embedding $H^1\hookrightarrow B^{1-2/p}_{2,p}$, yields
    \begin{align}
        \label{SL6}\sup_{N\ge 1}\|w^N_{\cut}\|_{L^p(\Omega;H^{s,p}(0,T;H^{1-2s}))}\lesssim_{s,p,\alpha,\nu,\kappa,T,R}1+\|w_{\ini}\|_{H^1}.
    \end{align}
    
    \noindent\textbf{Step I\!I.} The unique solvability of \eqref{SL1} follows from a standard Galerkin-type argument combined with the a priori estimate \eqref{SL6}. For simplicity, we give a sketch and omit the details. Indeed, the Galerkin approximation is a finite-dimensional SDE with locally Lipschitz coefficients. Moreover, the cut-off makes the nonlinear drift globally controlled, and the estimate \eqref{SL6} is uniform in the Galerkin dimension. Passing to the limit ensures a global solution. The uniqueness holds, because of the local Lipschitz property of $w\mapsto f_R(w)$. Therefore, for each $N\ge1$, there exists a unique global solution $w^N_{\cut}$.

    \noindent\textbf{Step I\!I\!I.} It remains to prove the tightness of the laws $\{\DDD(w^N_{\cut})\}_{N\ge1}$. Now choose $p\in (2,p_0)$ and $\frac{1}{p}<s<\frac{1}{2}$. Set
    \[\delta_0:=1-2s.\]
    Fix arbitrary positive $\delta<\delta_0$. By applying Theorems 2.1 and 2.2 in \cite{FG95} combined with the embedding
    \[H^{s,p}(0,T;H^{1-2s})\hookrightarrow W^{s,p}(0,T;H^{1-2s}),\]
    where $W^{s,p}$ denotes the fractional Slobodeckij--Sobolev space, one has 
    \[H^{s,p}(0,T;H^{1-2s})\hookrightarrow\hookrightarrow C([0,T];H^\delta)\]
    and 
    \[L^p(0,T;H^1)\cap H^{s,p}(0,T;H^{1-2s}) \hookrightarrow\hookrightarrow  L^p(0,T;H^{1-\delta})\hookrightarrow L^2(0,T;H^{1-\delta}).\]
    Together with \eqref{SL6}, this ensures the tightness of $\{\DDD(w^N_{\cut})\}_{N\ge1}$ in $\chi_{\delta,T}$.
    \end{proof}

    \subsection{Scaling limit of the It\^o--Stratonovich corrector}
    Next, we justify the scaling limit of the It\^o--Stratonovich corrector under the convergence topology specified in Proposition \ref{propSL1}.
    \subsubsection{The non-degenerate case}
    We first address the non-degenerate case. For $s>1/2$, define
    \[\mathfrak T_h: (L^2(\partial D))^2\to (H^{-s}(D))^2,\qquad g\mapsto \FT_h g,\]
    where $\FT_h g$ is given by 
    \begin{align*}\langle \FT_h g, \vartheta\rangle=\int_{\T^2} g\cdot \vartheta|_{y=0}\ddd x-\int_{\T^2} g\cdot \vartheta|_{y=1}\ddd x.
    \end{align*}
    Due to the trace theorem, the above operator $\FT_h$ is well-defined. Based on $\FT_h$, we introduce the operator $\FT:(L^2(\partial D))^2\to (H^{-s}(D))^3$ by 
    \begin{align}
        \label{SL6-0}\langle \FT g, \phi\rangle:=\langle\FT_h g, \phi_h\rangle,\qquad \forall \phi\in (H^s(D))^3.\end{align}
    
    The main result is stated as follows.
    \begin{prop}\label{propSL2}
        Let $\{\theta^N_\cdot\}_{N\ge 1}$ be defined in \eqref{SL2} and $s\in (1/2,1]$. For any $\varphi\in H_{\curl}\cap C^\infty$ satisfying $\varphi_h\in C_c^\infty(D)$,
        \begin{align}
            \label{SL6-1}
            \lim_{N\to\infty}\left\|S_{\theta^N}\varphi-\frac{4\kappa}{5}\Delta\varphi-\FT\left(\frac{2\kappa}{15}\nabla_x\varphi_3|_{\partial D}\right)\right\|_{H^{-s}}=0,
        \end{align}
        where $S_{\theta}$ is defined in \eqref{INTRO9}$_1$, and $\FT$ is given in \eqref{SL6-0}.
    \end{prop}
    \begin{proof}
    \noindent\textbf{Step I.} We first consider the case of a single mode. Let 
    \[l:=(l_h,n/2)=(l_1,l_2,n/2),\qquad l_1,l_2,n\in \Z\]
        be a wave number, and 
        \begin{align}
            \label{SL6-2}\varphi:=e^{2\pi i l_h\cdot x}\left(\begin{matrix}
        ib_{1}\sin n\pi y\\
        ib_{2}\sin n\pi y\\
         b_{3} \cos n\pi y
    \end{matrix}\right),
        \end{align}
    with $b:=(b_1,b_2,b_3)\in\C^3$ satisfying $b\cdot l=0$. If $n=0$, we take $b_h=0$. By the Euler formula, one has 
    \[\varphi=\frac{1}{2}(e^{2\pi il^+\cdot \z}b^{+}-e^{2\pi il^-\cdot \z}b^-),\]
    where $\z:=(x,y)$, 
    \[l^+:=l,\qquad b^+=b,\]
    and 
    \[l^-:=(l_h,-n/2),\qquad b^-:=(b_h,-b_3).\]
    Similarly, it follows that 
    \[\sigma_{k,j}=\frac{c_k}{2}(e^{2\pi ik^+\cdot \z}a_{k,j}^++e^{2\pi ik^-\cdot \z}a_{k,j}^-)\]
    with $c_k:=\sqrt{2}$ for $m\neq0$ and $c_k:=1$ for $m=0$, which implies 
    \begin{align}
        \label{SL7}
        (\sigma_{-k,j}\cdot\nabla)\varphi=\frac{c_k}{2}\pi i\ssum_{\rr,\s=\pm}e^{2\pi i(l^\rr-k^\s)\cdot\z}(a_{-k,j}^\s\cdot l^\rr) \fd_\rr b^\rr
    \end{align}
    with $\fd_+:=1$ and $\fd_-:=-1$. This, together with Proposition \ref{propA1}, yields
    \begin{align}
        \label{SL8}\Pi(\sigma_{-k,j}\cdot\nabla)\varphi=(\sigma_{-k,j}\cdot\nabla)\varphi-\nabla \psi_{k,j},
    \end{align}
    where $\psi_{k,j}$ solves
    \[\begin{cases}
        \Delta\psi_{k,j}=-c_k\pi^2\ssum_{\rr,\s=\pm}e^{2\pi i(l^\rr-k^\s)\cdot\z}(a_{-k,j}^\s\cdot l^\rr) (\fd_\rr b^\rr\cdot(l^\rr-k^\s)),\\
        \psi_{k,j}|_{y=0,1}=0.
    \end{cases}\]
    Plugging the ansatz $\psi_{k,j}=\phi_{k,j}(y)e^{2\pi i(l_h-k_h)\cdot x}$ into the above equation and applying the principle of superposition, one has the decomposition $\phi_{k,j}=\ssum_{\rr,\s=\pm}\phi^{\rr,\s}_{k,j}$ with $\phi^{\rr,\s}_{k,j}$ satisfying
    \begin{align}\label{SL8-1}\begin{cases}
        (\partial_y^2-4\pi^2|l_h-k_h|^2)\phi^{\rr,\s}_{k,j}=-c_k\pi^2e^{\pi i(\rr n-\s m)y}(a_{-k,j}^\s\cdot l^\rr) (\fd_\rr b^\rr\cdot(l^\rr-k^\s))=:F^{\rr,\s}_{k,j},\\
        \phi^{\rr,\s}_{k,j}|_{y=0,1}=0.
    \end{cases}
    \end{align}
    Since $\theta^N_k$ is supported on high modes $\{N\le |k|\le 2N\}$, to study the scaling limit of $S_{\theta^N}$, one has to understand the asymptotic behavior of $\phi^{\rr,\s}_{k,j}$ as $|k|\to\infty$. Notice that the above equation with large $|k|$ is a singular-limit problem. It is natural to introduce the decomposition
    \begin{align}
        \label{SL8-2}\phi^{\rr,\s}_{k,j}=\phi^{\rr,\s}_{\bulk,k,j}+\phi^{\rr,\s}_{\bd,k,j}.
    \end{align}
    Here, the bulk part
    \begin{align}
        \label{SL9}\phi^{\rr,\s}_{\bulk,k,j}:=-\frac{1}{4\pi^2|l^\rr-k^\s|^2}F^{\rr,\s}_{k,j}
    \end{align}
    satisfies
    \[(\partial_y^2-4\pi^2|l_h-k_h|^2)\phi^{\rr,\s}_{\bulk,k,j}=F^{\rr,\s}_{k,j},\]
    and the boundary layer term $\phi^{\rr,\s}_{\bd,k,j}$ is introduced to compensate the mismatch of the boundary conditions of $\phi^{\rr,\s}_{\bulk,k,j},\phi^{\rr,\s}_{k,j}$. Notice that $\phi^{\rr,\s}_{\bd,k,j}$ satisfies 
    \begin{align*}
    \begin{cases}
        (\partial_y^2-4\pi^2|l_h-k_h|^2)\phi^{\rr,\s}_{\bd,k,j}=0,\\
        \phi^{\rr,\s}_{\bd,k,j}|_{y=0,1}=-\phi^{\rr,\s}_{\bulk,k,j}|_{y=0,1},
    \end{cases}
    \end{align*}
    it can be solved explicitly by
    \begin{align}
        \label{SL10}
        \phi^{\rr,\s}_{\bd,k,j}=-\phi^{\rr,\s}_{\bulk,k,j}(0)\frac{\sinh \lambda_{k}(1-y)}{\sinh \lambda_{k}}-\phi^{\rr,\s}_{\bulk,k,j}(1)\frac{\sinh \lambda_{k}y}{\sinh \lambda_{k}},
    \end{align}
    if $\lambda_k:=2\pi|l_h-k_h|\neq0$; and
    \begin{align}
        \label{SL10-1}
        \phi^{\rr,\s}_{\bd,k,j}(y)=-y\phi^{\rr,\s}_{\bulk,k,j}(1)-(1-y)\phi^{\rr,\s}_{\bulk,k,j}(0)
    \end{align}
    if $\lambda_k:=2\pi|l_h-k_h|=0$. 
    
    To summarize, from \eqref{SL8}, it follows that 
    \begin{align*}\Pi(\sigma_{-k,j}\cdot\nabla)\varphi=(\sigma_{-k,j}\cdot\nabla)\varphi-\nabla\left[e^{2\pi i(l_h-k_h)\cdot x}\ssum_{\rr,\s=\pm}\left(\phi^{\rr,\s}_{\bulk,k,j}+\phi^{\rr,\s}_{\bd,k,j}\right)\right].
    \end{align*}
    This gives the decomposition
    \begin{align}
        \label{SL17-1}S_{\theta^N}\varphi=S_{\theta^N,\bulk}\varphi-S_{\theta^N,\bd}\varphi,
    \end{align}
    where
    \begin{align*}
        S_{\theta^N,\bulk}\varphi&:=2\kappa \ssum_{k\neq 0}\ssum_{j=1,2}(\theta^N_k)^2\Pi(\sigma_{k,j}\cdot\nabla)(\sigma_{-k,j}\cdot \nabla)\varphi\\&\quad-2\kappa \ssum_{k\neq 0}\ssum_{j=1,2}(\theta^N_k)^2
    \Pi(\sigma_{k,j}\cdot\nabla)\nabla\left(e^{2\pi i(l_h-k_h)\cdot x}\ssum_{\rr,\s=\pm}\phi^{\rr,\s}_{\bulk,k,j}\right),
    \end{align*}
    and
    \begin{align}
        \label{SL11}S_{\theta^N,\bd}\varphi=2\kappa \ssum_{k\neq 0}\ssum_{j=1,2}(\theta^N_k)^2\Pi(\sigma_{k,j}\cdot\nabla)\nabla\left(e^{2\pi i(l_h-k_h)\cdot x}\ssum_{\rr,\s=\pm}\phi^{\rr,\s}_{\bd,k,j}\right).
    \end{align}
    
    \noindent\textbf{Step I\!I.} In this step, we estimate $S_{\theta^N,\bd}\varphi$. By the solution formulas \eqref{SL9}--\eqref{SL10-1}, it follows that 
    \begin{align}
        \label{SL12}
        \|\phi^{\rr,\s}_{\bd,k,j}\|&\lesssim\|\phi^{\rr,\s}_{\bulk,k,j}\|_\infty\left(\|e^{-\lambda_k y}\|+\|e^{-\lambda_k (1-y)}\|\right)\notag\\&\lesssim\frac{1}{\sqrt{\lambda_k}|l^\rr-k^\s|^2}\|F^{\rr,\s}_{k,j}\|_\infty\lesssim \frac{1}{\sqrt{|l_h-k_h|}|l^\rr-k^\s|},
    \end{align}
    and
    \begin{align}
        \label{SL13}
        \|\partial_y\phi^{\rr,\s}_{\bd,k,j}\|\lesssim \frac{\sqrt{|l_h-k_h|}}{|l^\rr-k^\s|},\qquad \|\partial^2_y\phi^{\rr,\s}_{\bd,k,j}\|\lesssim \frac{|l_h-k_h|^{\frac{3}{2}}}{|l^\rr-k^\s|},
    \end{align}
    if $\lambda_k\neq0$. For $\lambda_k=0$, $e^{2\pi i(l_h-k_h)\cdot x}\phi^{r,s}_{\bd,k,j}=\phi^{r,s}_{\bd,k,j}$ is affine in $y$ and independent of $x$, which implies that
    \[(\sigma_{k,j}\cdot\nabla)\nabla\left(e^{2\pi i(l_h-k_h)\cdot x}\ssum_{\rr,\s=\pm}\phi^{\rr,\s}_{\bd,k,j}\right)=0,\]
    and therefore its contribution to $S_{\theta^N,\bd}\varphi$ vanishes. On the other hand, since
    \[\ssum_{N\le |k|\le 2N}|k|^{-2\gamma}\sim N^{3-2\gamma},\]
    it follows that the constant $\CC_N$ in \eqref{SL2} satisfies $\CC_N\sim N^{\gamma-3/2}$, which yields 
    \begin{align}
        \label{SL14}
        \theta_k^N\sim N^{-\frac{3}{2}}.
    \end{align}
    Combining \eqref{SL11}--\eqref{SL14}, one gets 
    \begin{align}\label{SL15}
        \|S_{\theta^N,\bd}\varphi\|&\lesssim\kappa \ssum_{N\le |k|\le 2N}\ssum_{j=1,2}N^{-3}\left\|\nabla^2\left(e^{2\pi i(l_h-k_h)\cdot x}\ssum_{\rr,\s=\pm}\phi^{\rr,\s}_{\bd,k,j}\right)\right\|\notag\\&\lesssim\kappa N^{-3}\ssum_{N\le |k|\le 2N}\frac{|l_h-k_h|^{\frac{3}{2}}}{|l^\rr-k^\s|}\lesssim \kappa N^{\frac{1}{2}}.
    \end{align}
    Next, we establish an estimate in $H^{-1}$ for $S_{\theta^N,\bd}\varphi$. Let $\vartheta\in H^1$ be a test function. By integrating by parts and using \eqref{SL12} and \eqref{SL13}, one has
    \begin{align*}
        &\left\langle \Pi(\sigma_{k,j}\cdot\nabla)\nabla\left(e^{2\pi i(l_h-k_h)\cdot x}\ssum_{\rr,\s=\pm}\phi^{\rr,\s}_{\bd,k,j}\right),\vartheta\right\rangle\\&\qquad\qquad\qquad\qquad=\left\langle \nabla\left(e^{2\pi i(l_h-k_h)\cdot x}\ssum_{\rr,\s=\pm}\phi^{\rr,\s}_{\bd,k,j}\right),(\sigma_{-k,j}\cdot\nabla)\Pi\vartheta\right\rangle\\&\qquad\qquad\qquad\qquad\lesssim \frac{\sqrt{|l_h-k_h|}}{|l^\rr-k^\s|}\|\vartheta\|_{H^1},
    \end{align*}
    which implies 
    \[\left\|\Pi(\sigma_{k,j}\cdot\nabla)\nabla\left(e^{2\pi i(l_h-k_h)\cdot x}\ssum_{\rr,\s=\pm}\phi^{\rr,\s}_{\bd,k,j}\right)\right\|_{H^{-1}}\lesssim \frac{\sqrt{|l_h-k_h|}}{|l^\rr-k^\s|}\]
    and thus 
    \begin{align}
        \label{SL16}
        \|S_{\theta^N,\bd}\varphi\|_{H^{-1}}\lesssim \kappa \ssum_{k\neq 0}(\theta^N_k)^2 \frac{\sqrt{|l_h-k_h|}}{|l^\rr-k^\s|}\lesssim \kappa N^{-\frac{1}{2}}.
    \end{align}
    Therefore, by using interpolation and \eqref{SL15}, \eqref{SL16}, it follows that 
    \begin{align}
        \label{SL17}
        \lim_{N\to\infty}\left\|S_{\theta^N,\bd}\varphi\right\|_{H^{-s}}=0,\qquad s\in (1/2,1].
    \end{align}

    \noindent\textbf{Step I\!I\!I.} We turn to the bulk part $S_{\theta,\bulk}$. First, in the computation below we shall use the normalization $c_k=\sqrt2$ for all $k$. This does not change the scaling limit. Indeed, it only modifies the modes on the layer $\{m=0\}$, which contains only $O(N^2)$ lattice points in the shell $\{N\le |k|\le 2N\}$. Moreover, on this layer, it follows
    \[\sigma_{k,j}\cdot \nabla= c_ke^{2\pi ik_h\cdot x}a_{k,j,h}\cdot\nabla_x,\qquad a_{k,j,h}\cdot k_h=0,\]
    and
    \[(\sigma_{k,j}\cdot \nabla)(\sigma_{-k,j}\cdot \nabla)\varphi=c_k^2(a_{k,j,h}\cdot\nabla_x)^2\varphi,\]
    and 
    \begin{align*}
        (\sigma_{k,j}\cdot\nabla)\nabla&\big(e^{2\pi i(l_h-k_h)\cdot x}\ssum_{\rr,\s=\pm}\phi^{\rr,\s}_{\bulk,k,j}\big)\\&\qquad\qquad=-2\pi^2c_ke^{2\pi i l_h\cdot x}(a_{k,j,h}\cdot l_h)(2(l_h-k_h),\rr n)\left(\ssum_{\rr,\s=\pm}\phi^{\rr,\s}_{\bulk,k,j}\right).
    \end{align*}
    Together with \eqref{SL9} and \eqref{SL14}, this gives 
    \begin{align*}
        \|\ssum_{k\neq 0,m=0}\ssum_{j=1,2}(\theta^N_k)^2\Pi(\sigma_{k,j}\cdot\nabla)(\sigma_{-k,j}\cdot \nabla)\varphi\|\lesssim  \ssum_{N\le |k|\le 2N,m=0}(\theta^N_k)^2\lesssim N^{-1},
    \end{align*}
    and 
    \begin{align*}
        \|\ssum_{k\neq 0,m=0}\ssum_{j=1,2}(\theta^N_k)^2\Pi (\sigma_{k,j}\cdot\nabla)\nabla&\big(e^{2\pi i(l_h-k_h)\cdot x}\ssum_{\rr,\s=\pm}\phi^{\rr,\s}_{\bulk,k,j}\big)\|\\&\qquad\lesssim  \ssum_{N\le |k|\le 2N,m=0}(\theta^N_k)^2\lesssim N^{-1}.
    \end{align*}
    Therefore, modifying the normalization parameter $c_k$ on the layer $\{m=0\}$ does not change the scaling limit of $S_{\theta^N,\bulk}\varphi$. Since the estimates \eqref{SL15} and \eqref{SL16} are uniform with respect to $c_k$, the modification does not affect the vanishing of $S_{\theta^N,\bd}\varphi$, and thus does not change the scaling limit of $S_{\theta^N}\varphi$.

    Let
    \begin{align}\label{SL17-2}\A^{\rr,\s}_{k,j}:=-\sqrt{2}\pi^2(a_{-k,j}^\s\cdot l^\rr) \fd_\rr b^\rr,\end{align}
    so one has 
    \[F^{\rr,\s}_{k,j}=e^{\pi i(\rr n-\s m)y}\A^{\rr,\s}_{k,j}\cdot (l^\rr-k^\s),\qquad \phi^{\rr,\s}_{\bulk,k,j}=-\frac{e^{\pi i(\rr n-\s m)y}\A^{\rr,\s}_{k,j}\cdot (l^\rr-k^\s)}{4\pi^2|l^\rr-k^\s|^2},\]
    and 
    \[(\sigma_{-k,j}\cdot\nabla)\varphi=\frac{1}{2\pi i}\ssum_{\rr,\s=\pm}e^{2\pi i(l^\rr-k^\s)\cdot\z}\A^{\rr,\s}_{k,j}.\]
    Therefore, it follows that 
    \begin{align*}
    2\kappa \ssum_{k\neq 0}\ssum_{j=1,2}&(\theta^N_k)^2
    \Pi(\sigma_{k,j}\cdot\nabla)(\sigma_{-k,j}\cdot \nabla)\varphi\\&=\frac{-i\kappa}{\pi}\Pi\ssum_{k\neq 0,j=1,2}\ssum_{\rr,\s=\pm}(\theta^N_k)^2\A^{\rr,\s}_{k,j}
    (\sigma_{k,j}\cdot\nabla) e^{2\pi i(l^\rr-k^\s)\cdot\z},
    \end{align*}
    and
    \begin{align*}
    2\kappa \ssum_{k\neq 0}&\ssum_{j=1,2}(\theta^N_k)^2
    \Pi(\sigma_{k,j}\cdot\nabla)\nabla\left(e^{2\pi i(l_h-k_h)\cdot x}\ssum_{\rr,\s=\pm}\phi^{\rr,\s}_{\bulk,k,j}\right)\\
        &=\frac{-\kappa}{2\pi^2} \ssum_{k\neq 0,j=1,2}\ssum_{\rr,\s=\pm}(\theta^N_k)^2
    \Pi(\sigma_{k,j}\cdot\nabla)\nabla\left(e^{2\pi i(l^\rr-k^\s)\cdot\z}\frac{\A^{\rr,\s}_{k,j}\cdot (l^\rr-k^\s)}{|l^\rr-k^\s|^2}\right)\\&=\frac{-i\kappa}{\pi}\Pi\ssum_{k\neq 0,j=1,2}\ssum_{\rr,\s=\pm}(\theta^N_k)^2
    \left(\PPPP_{l^\rr-k^\s}\A^{\rr,\s}_{k,j}(\sigma_{k,j}\cdot\nabla)e^{2\pi i(l^\rr-k^\s)\cdot\z}\right),
    \end{align*}
    where for any vector $\zeta\in\R^3$, 
    \[\PPPP_{\zeta}:=\frac{\zeta\otimes\zeta}{|\zeta|^2},\qquad\PPPP_{\zeta}^\perp:=I_3-\PPPP_{\zeta}\]
    are projection matrices in $\R^3$. This yields
    \begin{align*}
        S_{\theta^N,\bulk}\varphi=\frac{-i\kappa}{\pi}\Pi\ssum_{k\neq 0,j=1,2}\ssum_{\rr,\s=\pm}(\theta^N_k)^2\PPPP_{l^\rr-k^\s}^\perp
    \A^{\rr,\s}_{k,j}(\sigma_{k,j}\cdot\nabla)e^{2\pi i(l^\rr-k^\s)\cdot\z}.
    \end{align*}
    Since 
    \[(\sigma_{k,j}\cdot\nabla)e^{2\pi i(l^\rr-k^\s)\cdot\z}=\sqrt{2}\pi i\ssum_{\ttt=\pm}e^{2\pi i(k^\ttt+l^\rr-k^\s)\cdot \z}a_{k,j}^{\ttt}\cdot (l^\rr-k^\s),\]
    we introduce the decomposition
    \begin{align}
        \label{SL18}S_{\theta^N,\bulk}\varphi=S_{\theta^N,\bulk,1}\varphi+S_{\theta^N,\bulk,2}\varphi,
    \end{align}
    where 
    \begin{align}
        \label{SL19}S_{\theta^N,\bulk,1}\varphi&:=\sqrt{2}\kappa\Pi\ssum_{k\neq 0,j=1,2}\ssum_{\rr,\s=\pm,\ttt=\s}(\theta^N_k)^2\PPPP_{l^\rr-k^\s}^\perp\A^{\rr,\s}_{k,j} e^{2\pi il^\rr\cdot \z}a_{k,j}^{\s}\cdot (l^\rr-k^\s)
    \end{align}
    is the non-oscillatory bulk part, corresponding to the matched interactions $\ttt=\s$, and 
    \begin{align}
        \label{SL20}S_{\theta^N,\bulk,2}\varphi:=\sqrt{2}\kappa\Pi\ssum_{k\neq 0,j=1,2}\ssum_{\rr,\s=\pm,\ttt\neq\s}(\theta^N_k)^2&\PPPP_{l^\rr-k^\s}^\perp\A^{\rr,\s}_{k,j}\notag\\\times &e^{2\pi i(k^\ttt+l^\rr-k^\s)\cdot \z}a_{k,j}^{\ttt}\cdot (l^\rr-k^\s)
    \end{align}
    is the oscillatory bulk part, corresponding to the unmatched interactions $\ttt\neq\s$. The latter is oscillatory in the vertical direction.
    
    \noindent\textbf{Step I\!V.} In this step, we justify the scaling limit of the non-oscillatory bulk part $S_{\theta^N,\bulk,1}\varphi$. By the definition \eqref{SL17-2} of $\A^{\rr,\s}_{k,j}$ and \eqref{INTRO1-2}, it follows that 
    \[S_{\theta^N,\bulk,1}\varphi=-2\pi^2\kappa\Pi\ssum_{k\neq 0,j=1,2}\ssum_{\rr,\s=\pm}(\theta^N_k)^2|a_{k,j}^{\s}\cdot l^\rr|^2\PPPP_{l^\rr-k^\s}^\perp (e^{2\pi il^\rr\cdot \z}\fd_\rr b^\rr).\]
    Since $l,b$ are given by the function $\varphi$, it suffices to compute the coefficient matrix
    \[\SSSS^\rr_{\theta^N}:=-2\pi^2\kappa\ssum_{k\neq 0,j=1,2}\ssum_{\s=\pm}(\theta^N_k)^2 |a_{k,j}^{\s}\cdot l^\rr|^2\PPPP_{l^\rr-k^\s}^\perp.\]
    By applying the estimate
    \begin{align}
        \label{SL20-1}|\PPPP_{l^\rr-k^\s}-\PPPP_{k^\s}|\lesssim\frac{|l|}{|k|},
    \end{align}
    see Lemma 5.5 in \cite{FL21}, one has
    \[\SSSS^\rr_{\theta^N}=-2\pi^2\kappa\ssum_{k\neq 0,j=1,2}\ssum_{\s=\pm}(\theta^N_k)^2|a_{k,j}^{\s}\cdot l^\rr|^2\PPPP_{k^\s}^\perp +o(1).\]
    Furthermore, by using the radial symmetry of $\theta^N_\cdot$, for any $\s$, one has 
    \begin{align*}
        \ssum_{k\neq 0,j=1,2}(\theta^N_k)^2 |a_{k,j}^{\s}\cdot l^\rr|^2\PPPP_{k^\s}^\perp&=\ssum_{k\neq 0,j=1,2}(\theta^N_k)^2 |a^\s_{k^\s,j}\cdot l^\rr|^2\PPPP_{k}^\perp,
    \end{align*}
    where, by using \eqref{INTRO2}, it follows
    \[\ssum_{k\neq 0,j=1,2}(\theta^N_k)^2 |a^+_{k^+,j}\cdot l^\rr|^2\PPPP_{k}^\perp=\ssum_{k\neq 0}(\theta^N_k)^2|\PPPP_{k}^\perp l^\rr|^2 \PPPP_{k}^\perp,\]
    and
    \begin{align*}
        \ssum_{k\neq 0,j=1,2}(\theta^N_k)^2 |a^-_{k^-,j}\cdot l^\rr|^2\PPPP_{k}^\perp&=\ssum_{k\neq 0,j=1,2}(\theta^N_k)^2 |a_{k^-,j}\cdot l^{-\rr}|^2\PPPP_{k}^\perp\\&=\ssum_{k\neq 0,j=1,2}(\theta^N_k)^2 (l^{-\rr})^T\left(I_3-\frac{k^-\otimes k^{-}}{|k|^2}\right)l^{-\rr}\PPPP_{k}^\perp\\&=\ssum_{k\neq 0}(\theta^N_k)^2|\PPPP_{k}^\perp l^\rr|^2 \PPPP_{k}^\perp.
    \end{align*}
    Therefore,
    \begin{align}\label{SL21}
    \SSSS^\rr_{\theta^N}=-4\pi^2\kappa\ssum_{k\neq 0}(\theta^N_k)^2|\PPPP_{k}^\perp l^\rr|^2 \PPPP_{k}^\perp+o(1).\end{align}
    For any $\rr=\pm$, by proceeding as in the proof of Lemma 5.6 in \cite{FL21} and using the polar coordinates, it follows
    \begin{align}\label{SL22}
        \lim_{N\to\infty}\ssum_{k\neq 0}(\theta^N_k)^2|\PPPP_{k}^\perp l^\rr|^2 \PPPP_{k}^\perp =\frac{1}{4\pi}\!\int_{|\zeta|=1}\!|\PPPP_{\zeta}^\perp l^\rr|^2 \PPPP_{\zeta}^\perp\ddd S.
    \end{align}
    Let $Q_{l^\rr}$ be the orthogonal matrix such that $Q^T\frac{l^\rr}{|l^\rr|}=e_1$, where $e_1$ is the unit vector in $\R^3$ with the first component being $1$. Using the coordinate transform $\zeta\mapsto Q\zeta$, one gets
    \begin{align}\label{SL23}\notag
        \frac{1}{4\pi}\!\int_{|\zeta|=1}\!|\PPPP_{\zeta}^\perp l^\rr|^2 \PPPP_{\zeta}^\perp\ddd S&=\frac{|l^\rr|^2}{4\pi}Q\left(\int_{|\zeta|=1}|\PPPP_{\zeta}^\perp e_1|^2 \PPPP_{\zeta}^\perp \ddd S\right)Q^T\\&=\frac{|l^\rr|^2}{4\pi}Q\left(\int_{|\zeta|=1}(1-\zeta_1^2) (I_3-\zeta\otimes \zeta)\ddd S\right)Q^T=:\frac{|l^\rr|^2}{4\pi}Q\MM Q^T.
    \end{align}
    By symmetry, for $i\neq j$,  
    \[\MM_{ij}=\int_{|\zeta|=1}(1-\zeta_1^2) (\delta_{ij}-\zeta_i\zeta_j)\ddd S=0.\]
    Through direct computation, one has 
    \[\MM_{11}=\int_{|\zeta|=1}(1-\zeta_1^2)^2 \ddd S=\frac{32\pi}{15},\qquad \MM_{22}=\int_{|\zeta|=1}(1-\zeta_1^2)(1-\zeta_2^2) \ddd S=\frac{8\pi}{5},\]
    and by symmetry, $\MM_{33}=\MM_{22}=\frac{8\pi}{5}$. Hence, 
    \[\MM=4\pi\diag \left\{\frac{8}{15},\frac{2}{5},\frac{2}{5}\right\},\]
    which, together with \eqref{SL23}, yields
    \[\frac{1}{4\pi}\int_{|\zeta|=1} |\PPPP_{\zeta}^\perp l^\rr|^2 \PPPP_{\zeta}^\perp\ddd S=\frac{2}{5}|l^\rr|^2I_3+\frac{2}{15}l^\rr\otimes l^\rr.\]
    Combining this with \eqref{SL21} and \eqref{SL22} and using the fact $b^\rr\cdot l^\rr=0$, it follows that 
    \begin{align}
        \label{SL24}
        \notag\lim_{N\to\infty}S_{\theta^N,\bulk,1}\varphi&=-4\pi^2\kappa\Pi\ssum_{\rr=\pm}\fd_\rr\left(\frac{2}{5}|l^\rr|^2e^{2\pi il^\rr\cdot \z} b^\rr+\frac{2}{15}e^{2\pi il^\rr\cdot \z} l^\rr\otimes l^\rr b^\rr\right)\\&=\frac{4\kappa}{5}\Pi\Delta\varphi=\frac{4\kappa}{5}\Delta\varphi,
    \end{align}
    where the convergence holds in $L^2$.

    \noindent\textbf{Step V.} We turn to the oscillatory bulk part $S_{\theta^N,\bulk,2}$. Let $J:=\diag\{1,1,-1\}.$ Notice that $Jl^+=l^-$, $Jk^+=k^-$, $Ja_{k,j}^+=a_{k,j}^-$,
    \[\A^{+,-}_{k,j}=-\sqrt{2}\pi^2(a_{-k,j}^-\cdot l^+) b^+=\sqrt{2}\pi^2(a_{-k,j}^+\cdot l^-)\fd_-Jb^-=-J\A^{-,+}_{k,j},\]
    $\A^{+,+}_{k,j}=-J\A^{-,-}_{k,j}$, and 
    \[\PPPP_{Jk}^{\perp}=\left(I_3-\frac{Jk\otimes Jk}{|k|^2}\right)=J\PPPP_{k}^{\perp}J.\]
    By using the above relations, it follows that
    \begin{align*}
        S_{\theta^N,\bulk,2}\varphi =\sqrt{2}\kappa\Pi &\ssum_{k\neq0,j=1,2}\ssum_{\rr=\pm} (\theta_k^N)^2(a_{k,j}^+\cdot l^\rr+a_{k,j,3}m)e^{2\pi il_h\cdot x}\\&\times\left(\PPPP_{l^\rr-k^-}^\perp \A_{k,j}^{\rr,-}e^{2\pi i(\frac{\rr n}{2}+m)y}-J(\PPPP_{l^\rr-k^-}^\perp \A_{k,j}^{\rr,-})e^{-2\pi i(\frac{\rr n}{2}+m)y}\right).
    \end{align*}
    Define 
    \begin{align}
        \label{SL24-1}\V^\rr_{k,j}:=(\theta_k^N)^2(a_{k,j}^+\cdot l^\rr+a_{k,j,3}m)\PPPP_{l^\rr-k^-}^\perp \A_{k,j}^{\rr,-},
    \end{align}
    so one has 
    \begin{align}
        \label{SL25}
        S_{\theta^N,\bulk,2}\varphi &=\sqrt{2}\kappa \ssum_{k\neq0,j=1,2}\ssum_{\rr=\pm} \Pi\left(\V^\rr_{k,j}e^{2\pi i(\frac{\rr n}{2}+m)y}-J\V^\rr_{k,j}e^{-2\pi i(\frac{\rr n}{2}+m)y}\right)e^{2\pi il_h\cdot x}\notag\\&=2\sqrt{2}\kappa\ssum_{k\neq0,j=1,2}\ssum_{\rr=\pm} \Pi \left(\V^\rr_{k,j,h}i\sin(\lambda_{m,\rr}y), \V^\rr_{k,j,3}\cos(\lambda_{m,\rr}y)\right)e^{2\pi il_h\cdot x}\notag\\
        &=:2\sqrt{2}\kappa \ssum_{k\neq0,j=1,2}\ssum_{\rr=\pm} \Pi\UU^\rr_{k,j},
    \end{align}
    where $\lambda_{m,\rr}:=2\pi \left(\frac{\rr n}{2}+m\right)$. Applying Proposition \ref{propA1} and using the fact
    \begin{align*}
        \int_D \UU^{\rr}_{k,j,3}\ddd x\ddd y&=\I_{\{l_h=0,\lambda_{m,\rr}=0\}}\V^\rr_{k,j,3}\\[-0.4em]&=\I_{\{l_h=0,\lambda_{m,\rr}=0\}}(\theta_k^N)^2(a_{k,j}^+\cdot l^\rr+a_{k,j,3}m)(\PPPP_{l^\rr-k^-}^\perp \A_{k,j}^{\rr,-})_3\\&=\I_{\{l_h=0,\lambda_{m,\rr}=0\}}(\theta_k^N)^2(a_{k,j,3}(\rr n/2)+a_{k,j,3}m)(\PPPP_{l^\rr-k^-}^\perp \A_{k,j}^{\rr,-})_3=0,
    \end{align*}
    it follows
    \begin{align}\label{SL25-1}
    \Pi \UU^\rr_{k,j}= \UU^\rr_{k,j}-\nabla \pp^\rr_{k,j},
    \end{align}
    where $\pp^\rr_{k,j}$ satisfies
    \begin{align}
        \label{SL26}
        \begin{cases}
        \Delta\pp^\rr_{k,j}=\dive \UU^\rr_{k,j}= - e^{2\pi il_h\cdot x}\sin(\lambda_{m,\rr} y)(2\pi l_h\cdot\V^\rr_{k,j,h}+\lambda_{m,\rr}\V^\rr_{k,j,3}),\\
        \pp^\rr_{k,j}|_{y=0,1}=0.
        \end{cases}
    \end{align}
    For $4\pi^2|l_h|^2+\lambda_{m,\rr}^2>0$, the solution is given explicitly by 
    \begin{align}
        \label{SL28}
        \pp^\rr_{k,j}:= e^{2\pi il_h\cdot x}\sin(\lambda_{m,\rr} y)\frac{2\pi l_h\cdot\V^\rr_{k,j,h}+\lambda_{m,\rr}\V^\rr_{k,j,3}}{4\pi^2|l_h|^2+\lambda_{m,\rr}^2}.
    \end{align}
    On the other hand, if $4\pi^2|l_h|^2+\lambda_{m,\rr}^2=0$, then $\V^\rr_{k,j}=0$, which implies $\Pi\UU^\rr_{k,j}=0$, and thus this case does not contribute to the scaling limit.

    We claim that 
    \begin{align}
        \label{SL29}
        \lim_{N\to\infty} 2\sqrt{2}\kappa \ssum_{k\neq0,j=1,2}\ssum_{\rr=\pm}( \UU^\rr_{k,j,3}-\partial_y \pp^\rr_{k,j})=0
    \end{align}
    in $L^2$. Indeed, notice that 
    \[\partial_y\pp^\rr_{k,j}= e^{2\pi il_h\cdot x}\cos(\lambda_{m,\rr} y)\frac{2\pi \lambda_{m,\rr}l_h\cdot\V^\rr_{k,j,h}+\lambda^2_{m,\rr}\V^\rr_{k,j,3}}{4\pi^2|l_h|^2+\lambda_{m,\rr}^2},\]
    which gives 
    \begin{align*}
        \UU^\rr_{k,j,3}-\partial_y\pp^\rr_{k,j}=e^{2\pi il_h\cdot x}\cos(\lambda_{m,\rr} y)\frac{4\pi^2|l_h|^2\V^\rr_{k,j,3}-2\pi \lambda_{m,\rr}l_h\cdot\V^\rr_{k,j,h}}{4\pi^2|l_h|^2+\lambda_{m,\rr}^2}.
    \end{align*}
    Since $l_h$ is given as the tangential frequency of $\varphi$, by the Bessel inequality in the vertical direction combined with the estimate 
    \begin{align}
        \label{SL30}
        |\V_{k,j}^\rr|\lesssim (\theta_k^N)^2 |m|\lesssim N^{-3}|m|,
    \end{align}
    it follows
    \begin{align}\label{SL30-1}
        &\left\|\ssum_{k\neq 0,j=1,2}\ssum_{\rr=\pm}(\UU^\rr_{k,j,3}-\partial_y\pp^\rr_{k,j})\right\|\notag\\&\lesssim\ssum_{k_h\in\Z^2,j=1,2}\left\|\ssum_{m\in\Z}\I_{\{N\le |k|\le 2N\}}\frac{4\pi^2|l_h|^2\V^\rr_{k,j,3}-2\pi \lambda_{m,\rr}l_h\cdot\V^\rr_{k,j,h}}{4\pi^2|l_h|^2+\lambda_{m,\rr}^2}\cos(\lambda_{m,\rr} y)\right\|_{L^2_y}\notag\\&\lesssim N^{-3}\ssum_{k_h\in\Z^2}\left(\ssum_{m\in\Z}\I_{\{N\le |k|\le 2N\}}\right)^{\frac{1}{2}}\lesssim N^{-\frac{1}{2}}\to 0,
    \end{align}
    as $N\to\infty.$ 
    
    Next, we turn to the horizontal components. Notice that 
    \begin{align*}
        \nabla_x\pp^\rr_{k,j}:=2\pi ie^{2\pi il_h\cdot x}\sin(\lambda_{m,\rr} y)\frac{2\pi l_h\cdot\V^\rr_{k,j,h}+\lambda_{m,\rr}\V^\rr_{k,j,3}}{4\pi^2|l_h|^2+\lambda_{m,\rr}^2}l_h,
    \end{align*}
    which gives 
    \begin{align*}
        \UU^\rr_{k,j,h}-\nabla_x\pp^\rr_{k,j}&=ie^{2\pi il_h\cdot x}\sin(\lambda_{m,\rr} y)\V^\rr_{k,j,h}-ie^{2\pi il_h\cdot x}\sin(\lambda_{m,\rr} y)\frac{4\pi^2 l_h\otimes l_h}{4\pi^2|l_h|^2+\lambda_{m,\rr}^2}\V^\rr_{k,j,h}\\&\qquad-ie^{2\pi il_h\cdot x}\sin(\lambda_{m,\rr} y)\frac{2\pi\lambda_{m,\rr}\V^\rr_{k,j,3}}{4\pi^2|l_h|^2+\lambda_{m,\rr}^2}l_h\\&=:ie^{2\pi il_h\cdot x}\sin(\lambda_{m,\rr} y)\V^\rr_{k,j,h}+\RR_{k,j}.
    \end{align*}
    Repeating the derivation of \eqref{SL30-1}, one has 
    \begin{align}
        \label{SL31-1}
        \lim_{N\to\infty}\left\|\ssum_{k\neq 0}\ssum_{j=1,2}\RR_{k,j}\right\|=0,
    \end{align}
    so it remains to control the main part
    \[\SSSSS_{\theta^N}:=2\sqrt{2}\kappa i\ssum_{k\neq 0,j=1,2}\ssum_{\rr=\pm}e^{2\pi il_h\cdot x}\sin(\lambda_{m,\rr} y)\V^\rr_{k,j,h}.\]
    Define the boundary trace
    \[\SSSSS_{\theta^N,\bd}:=\begin{cases}
        2\sqrt{2}\kappa i\left(\ssum_{k\neq 0,\lambda_{m,\rr} \neq0}\ssum_{j=1,2,\rr=\pm}\lambda_{m,\rr}^{-1}\V^\rr_{k,j,h}\right)e^{2\pi il_h\cdot x},\qquad \ \ \ \ \mbox{on $\{y=0\}$},\\2\sqrt{2}\kappa i(-1)^n\left(\ssum_{k\neq 0,\lambda_{m,\rr} \neq0}\ssum_{j=1,2,\rr=\pm}\lambda_{m,\rr}^{-1}\V^\rr_{k,j,h}\right)e^{2\pi il_h\cdot x},\ \ \mbox{on $\{y=1\}$}.
    \end{cases}\]
    For any smooth function $\vartheta$, by integrating by parts, it follows
    \begin{align*}
        \langle \SSSSS_{\theta^N},\vartheta\rangle&=2\sqrt{2}\kappa i\ssum_{k\neq 0,j=1,2}\ssum_{\rr=\pm}\langle e^{2\pi il_h\cdot x}\sin(\lambda_{m,\rr} y),\V^\rr_{k,j,h}\cdot\vartheta\rangle\\&=\int_{\T^2}\SSSSS_{\theta^N,\bd}\cdot \vartheta|_{y=0}\ddd x-\int_{\T^2}\SSSSS_{\theta^N,\bd}\cdot \vartheta|_{y=1}\ddd x\\&\qquad+2\sqrt{2}\kappa i\ssum_{k\neq 0,\lambda_{m,\rr} \neq0}\ssum_{j=1,2,\rr=\pm}\lambda^{-1}_{m,\rr}\langle e^{2\pi il_h\cdot x}\cos(\lambda_{m,\rr} y),\V^\rr_{k,j,h}\cdot\partial_y\vartheta\rangle\\&=\langle\FT_h \SSSSS_{\theta^N,\bd}, \vartheta\rangle\\&\qquad+2\sqrt{2}\kappa i\ssum_{k\neq 0,\lambda_{m,\rr} \neq0}\ssum_{j=1,2,\rr=\pm}\lambda^{-1}_{m,\rr}\langle e^{2\pi il_h\cdot x}\cos(\lambda_{m,\rr} y),\V^\rr_{k,j,h}\cdot\partial_y\vartheta\rangle,
    \end{align*}
    which implies 
    \begin{align*}\|\SSSSS_{\theta^N}-\FT_h \SSSSS_{\theta^N,\bd}\|_{H^{-s}}\le \|\tilde\RR\|_{H^{1-s}},
    \end{align*}
    where 
    \[\tilde\RR=2\sqrt{2}\kappa i\ssum_{k\neq 0,\lambda_{m,\rr} \neq0}\ssum_{j=1,2,\rr=\pm}\lambda^{-1}_{m,\rr}e^{2\pi il_h\cdot x}\cos(\lambda_{m,\rr} y)\V^\rr_{k,j,h}.\]
    The remainder $\tilde\RR$ is estimated in the same way as in the vertical case, which gives
    \begin{align*}
        \|\tilde\RR\|_{H^{1-s}}\lesssim_\kappa N^{\frac{1}{2}-s},
    \end{align*}
    and thus
    \begin{align}\label{SL31-2}\|\SSSSS_{\theta^N}-\FT_h \SSSSS_{\theta^N,\bd}\|_{H^{-s}}\lesssim_\kappa N^{\frac{1}{2}-s}\to 0,
    \end{align}
    as $N\to\infty$. Therefore, it suffices to compute the scaling limit of the coefficient in $\SSSSS_{\theta^N,\bd}$
    \[\tilde\SSSSS_{\theta^N}:=2\sqrt{2}\kappa i\ssum_{k\neq 0,\lambda_{m,\rr} \neq0}\ssum_{j=1,2,\rr=\pm}\lambda_{m,\rr}^{-1}\V^\rr_{k,j,h}.\]
    By applying \eqref{INTRO1-2}, \eqref{SL14}, \eqref{SL17-2}, \eqref{SL20-1}, and \eqref{SL24-1}, one has 
    \begin{align*}
        \tilde\SSSSS_{\theta^N}&=2\sqrt{2}\kappa i\ssum_{k\neq 0,\lambda_{m,\rr} \neq0}\ssum_{j=1,2,\rr=\pm}(\theta_k^N)^2\frac{2\pi a_{k,j,h}\cdot l_h+a_{k,j,3}\lambda_{m,\rr}}{2\pi \lambda_{m,\rr}}(\PPPP_{l^\rr-k^-}^\perp \A_{k,j}^{\rr,-})_h\notag\\&=-2\pi i\kappa\ssum_{k\neq 0}\ssum_{j=1,2,\rr=\pm}(\theta_k^N)^2a_{k,j,3}(a_{k,j}^-\cdot l^\rr )(\fd_{\rr}\PPPP_{k^-}^\perp b^{\rr})_h+o(1).
    \end{align*}
     The above reduction can be justified rigorously by splitting $\SSSSS_{\theta^N}$ into the regions $|m|\le \varepsilon N$ and $|m|>\varepsilon N$. Indeed, the former contributes $O(\varepsilon)$, while an application of the above reduction on the latter gives an error term satisfying $O_\varepsilon(N^{-1})$. Hence, by first letting $N\to\infty$ and then $\varepsilon\to0$, the total error term vanishes. The restriction $\lambda_{m,\rr}\ne0$ is removed, since the condition $\{\lambda_{m,\rr}=0\}=\{m=-\rr n/2\}$ gives at most one horizontal layer, which contains $O(N^2)$ lattice points in the shell and contributes at most $O(N^{-1})$.

     Now we compute the first term on the right-hand side of the above identity. By applying \eqref{INTRO2}, it follows
    \begin{align*}
        \ssum_{j=1,2}a_{k,j,3} (a_{k,j}^-\cdot l^\rr)=\ssum_{j=1,2}a_{k,j,3} (a_{k,j}\cdot Jl^\rr)=\row_3\PPPP_k^\perp\cdot Jl^\rr.
    \end{align*}
    Therefore, by setting $\hk:=\frac{k}{|k|}$, one has
    \begin{align*}
        \tilde\SSSSS_{\theta^N}&=-2\pi i\kappa\ssum_{k\neq 0}\ssum_{\rr=\pm}(\row_3\PPPP_k^\perp\cdot Jl^\rr)(\theta_k^N)^2(\fd_{\rr}\PPPP_{k^-}^\perp b^{\rr})_h+o(1)\\
        &=2\pi i\kappa\ssum_{k\neq0}\ssum_{\rr=\pm}(\theta_k^N)^2\left(\hk_3\hk_h\cdot l_h+\frac{\rr n}{2}(1-\hk_3^2)\right)\\&\qquad\qquad\qquad\qquad\qquad\qquad\times\left(\rr\big(b_h-(\hk_h\cdot b_h)\hk_h\big)+ \hk_3b_3\hk_h\right)+o(1)\\&=4\pi i\kappa\ssum_{k\neq0}(\theta_k^N)^2\left(\hk_3^2b_3(\hk_h\otimes \hk_h) l_h+\frac{n}{2}(1-\hk_3^2)(I_2-\hk_h\otimes\hk_h)b_h \right)+o(1).
    \end{align*}
    Proceeding as in Step I\!V, it follows that
    \begin{align*}
        \lim_{N\to\infty}\tilde\SSSSS_{\theta^N}&=i\kappa\left(b_3l_h\int_{|\zeta|=1}\zeta_3^2\zeta_h\otimes\zeta_h\ddd S+\frac{nb_h}{2}\int_{|\zeta|=1}(1-\zeta_3^2)(I_2-\zeta_h\otimes\zeta_h)\ddd S\right)\\&=\frac{4\pi i\kappa}{15}(3nb_h+b_3l_h),
    \end{align*}
    which, together with the definition \eqref{SL6-2} of $\varphi$, implies 
    \[\lim_{N\to\infty}\SSSSS_{\theta^N,\bd}=\frac{4\kappa}{5}\partial_y\varphi_h|_{\partial D}+\frac{2\kappa}{15}\nabla_x\varphi_3|_{\partial D}\]
    in $L^2$, and thus 
    \begin{align}
        \label{SL31-3}\lim_{N\to\infty}\FT_h\SSSSS_{\theta^N,\bd}=\FT_h\left(\frac{4\kappa}{5}\partial_y\varphi_h|_{\partial D}+\frac{2\kappa}{15}\nabla_x\varphi_3|_{\partial D}\right)
    \end{align}
    in $H^{-s}$. To summarize, by \eqref{SL25}, \eqref{SL25-1}, \eqref{SL29}, \eqref{SL31-1}--\eqref{SL31-3}, one has
    \begin{align}
        \label{SL31-4}\lim_{N\to\infty}S_{\theta^N,\bulk,2}\varphi=\FT\left(\frac{4\kappa}{5}\partial_y\varphi_h|_{\partial D}+\frac{2\kappa}{15}\nabla_x\varphi_3|_{\partial D}\right)
    \end{align}
    in $H^{-s}$.

    \noindent\textbf{Step V\!I\!I.} Combining \eqref{SL17-1}, \eqref{SL17}, \eqref{SL18}, \eqref{SL24}, and \eqref{SL31-4}, one has the convergence 
    \begin{align}
            \label{SL31-6}
            \lim_{N\to\infty}\left\|S_{\theta^N}\varphi-\frac{4\kappa}{5}\Delta\varphi-\FT\left(\frac{4\kappa}{5}\partial_y\varphi_h|_{\partial D}+\frac{2\kappa}{15}\nabla_x\varphi_3|_{\partial D}\right)\right\|_{H^{-s}}=0,
        \end{align}
    for a single mode $\varphi$. The extension to general smooth functions follows from a standard argument as in \cite{FL21}. Therefore, we give a brief sketch and omit the details. Since $\varphi_h\in C^\infty_c(D)$, one has $\partial_y\varphi_3\in C^\infty_c(D)$, which implies that any such smooth function $\varphi$ has the following expansion
    \[\varphi=\ssum_{l}c_l\varphi_l.\]
    Here, $\varphi_l$ is the single mode given in \eqref{SL6-2}, and the coefficient $c_l$ decays rapidly. For any integer $M\ge 1$, define 
    \[\varphi^{<M}:=\ssum_{|l|< M}c_l\varphi_l,\qquad \varphi^{\ge M}:=\ssum_{|l|\ge M}c_l\varphi_l.\]
    It is clear that \eqref{SL31-6} holds for $\varphi^{<M}$. As for $\varphi^{\ge M}$, notice that from the single-mode estimates established in Steps I--V\!I, one has 
    \[\sup_{N\ge 1}\|S_{\theta^N}\varphi_l\|_{H^{-s}}\le C_l,\]
    where $C_l$ has a polynomial growth in $l$. Hence, by choosing $M$ sufficiently large and using the rapid decay of $c_l$, it follows
    \[\sup_{N\ge 1}\left\|S_{\theta^N}\varphi^{\ge M}\right\|_{H^{-s}}\le \ssum_{|l|\ge M} |c_l|C_l\le \epsilon,\]
    and 
    \[\left\|\frac{4\kappa}{5}\Delta\varphi^{\ge M}+\FT\left(\frac{4\kappa}{5}\partial_y\varphi^{\ge M}_h|_{\partial D}+\frac{2\kappa}{15}\nabla_x\varphi_3^{\ge M}|_{\partial D}\right)\right\|_{H^{-s}}\le \epsilon.\]
    This, together with the convergence of $\varphi^{<M}$, implies \eqref{SL31-6} in the general case. Since $\varphi_h\in C_c^\infty(D)$, we obtain \eqref{SL6-1}.
    \end{proof}
       
    \subsubsection{The degenerate case}
    Next, we consider the degenerate case. The main result is given below.
    \begin{prop}\label{propSL2-1}
        Let $\{\theta^N_\cdot\}_{N\ge 1}$ be defined in \eqref{SL2}. For any $\varphi\in H_{\curl}\cap C^\infty$,
        \begin{align}
            \label{SL34}
            \lim_{N\to\infty}\left\|S_{\theta^N}\varphi-\kappa\FA_{\eff}\varphi\right\|=0,
        \end{align}
        where $S_{\theta}$ is defined in \eqref{INTRO9}$_2$ and 
        \begin{align}
            \label{SL35}
            \FA_{\eff}\varphi:=\Pi\left(\frac{1}{4}\Delta_x \varphi_h+\frac{1}{2}\nabla_x\dive_h\varphi_h,\Delta_x\varphi_3\right).
        \end{align}
    \end{prop}
    \begin{proof} Notice that 
    \begin{align*}
        S_\theta\varphi=2\kappa \ssum_{k\neq 0}\theta_k^2
    \Pi(\sigma_k\cdot\nabla_x)(\sigma_{-k}\cdot \nabla_x)\varphi-2\kappa \ssum_{k\neq 0}\theta_k^2
    \Pi(\sigma_k\cdot\nabla_x)\Pi^\perp(\sigma_{-k}\cdot \nabla_x)\varphi,
    \end{align*}
    where, by using the definition \eqref{INTRO5} of $\sigma_k$ and \eqref{INTRO4-1}, it follows 
    \begin{align*}
        \ssum_{k\neq 0}\theta_k^2\Pi(\sigma_k\cdot\nabla_x)(\sigma_{-k}\cdot \nabla_x)\varphi&=\Pi\ssum_{k\neq 0}\theta_k^2\frac{(k^\perp\cdot\nabla_x)^2\varphi}{|k|^2}\\&=\Pi\ssum_{k\neq 0}\theta_k^2\frac{1}{|k|^2}(k_2^2\partial_{x_1}^2\varphi+k_1^2\partial_{x_2}^2\varphi-2k_1k_2\partial_{x_1}\partial_{x_2}\varphi)\\&=\frac{1}{2}\Pi\Delta_x\varphi=\frac{1}{2}\Delta_x\varphi.
    \end{align*}
    This implies
    \begin{align}
        \label{SL38}S_\theta\varphi=\kappa\Delta_x\varphi-S_\theta^\perp\varphi,
    \end{align}
    where
    \[S_\theta^\perp\varphi:=2\kappa \ssum_{k\neq 0}\theta_k^2
    \Pi(\sigma_k\cdot\nabla_x)\Pi^\perp(\sigma_{-k}\cdot \nabla_x)\varphi.\]
    Therefore, it remains to compute the scaling limit of $S_\theta^\perp\varphi$. As in the proof of Proposition \ref{propSL2}, the general case can be reduced to a single tangential Fourier mode of the form
    \begin{align}\label{SL36}
        \varphi(x,y)=e^{2\pi i l\cdot x}b(y)\in H_{\curl},
    \end{align}
    where $b=(b_h,b_3)^T=(b_1,b_2,b_3)^T$ is a smooth vector field satisfying 
    \begin{align}\label{SL37}
    2\pi il\cdot b_h(y)+b_3'(y)=0.
    \end{align}
    Moreover, if $l=0$, then $b_3\equiv 0$, due to the horizontal zero-mean condition for vertical component in \eqref{INTRO1-1}. Hence, $\varphi$ is a shear flow. In this case, $S_{\theta^N}\varphi=\FA_{\eff}\varphi=0.$ Thus, we assume that $l\neq 0$. The rest of the proof is divided into the following steps.

    \noindent\textbf{Step I.} Define 
    \[\fd_k=\begin{cases}
        1,\qquad k\in\Z_{0,+},\\
        -1,\qquad k\in \Z_{0,-}.
    \end{cases}\]
    Then, one has
    \[(\sigma_{-k}\cdot \nabla_x)\varphi=-2\pi i\fd_{-k}\frac{k^\perp\cdot l}{|k|}e^{2\pi i(l-k)\cdot x}b(y),\]
    and by using the divergence free condition \eqref{SL37}, it follows
    \begin{align*}
        \dive( (\sigma_{-k}\cdot \nabla_x)\varphi)&=-2\pi i\fd_{-k}\frac{k^\perp\cdot l}{|k|}e^{2\pi i(l-k)\cdot x}(2\pi i(l-k)\cdot b_h+b_3')\\&=-4\pi^2\fd_{-k}\frac{k^\perp\cdot l}{|k|}e^{2\pi i(l-k)\cdot x}(k\cdot b_h).
    \end{align*}
    Together with Proposition \ref{propA1}, this implies
    \[\Pi^\perp(\sigma_{-k}\cdot \nabla_x)\varphi=\nabla (\fd_{k}\phi_{k}(y)e^{2\pi i(l-k)\cdot x}),\]
    where $\phi_{k}$ solves the Dirichlet--Poisson problem
    \begin{align}\label{SL42}
        \begin{cases}
            \phi_{k}'' -\lambda^2_{k}\phi_{k}=4\pi^2\frac{k^\perp\cdot l}{|k|}(k\cdot b_h(y))=:f_{k}(y),\\
            \phi_{k}(0)=\phi_{k}(1)=0,
        \end{cases}
    \end{align}
    and $\lambda_{k}:=2\pi|l-k|$. Therefore,
    \begin{align*}
        S^\perp_\theta \varphi=4\pi \kappa i \Pi\ssum_{k\neq 0}\theta_k^2 \frac{k^\perp\cdot l}{|k|} e^{2\pi i l\cdot x}(2\pi i (l_1-k_1)\phi_{k},2\pi i (l_2-k_2)\phi_{k},\phi_{k}').
    \end{align*}
    To study the asymptotic behavior of $\phi_k$ as $|k|\to \infty$, we introduce the decomposition 
    \begin{align}
        \label{SL43}\phi_k=\phi_{\bulk,k}+\phi_{\bd,k}+\phi_{\rem,k},
    \end{align}
    where
    \begin{align}\label{SL39}
    \phi_{\bulk,k}:=-\lambda_k^{-2}f_{k},
    \end{align}
    and $\phi_{\bd,k},\phi_{\rem,k}$ satisfy the following problems
   \begin{align}\label{SL40}
    \begin{cases}
        \phi_{\bd,k}'' -\lambda_k^2\phi_{\bd,k}=0,\\
        \phi_{\bd,k}(0)=\lambda_k^{-2}f_{k}(0),\\
        \phi_{\bd,k}(1)=\lambda_k^{-2}f_{k}(1),
    \end{cases}
    \end{align}
    and 
    \begin{align}\label{SL41}
    \begin{cases}
        \phi_{\rem,k}'' -\lambda_k^2\phi_{\rem,k}=\lambda_k^{-2}f_{k}'',\\
        \phi_{\rem,k}(0)= \phi_{\rem,k}(1)=0.
    \end{cases}
    \end{align}
    Compared with the decomposition \eqref{SL8-2} in the non-degenerate case, the decomposition \eqref{SL43} contains an additional remainder term. The reason is that the source term in \eqref{SL42} is not a full trigonometric mode in all variables. Consequently, the one-dimensional elliptic problem does not admit a simple explicit particular solution. We therefore use an approximate bulk term and introduce a remainder, which accounts for the error made by replacing the full one-dimensional resolvent with its leading-order high-frequency approximation.

    \noindent\textbf{Step I\!I.} Define
    \begin{align*}
        S^\perp_{\theta,\bulk} \varphi\!:=4\pi \kappa i \Pi\ssum_{k\neq 0}\theta_k^2 \frac{k^\perp\cdot l}{|k|} e^{2\pi i l\cdot x}(2\pi i (l_1-k_1)\phi_{\bulk,k},2\pi i (l_2-k_2)\phi_{\bulk,k},\phi_{\bulk,k}'),
    \end{align*}
    and
    \begin{align*}
        S^\perp_{\theta,\bd} \varphi\!:=4\pi \kappa i \Pi\ssum_{k\neq 0}\theta_k^2 \frac{k^\perp\cdot l}{|k|} e^{2\pi i l\cdot x}(2\pi i (l_1-k_1)\phi_{\bd,k},2\pi i (l_2-k_2)\phi_{\bd,k},\phi_{\bd,k}'),
    \end{align*}
    and
    \begin{align*}
        S^\perp_{\theta,\rem} \varphi\!:=4\pi \kappa i \Pi\ssum_{k\neq 0}\theta_k^2 \frac{k^\perp\cdot l}{|k|} e^{2\pi i l\cdot x}(2\pi i (l_1-k_1)\phi_{\rem,k},2\pi i (l_2-k_2)\phi_{\rem,k},\phi_{\rem,k}').
    \end{align*}
    In this step, we prove that 
    \begin{align}
        \label{SL44}
        \lim_{N\to\infty}\|S^\perp_{\theta^N,\bd}\varphi +S^\perp_{\theta^N,\rem}\varphi\|=0.
    \end{align}
    Indeed, notice that $\phi_{\bd,k}$ has an explicit solution formula
    \[\phi_{\bd,k}(y)=\lambda_k^{-2}f_{k}(0)\frac{\sinh (\lambda_k(1-y))}{\sinh \lambda_k}+\lambda_k^{-2}f_{k}(1)\frac{\sinh (\lambda_k y)}{\sinh \lambda_k},\]
    which, together with the estimate
    \[\left|\frac{\sinh (\lambda_k(1-y))}{\sinh \lambda_k}\right|\lesssim e^{-\lambda_k y},\qquad \left|\frac{\sinh (\lambda_k y)}{\sinh \lambda_k}\right|\lesssim e^{-\lambda_k(1-y)},\]
    yields
    \begin{align}
        \label{SL45}
        \|\phi_{\bd,k}\|\lesssim \frac{|k||l|\|b_h\|_{H^1_y}}{\lambda_k^{5/2}},\qquad \|\phi_{\bd,k}'\|\lesssim \frac{|k||l|\|b_h\|_{H^1_y}}{\lambda_k^{3/2}}.
    \end{align}
    On the other hand, the remainder term $\phi_{\rem,k}$ can be bounded through a direct energy estimate
    \[\|\phi_{\rem,k}'\|^2+\lambda_k^2\|\phi_{\rem,k}\|^2\le \lambda_k^{-2}\|f''_{k,l}\|\|\phi_{\rem,k}\|,\]
    which implies 
    \begin{align}
        \label{SL46}
        \|\phi_{\rem,k}\|\lesssim \frac{|k||l|\|b_h\|_{H^2_y}}{\lambda_k^4},\qquad \|\phi_{\rem,k}'\|\lesssim \frac{|k||l|\|b_h\|_{H^2_y}}{\lambda_k^3}.
    \end{align}
    This, together with \eqref{SL45}, leads to
    \begin{align*}
       \|S^\perp_{\theta^N,\bd}\varphi \|&\lesssim \kappa\ssum_{k\neq 0}(\theta^N_k)^2|l|\left(|l-k|\|\phi_{\bd,k}\|+\|\phi_{\bd,k}'\|\right)\\&\lesssim \kappa|l|^2\|b_h\|_{H^1_y}\ssum_{k\neq 0}(\theta^N_k)^2|k|\lambda_k^{-3/2}\lesssim \left(\kappa|l|^2\|b_h\|_{H^1_y}\right) N^{-1 /2}\to0,
    \end{align*}
    and 
    \begin{align*}
       \|S^\perp_{\theta^N,\rem}\varphi \|&\lesssim \kappa\ssum_{k\neq 0}(\theta^N_k)^2|l|\left(|l-k|\|\phi_{\rem,k}\|+\|\phi_{\rem,k}'\|\right)\\&\lesssim \kappa|l|^2\|b_h\|_{H^2_y}\ssum_{k\neq 0}(\theta^N_k)^2|k|\lambda_k^{-3}\lesssim \left(\kappa|l|^2\|b_h\|_{H^2_y}\right) N^{-2}\to0, 
    \end{align*}
    as $N\to\infty$.

    \noindent\textbf{Step I\!I\!I.} Now we turn to $S^\perp_{\theta^N,\bulk}\varphi$. Since $\Pi$ is a continuous linear operator, it suffices to study
    \[\FS_{\theta^N}:=4\pi \kappa i \ssum_{k\neq 0}\theta_k^2 \frac{k^\perp\cdot l}{|k|} e^{2\pi i l\cdot x}(2\pi i (l_1-k_1)\phi_{\bulk,k},2\pi i (l_2-k_2)\phi_{\bulk,k},\phi_{\bulk,k}').\]
    For the vertical component, it follows
    \begin{align}
        \notag\label{SL47}\|\FS_{\theta^N,3} \|_{L^2}&\lesssim \kappa\ssum_{k\neq 0}(\theta^N_k)^2|l|\|\phi_{\bulk,k}'\|_{L^2_y}\lesssim \kappa\ssum_{k\neq 0}(\theta^N_k)^2|l|\lambda_k^{-2}\|f_{k}'\|\\&\lesssim \kappa\ssum_{k\neq 0}(\theta^N_k)^2|l|^2\lambda_k^{-2}|k|\|b_h\|_{H^1_y}\lesssim\left(\kappa|l|^2\|b_h\|_{H^1_y}\right)N^{-1}\to0.
    \end{align}
    As for the horizontal components, one has 
    \begin{align*} 
    \FS_{\theta^N,h}&=-8\pi^2 \kappa  e^{2\pi i l\cdot x}\ssum_{k\neq 0}(\theta_k^N)^2 \left(\frac{k^\perp\cdot l}{|k|}\right) (l-k)\phi_{\bulk,k}(y)\\&=8\pi^2 \kappa  e^{2\pi i l\cdot x}\ssum_{k\neq 0}(\theta_k^N)^2 \left(\frac{k^\perp\cdot l}{|k|}\right)(l-k) \lambda_k^{-2}\left( 4\pi^2\frac{k^\perp \cdot l}{|k|}k\cdot b_h(y)\right)\\&=8\pi^2 \kappa  e^{2\pi i l\cdot x}\left(\ssum_{k\neq 0}(\theta^N_k)^2 \left(\frac{k^\perp\cdot l}{|k|}\right)^2 \frac{(l-k)\otimes k}{|l-k|^2}\right)b_h(y).\end{align*}
    Thus, it remains to justify the convergence of the coefficient matrix
    \[\MM_{\theta^N}:=\ssum_{k\neq 0}(\theta^N_k)^2 \left(\frac{k^\perp\cdot l}{|k|}\right)^2 \frac{(l-k)\otimes k}{|l-k|^2}.\]
    Since this can be done in the same way as in Step~I\!V of the proof of Proposition~\ref{propSL2}, we only give a brief sketch. By using a two-dimensional version of the estimate \eqref{SL20-1}
    \[\left|\frac{(l-k)\otimes (l-k)}{|l-k|^2}-\frac{k\otimes k}{|k|^2}\right|\lesssim \frac{|l|}{|k|},\]
    and repeating the derivation of \eqref{SL22}, one gets
    \begin{align*}
        \MM_{\theta^N}=-\ssum_{k\neq 0}(\theta^N_k)^2 \left(\frac{k^\perp\cdot l}{|k|}\right)^2 \frac{k\otimes k}{|k|^2}+o(1)\to -\frac{1}{2\pi}\int_{|\xi|=1}\left(\frac{\xi^\perp\cdot l}{|\xi|}\right)^2\frac{\xi\otimes \xi}{|\xi|^2}\ddd S.
    \end{align*}
    Let $Q$ be the two-dimensional orthogonal matrix such that $Q^T\frac{l}{|l|}=e_1$. By applying the coordinate transform $\xi\to Q\xi$, it follows
    \begin{align*}
        \int_{|\xi|=1}\left(\frac{\xi^\perp\cdot l}{|\xi|}\right)^2\frac{\xi\otimes \xi}{|\xi|^2}\ddd S&=|l|^2Q\left(\int_{|\xi|=1}\left(\frac{\xi^\perp\cdot e_1}{|\xi|}\right)^2\frac{\xi\otimes \xi}{|\xi|^2}\ddd S\right)Q^T\\&=\pi|l|^2 Q\diag\left\{\frac{1}{4},\frac{3}{4}\right\}Q^T,
    \end{align*}
    which implies that
    \begin{align*}
        \lim_{N\to\infty}\FS_{\theta^N,h}&=-4\pi^2|l|^2 \kappa \left( Q\diag\left\{\frac{1}{4},\frac{3}{4}\right\}Q^T\right)e^{2\pi i l\cdot x} b_h(y)\\&=-4\pi^2|l|^2 \kappa \left( \frac{3}{4}I_2-\frac{1}{2}\frac{l\otimes l}{|l|^2}\right)e^{2\pi i l\cdot x} b_h(y)=\frac{3\kappa}{4}\Delta_x \varphi_h-\frac{\kappa}{2}\nabla_x\dive_h\varphi_h
    \end{align*}
    and 
    \begin{align*}
        \lim_{N\to\infty}S^\perp_{\theta^N,\bulk}\varphi=\kappa\Pi\left(\frac{3}{4}\Delta_x \varphi_h-\frac{1}{2}\nabla_x\dive_h\varphi_h,0\right)
    \end{align*}
    in $L^2$. Combining this with \eqref{SL38}, \eqref{SL44}, one obtains \eqref{SL34}.
    \end{proof}
    \subsection{Scaling limit of the cut-off solutions}
    Let $w^N_{\cut}$ be the solution of the cut-off equation \eqref{SL1}, corresponding to the noise coefficients $\theta^N_\cdot$ given in \eqref{SL2}. The goal of this subsection is to justify the limit of $w^N_{\cut}$, as $N\to\infty$. We shall mainly focus on the non-degenerate case. Since the degenerate case is obtained by the same tightness argument, combined with its own corrector limit, we only state the result of that case at the end.
    \begin{prop}\label{propSL3}
        Let $\alpha,\nu,\kappa,R>0$, and 
        \[w_{\ini}\in L^\infty(\Omega;H_{\curl}\cap H^1)\]
        be $\FFF_0$-measurable. Suppose that $W$ is the non-degenerate random field defined in \eqref{INTRO4}. Then, there exists $\delta_0$ depending on $\alpha,\nu,\kappa$ such that for any $T>0$, $\epsilon\in(0,1)$, and $\delta<\delta_0$,
        \begin{align}
            \label{LE00}\lim_{N\to\infty}\PP\left\{\sup_{t\le T}\|w_{\cut}^N(t)-w_{\cut}^{\det}(t)\|_{H^{\delta}}>\epsilon\right\}=0,
        \end{align}
        where $w_{\cut}^{\det}$ satisfies the following NS-type equation with cut-off and boundary feedback
    \begin{align}
    \label{LE6}
    \begin{cases}
    \partial_tw+\eta_R(\|w\|)[(u\cdot \nabla)w-(w\cdot\nabla)u]=\left(\nu+\frac{4\kappa}{5}\right)\Delta w+\frac{14\alpha\kappa}{15}\nabla\HH[n_3w_{3}],\\u=\curl^{-1}w,\\(\partial_y w_3-\alpha w_3)|_{y=0}=(w_h-\alpha 
    u_h^\perp)|_{y=0}=0,\\(\partial_y w_3+\alpha w_3)|_{y=1}=(w_h+\alpha 
    u_h^\perp)|_{y=1}=0,\\
    w|_{t=0}=w_{\ini}.
    \end{cases}
    \end{align} 
    Here, $n=(n_1,n_2,n_3)$ denotes the unit outward normal vector, and $\HH$ is the harmonic extension operator of a given boundary data, that is, for any boundary data $h$, $\HH[h]$ satisfies
        \begin{align}\label{SL6-1-1}
            \begin{cases}
                \Delta\HH[h]=0,\\
                \HH[h]|_{y=0,1}=h|_{y=0,1}.
            \end{cases}
        \end{align}
    \end{prop}
    \begin{proof}
        \textbf{Step I.} For simplicity, we assume that $w_{\ini}$ is deterministic. The case of general initial data follows from the Gy\"ongy--Krylov argument. Let $\varphi$ be a test function satisfying \eqref{INTRO4-3}. Notice that $\varphi_h,\partial_y\varphi_3\in C_c^\infty(D)$ and $u_3|_{y=0,1}=0$, one has
        \begin{align*}
            \langle \nabla w,\nabla \varphi\rangle=-\langle w,\Delta \varphi\rangle+\int_{\T^2}\left(w_3\partial_y\varphi_3|_{y=1}-w_3\partial_y\varphi_3|_{y=0}\right)\ddd x=-\langle w,\Delta \varphi\rangle,
        \end{align*}
        and 
        \begin{align*}
            \langle (u\cdot \nabla) w-(w\cdot \nabla )u,\varphi\rangle&=-\langle w, (u\cdot \nabla)\varphi \rangle+\langle u, (w\cdot \nabla)\varphi \rangle\\&\qquad+\int_{\T^2}\left(w_3(u\cdot\varphi)|_{y=1}-w_3(u\cdot\varphi)|_{y=0}\right)\ddd x\\&=\langle u\otimes w-w\otimes u,\nabla \varphi \rangle.
        \end{align*}
        Therefore, from \eqref{INTRO4-4}, the weak formulation of \eqref{SL1} is given by
        \begin{align}
        \label{LE1}
        \langle& w_{\cut}^N(t),\varphi\rangle-\langle w_{\ini},\varphi\rangle+\int_0^{t} \left(-\nu\langle w_{\cut}^N,\Delta \varphi\rangle+\nu\alpha\langle w_{\cut,3}^N,\varphi_3\rangle_{L^2(\partial D)}\right)\ddd s\notag\\&\qquad+\int_0^t\eta_R(\|w_{\cut}^N\|)\langle u_{\cut}^N\otimes w_{\cut}^N-w_{\cut}^N\otimes u_{\cut}^N,\nabla \varphi \rangle\ddd s\notag\\& \quad =\int_0^t\langle w_{\cut}^N,S_{\theta^N}\varphi\rangle \ddd s +\sqrt{2\kappa}\ssum_{k\neq 0}\ssum_{j=1,2}\int_0^{t}\theta^N_k\langle \Pi(\sigma_{k,j}\cdot\nabla)w_{\cut}^N,\varphi\rangle \ddd B_{t}^{k,j}
    \end{align}
    $\PP$-almost surely for any $t\in [0,T]$. Applying Proposition \ref{propSL1} with sufficiently small $\delta>0$ combined with the Prokhorov theorem and the Skorokhod theorem, passing to a subsequence and changing the underlying probability space if necessary, it follows 
    \begin{align}\label{LE1-1}
        w_{\cut}^N\to w_{\cut}^{\det}\qquad \mbox{$\PP$-almost surely in $\chi_{\delta,T}$}
    \end{align}
    for some process $w_{\cut}^{\det}$. This implies 
    \begin{align}\label{LE2}
        \langle w_{\cut}^N(t),\varphi&\rangle+\int_0^{t}  \left(-\nu\langle w_{\cut}^N,\Delta \varphi\rangle+\nu\alpha\langle w_{\cut,3}^N,\varphi_3\rangle_{L^2(\partial D)}\right)\ddd s\notag\\&\to \langle w_{\cut}^{\det}(t),\varphi\rangle+\int_0^{t}  \left(-\nu\langle w_{\cut}^{\det},\Delta \varphi\rangle+\nu\alpha\langle w_{\cut,3}^{\det},\varphi_3\rangle_{L^2(\partial D)}\right)\ddd s
    \end{align}
    and 
    \begin{align}\label{LE3}
        \int_0^t\eta_R(\|w_{\cut}^N\|)&\langle u_{\cut}^N\otimes w_{\cut}^N-w_{\cut}^N\otimes u_{\cut}^N,\nabla \varphi \rangle\ddd s\notag\\&\to \int_0^t\eta_R(\|w_{\cut}^{\det}\|)\langle u_{\cut}^{\det}\otimes w_{\cut}^{\det}-w_{\cut}^{\det}\otimes u_{\cut}^{\det},\nabla \varphi \rangle\ddd s
    \end{align}
    $\PP$-almost surely in $L^\infty(0,T)$.

    \noindent\textbf{Step I\!I.} It remains to deal with the It\^o--Stratonovich corrector and the martingale term. Indeed, by using \eqref{LE1-1} with sufficiently small $\delta$ and applying Proposition \ref{propSL2}, one has 
    \begin{align}
        \label{LE4}
        \int_0^t\langle w^N_{\cut},S_{\theta^N}\varphi\rangle \ddd s \to \frac{4\kappa}{5}\int_0^t\langle w_{\cut}^{\det},\Delta\varphi\rangle \ddd s -\frac{2\kappa}{15}\int_0^t\int_{\T^2}\nabla_x\varphi_3\cdot w_{\cut,h}^{\det}|^1_0\ddd x\ddd s
    \end{align}
    $\PP$-almost surely in $L^\infty(0,T)$. Since the tangential Navier boundary conditions in \eqref{SL1} are preserved in the above convergence, by integrating by parts, it follows
    \begin{align}
        \label{LE4-1-1}-\frac{2\kappa}{15}\int_{\T^2}\nabla_x\varphi_3\cdot w_{\cut,h}^{\det}|^1_0\ddd x=\frac{2\alpha\kappa}{15}\langle w_{\cut,3}^{\det},\varphi_3\rangle_{L^2(\partial D)}.
    \end{align}
    We turn to the martingale term. Notice that $\{\sigma_{k,j}\}_{k,j}$ forms an orthonormal system in $L^2$. By the Burkholder--Davis--Gundy inequality, the Bessel inequality, \eqref{SL14}, and \eqref{SL6} with $s=0, p=2$, it follows that 
    \begin{align}\label{LE5}
        \sqrt{2\kappa}\E\sup_{t\le T}&\left|\ssum_{k\neq 0}\ssum_{j=1,2}\int_0^{t}\theta^N_k\langle \Pi(\sigma_{k,j}\cdot\nabla)w_{\cut}^N,\varphi\rangle \ddd B_{t}^{k,j}\right|\notag\\&\qquad\qquad\lesssim_\kappa \E\left(\ssum_{k\neq 0}\ssum_{j=1,2}\int_0^T(\theta^N_k)^2|\langle \nabla\varphi w_{\cut}^N,\sigma_{k,j}\rangle|^2 \ddd t\right)^{\frac{1}{2}}\notag\\&\qquad\qquad\lesssim_{\kappa} N^{-\frac{3}{2}}\E\left(\int_0^T\| \nabla\varphi w_{\cut}^N\|^2 \ddd t\right)^{\frac{1}{2}}\notag\\&\qquad\qquad\lesssim_{\alpha,\nu,\kappa,T,R} N^{-\frac{3}{2}}\|\nabla\varphi\|_\infty (1+\|w_{\ini}\|_{H^1})\to 0,
    \end{align}
    as $N\to\infty$. 

    \noindent\textbf{Step I\!I\!I.} Combining \eqref{LE2}--\eqref{LE5}, it follows that any convergent subsequence of $\DDD (w_{\cut}^N)$ converges weakly to $\delta_{w_{\cut}^{\det}}$, where $w_{\cut}^{\det}$ satisfies the weak formulation
    \begin{align*}
        &\langle w_{\cut}^{\det}(t),\varphi\rangle-\langle w_{\ini},\varphi\rangle+\int_0^{t} \left(\left(\nu+\frac{4\kappa}{5}\right)\langle \nabla w_{\cut}^{\det},\nabla \varphi\rangle+\nu\alpha\langle w_{\cut,3}^{\det},\varphi_3\rangle_{L^2(\partial D)}\right)\ddd s\notag\\&-\int_0^t \frac{2\alpha\kappa}{15}\langle w_{\cut,3}^{\det},\varphi_3\rangle_{L^2(\partial D)}\ddd s+\int_0^t\eta_R(\|w_{\cut}^{\det}\|)\langle u_{\cut}^{\det}\otimes w_{\cut}^{\det}-w_{\cut}^{\det}\otimes u_{\cut}^{\det},\nabla \varphi \rangle\ddd s=0
    \end{align*}
    for any $t\in [0,T]$ and test function $\varphi$ satisfying \eqref{INTRO4-3}. Notice that 
    \begin{align*}
        \left(\nu+\frac{4\kappa}{5}\right)&\langle \nabla w_{\cut}^{\det},\nabla \varphi\rangle+\left(\nu\alpha-\frac{2\alpha\kappa}{15}\right)\langle w_{\cut,3}^{\det},\varphi_3\rangle_{L^2(\partial D)}\\&\!=\left(\!\nu+\frac{4\kappa}{5}\right)\!\!\left(\langle \nabla w_{\cut}^{\det},\nabla \varphi\rangle+\alpha\langle w_{\cut,3}^{\det},\varphi_3\rangle_{L^2(\partial D)}\right)-\!\frac{14\alpha\kappa}{15}\langle w_{\cut,3}^{\det},\varphi_3\rangle_{L^2(\partial D)}\\&\!=\left(\!\nu+\frac{4\kappa}{5}\right)\!\!\left(\langle \nabla w_{\cut}^{\det},\nabla \varphi\rangle+\alpha\langle w_{\cut,3}^{\det},\varphi_3\rangle_{L^2(\partial D)}\right)-\!\frac{14\alpha\kappa}{15}\langle \nabla\HH[n_3w_{\cut,3}^{\det}],\varphi\rangle
    \end{align*}
    where $\HH$ is defined in \eqref{SL6-1-1} and $n=(n_1,n_2,n_3)$ denotes the unit outward normal vector of the boundary. Hence, we conclude that $w_{\cut}^{\det}$ is a weak solution of \eqref{LE6}. Moreover, using the decomposition~\eqref{A0}, the limiting weak formulation can be extended, by density and orthogonality, from the test functions in Definition \ref{definition1} to the mixed space $H^1_0(D)^2\times H^1(D)$. Thus, the boundary-corrected energy method used in Proposition \ref{propLS1} can be applied to the difference of two solutions. This, together with the harmonic lifting estimate
    \begin{align}
        \label{LE6-1}\|\nabla\HH[n_3w_{3}]\|\lesssim \|w_3\|_{H^1},
    \end{align}
    and the cut-off in the nonlinear term, implies the uniqueness of weak solutions of \eqref{LE6}. 
    
    Hence, all subsequential limits of $\DDD (w_{\cut}^N)$ coincide with $\delta_{w^{\det}_{\cut}}$, which ensures the convergence of the whole family and thereby yields the convergence in probability of the cut-off solutions $w_{\cut}^N$.
    \end{proof}
    The degenerate case is obtained by the same compactness argument, combined with the scaling limit of its own corrector. The corresponding result is presented below.
    \begin{prop}\label{propSL4}
        Let $\alpha,\nu,\kappa,R>0$, and
        \[w_{\ini}\in L^\infty(\Omega;H_{\curl}\cap H^1)\]
        be $\FFF_0$-measurable. Suppose that $W$ is the degenerate random field defined in \eqref{INTRO6}. Then, there exists $\delta_0$ depending on $\alpha,\nu,\kappa$ such that for any $T>0$, $\epsilon\in(0,1)$, and $\delta<\delta_0$,
        \begin{align}
            \label{LE01}\lim_{N\to\infty}\PP\left\{\sup_{t\le T}\|w_{\cut}^N(t)-w_{\cut}^{\det}(t)\|_{H^{\delta}}>\epsilon\right\}=0,
        \end{align}
        where $w_{\cut}^{\det}$ satisfies the following NS-type equation with cut-off and anisotropic nonlocal dissipation
    \begin{align}
    \label{LE7}
    \begin{cases}
    \partial_tw+\eta_R(\|w\|)[(u\cdot \nabla)w-(w\cdot\nabla)u]=\nu\Delta w+\kappa\FA_{\eff} w,\\u=\curl^{-1}w,\\(\partial_y w_3-\alpha w_3)|_{y=0}=(w_h-\alpha 
    u_h^\perp)|_{y=0}=0,\\(\partial_y w_3+\alpha w_3)|_{y=1}=(w_h+\alpha 
    u_h^\perp)|_{y=1}=0,\\
    w|_{t=0}=w_{\ini}.
    \end{cases}
    \end{align} 
    Here, $\FA_{\eff}$ is defined in \eqref{SL35}.
    \end{prop}
    
    \section{Global well-posedness of the limiting equation without cut-off}\label{GWPofDNS}
    In this section, we establish the global well-posedness of the limiting equations without cut-off for small initial data.
    \subsection{The non-degenerate case}
    First, we consider the non-degenerate case. The corresponding limiting equation without cut-off is given by 
    \begin{align}
    \label{NDGWP1}
    \begin{cases}
    \partial_tw+(u\cdot \nabla)w-(w\cdot\nabla)u=\left(\nu+\frac{4\kappa}{5}\right)\Delta w+\frac{14\alpha\kappa}{15}\nabla\HH[n_3w_{3}],\\u=\curl^{-1}w,\\(\partial_y w_3-\alpha w_3)|_{y=0}=(w_h-\alpha 
    u_h^\perp)|_{y=0}=0,\\(\partial_y w_3+\alpha w_3)|_{y=1}=(w_h+\alpha 
    u_h^\perp)|_{y=1}=0,\\
    w|_{t=0}=w_{\ini},
    \end{cases}
    \end{align}
    where $\HH$ is defined in \eqref{SL6-1-1} and $n=(n_1,n_2,n_3)$ denotes the unit outward normal vector of the boundary. Throughout this section, the parameters $\alpha,\nu>0$ are fixed. Denote 
    \[\nu_{\kappa}:=\nu+\frac{4\kappa}{5},\qquad \beta_\kappa:=\frac{14\alpha\kappa}{15}.\]
    Although the limiting equation exhibits an enhanced viscosity, the boundary feedback term also modifies the boundary vorticity generation mechanism. This modification creates an additional difficulty, since the enhanced bulk dissipation does not immediately imply stability of the full linearized operator. To overcome this issue, we establish a resolvent estimate for the linear operator and use it to derive an exponential decay estimate for the corresponding semigroup. This linear decay estimate is then combined with a standard bootstrap argument to obtain global well-posedness for small initial data.
    
    \subsubsection{Linear decay estimate}
    Define the effective dissipation operator
    \[\LLL_{\eff}w:=\nu_\kappa\Delta w+\beta_\kappa \nabla\HH[n_3w_{3}]\]
    on the domain
    \[\DDDD(\LLL_{\eff}):=\{w\in H^2\cap H_{\curl}| \mbox{ $w$ satisfies the boundary conditions \eqref{NDGWP1}$_{3}$ and \eqref{NDGWP1}$_{4}$}\}.\]
    The main result is stated below.
    \begin{prop}\label{GWPprop1}
        Let $\alpha,\nu,\kappa>0$. Then, there are constants $C_{\alpha},c_\alpha>0$ depending only on $\alpha$ such that 
        \begin{align}
            \label{NDGWP2}\sup_{\realpart\lambda\ge -c_\alpha\nu_\kappa}\| (\lambda-\LLL_{\eff})^{-1}\|_{\LL(H_{\curl})}\le \frac{C_\alpha}{\nu_\kappa}
        \end{align}
    \end{prop}
    \begin{proof}
    By rescaling, it suffices to prove the following uniform-in-$\kappa$ resolvent estimate
    \begin{align}
            \label{NDGWP3}\sup_{\realpart\lambda\ge -c_\alpha}\| (\lambda-\nu_\kappa^{-1}\LLL_{\eff})^{-1}\|_{\LL(H_{\curl})}\le C_\alpha.
        \end{align}
    For this purpose, we consider the normalized effective dissipation operator
    \[\nu_\kappa^{-1}\LLL_{\eff} w= \Delta w+\fc_\kappa \nabla\HH[n_3w_{3}],\qquad \fc_\kappa=\frac{\beta_\kappa}{\nu_\kappa}\in \left(0,\frac{7\alpha}{6}\right),\]
    and the resolvent equation 
    \begin{align}
        \label{NDGWP4}
        \lambda w-\Delta w-\fc\nabla \HH[n_3w_3]=f.
    \end{align}
    Here and in what follows, we omit the dependence of $\fc_\kappa$ on $\kappa$, and establish estimates for the above resolvent equation uniformly in $\fc\in [0,\frac{7\alpha}{6}]$. The proof is divided into the following steps.

    \noindent\textbf{Step I.} We start with the vertical component $w_3$, which satisfies a closed equation
    \begin{align}
        \label{NDGWP5}
        \begin{cases}
            \lambda w_3-\Delta w_3-\fc\partial_y \HH[n_3w_3]=f_3,\\
            (\partial_y w_3-\alpha w_3)|_{y=0}=(\partial_y w_3+\alpha w_3)|_{y=1}=0.
        \end{cases}
    \end{align}
    Consider the tangential Fourier expansion
    \[w_{3}=\ssum_{k\neq 0} \hw_{3,k}e^{2\pi ik\cdot x},\]
    where we used the condition $\int_{\T^2}w_3\ddd x=0$ in \eqref{INTRO1-1}. From \eqref{NDGWP5}, it follows that 
    \begin{align}
        \label{NDGWP6}
        \begin{cases}
        (\lambda+4\pi^2|k|^2)\hw_{3,k}-\hw_{3,k}''-\fc \partial_y\HH_k=\hf_{3,k},\\
        (\hw'_{3,k}-\alpha \hw_{3,k})|_{y=0}=(\hw'_{3,k}+\alpha \hw_{3,k})|_{y=1}=0,
        \end{cases}
    \end{align}
    where $\HH_k:=\HH_k[n_3\hw_{3,k}]$ solves 
    \begin{align}
        \label{NDGWP7}
        \begin{cases}
            \HH_k''-4\pi^2|k|^2\HH_k=0,\\
            \HH_k|_{y=0,1}=n_3\hw_{3,k}.
        \end{cases}
    \end{align}
    After translating the vertical variable $y\mapsto y-\frac{1}{2}$, we regard the above equations as being posed on the interval $(-\frac{1}{2},\frac{1}{2})$. Notice that if $\hf_{3,k}$ is odd in $y$, then the corresponding solution $\tw_{3,k}$ remains odd; the same is true for the even part. Therefore, by linearity and by decomposing $\hf_{3,k}$ into its odd
    and even parts, it suffices to treat the odd and even cases separately.

    \noindent\textbf{Step I\!I.} We first give the details for the odd case. The even case can be handled by a similar argument, with the necessary modifications indicated afterwards. In the odd case, the equation \eqref{NDGWP6} reduces to 
    \begin{align}
        \label{NDGWP8}
        \begin{cases}
        (\lambda+4\pi^2|k|^2)\hw_{3,k}-\hw_{3,k}''-\fc \HH_k'=\hf_{3,k},\\
        \hw_{3,k}|_{y=0}=(\hw'_{3,k}+\alpha \hw_{3,k})|_{y=\frac{1}{2}}=0.
        \end{cases}
    \end{align}
    Let $\fq_k,\fu_k$ be the solutions of the following auxiliary Dirichlet problems, respectively:
    \begin{align}\label{NDGWP9}
        (\lambda+4\pi^2|k|^2)\fq_k-\fq_k''=\hf_{3,k},\qquad \fq_k(0)=\fq_k(1/2)=0,
    \end{align}
    and
    \begin{align}\label{NDGWP10}
        (\lambda+4\pi^2|k|^2)\fu_k-\fu_k''-\fc \HH_k'[n_3\fu_k]=0,\qquad \fu_k(0)=0,\qquad \fu_k(1/2)=1.
    \end{align}
    We shall choose a constant $\fa_k$ such that 
    \begin{align}
        \label{NDGWP11}\hw_{3,k}=\fq_k+\fa_k\fu_k
    \end{align}
    satisfies the resolvent equation \eqref{NDGWP8}. Indeed, notice that 
    \[\HH_k[n_3\fq_{k}]=\HH_k[0]=0,\]
    the equation \eqref{NDGWP8}$_1$ is satisfied by construction. Therefore, it remains to match the boundary condition  \eqref{NDGWP8}$_2$ at $y=\frac{1}{2}$. To this end, one needs to solve the equation \eqref{NDGWP10} explicitly. Since $\fu_k$ is odd, it follows 
    \[\HH_k[n_3\fu_k]=\HH_k[1]=\frac{\cosh(2\pi|k|y)}{\cosh(\pi|k|)},\qquad \HH_k'[n_3\fu_k]=2\pi |k|\frac{\sinh(2\pi|k|y)}{\cosh(\pi|k|)}.\] 
    Moreover, notice that 
    \[(\lambda+4\pi^2|k|^2)\frac{\fc}{\lambda}\HH_k'[n_3\fu_k]-\frac{\fc}{\lambda}\HH_k'''[n_3\fu_k]=\fc\HH_k'[n_3\fu_k],\]
    if $\lambda\neq0$. Therefore, the solution of \eqref{NDGWP10} is given by 
    \begin{align}\label{NDGWP12}
        \fu_k=C_{\neq}\sinh (\rho_{k,\lambda} y)+\frac{2\pi |k|\fc }{\lambda}\frac{\sinh(2\pi|k|y)}{\cosh(\pi|k|)},\qquad \lambda\neq 0,
    \end{align}
    where 
    \[\rho_{k,\lambda}:=(\lambda+4\pi^2|k|^2)^{\frac{1}{2}}\]
    is chosen with $\realpart \rho_{k,\lambda}>0$, and $C_{\neq}$ satisfies
    \[C_{\neq}\sinh (\rho_{k,\lambda}/2)=1-\frac{2\pi |k|\fc }{\lambda}\tanh(\pi|k|).\]
    If $\lambda=0$, the equation \eqref{NDGWP10} reduces to 
    \[ 4\pi^2|k|^2\fu_k-\fu_k''=2\pi |k|\fc\frac{\sinh(2\pi|k|y)}{\cosh(\pi|k|)}.\]
    Notice that 
    \[4\pi^2|k|^2(y\cosh(2\pi|k|y))-(y\cosh(2\pi|k|y))''=-4\pi|k|\sinh(2\pi|k|y)).\]
    Thus, the corresponding solution is given by
    \begin{align}\label{NDGWP13}
        \fu_k=C_0\sinh (2\pi|k|y)-\frac{\fc }{2}\frac{y\cosh(2\pi|k|y)}{\cosh(\pi|k|)},
    \end{align}
    where $C_0$ satisfies
    \[C_0\sinh(\pi|k|)=1+\frac{\fc}{4}.\]
    Plugging \eqref{NDGWP12} and \eqref{NDGWP13} into the boundary condition \eqref{NDGWP8}$_2$ at $y=\frac{1}{2}$ yields
    \begin{align}
        \label{NDGWP14}
        -\fq_k'(1/2)&=\fa_k\left(\rho_{k,\lambda}\coth(\rho_{k,\lambda}/2)\left(1-\frac{2\pi |k|\fc }{\lambda}\tanh(\pi|k|)\right)+\frac{4\pi^2|k|^2\fc}{\lambda}+\alpha\right)\notag\\&=\fa_k\left(\fm_k(\lambda)-\fn_k\fc\frac{\fm_k(\lambda)-\fm_k(0)}{\lambda}+\alpha\right)=:\FD(\lambda)\fa_k
    \end{align}
    for $\lambda\neq0$, where
    \begin{align}
        \label{NDGWP15}\fm_k(\lambda)=\rho_{k,\lambda}\coth(\rho_{k,\lambda}/2),\qquad \fn_k=2\pi |k|\tanh(\pi|k|).
    \end{align}
    Similarly, for $\lambda=0$, one has
    \begin{align}\label{NDGWP16}
        -\fq_k'(1/2)=\fa_k\left(\fm_k(0)-\fn_k\fc\fm_k'(0)+\alpha\right)=:\FD(0)\fa_k.
    \end{align}
    Hence, the validity of the decomposition \eqref{NDGWP11} reduces to appropriate lower bounds for the function $\FD(\lambda)$.

    \noindent\textbf{Step I\!I\!I.} To bound $\FD(\lambda)$, we establish some auxiliary estimates for $\fm_k$. Notice that 
    \[\fm_k(\lambda)=\fh_{k,\lambda}'(1/2),\]
    where $\fh_{k,\lambda}(y)=\frac{\sinh (\rho_{k,\lambda} y)}{\sinh (\rho_{k,\lambda} /2)}$ satisfies 
    \begin{align}
        \label{NDGWP17}(\lambda+4\pi^2|k|^2)\fh_{k,\lambda}-\fh_{k,\lambda}''=0,\qquad \fh_{k,\lambda}(0)=0,\qquad \fh_{k,\lambda}(1/2)=1.
    \end{align}
    Therefore, by the standard energy estimate, it follows 
    \begin{align}
        \label{NDGWP17-1}\realpart \fm_k(\lambda)=\int_0^{1/2}|\fh_{k,\lambda}'|^2\ddd y+\left(\realpart \lambda+4\pi^2|k|^2\right)\int_0^{1/2}|\fh_{k,\lambda}|^2\ddd y\ge 0,
    \end{align}
    if 
    \begin{align}\label{NDGWP18}
        \realpart \lambda\ge -2\pi^2.
    \end{align}
    Next, we turn to $\fm_{k}'$. Using \eqref{NDGWP17}, one has
    \begin{align*}
        (\fh_{k,\lambda}'\fh_{k,0}-\fh_{k,\lambda}\fh_{k,0}')'=\lambda \fh_{k,\lambda}\fh_{k,0},
    \end{align*}
    which gives 
    \begin{align}
        \label{NDGWP19}
        \frac{\fm_k(\lambda)-\fm_k(0)}{\lambda}= \int_0^{1/2}\fh_{k,\lambda}\fh_{k,0}\ddd y.
    \end{align}
    On the other hand, define 
    \begin{align}\label{NDGWP19-1}
        \Delta_k:=\partial_y^2-4\pi^2|k|^2,\qquad \DDD(\Delta_k):=H^2(0,1/2)\cap H^1_0(0,1/2).
    \end{align}
    From \eqref{NDGWP17}, it follows
    \[(\lambda-\Delta_k)(\fh_{k,\lambda}-\fh_{k,0})=-\lambda \fh_{k,0},\]
    which yields
    \begin{align}
        \label{NDGWP20}\fh_{k,\lambda}=\fh_{k,0}-(\lambda-\Delta_k)^{-1}\lambda \fh_{k,0},
    \end{align}
    if \eqref{NDGWP18} holds. Let $\{e_j\}_{j\ge 1}$ be the normalized eigenfunctions of $\Delta_k$, namely, 
    \[\Delta_k e_j= -4\pi^2(|k|^2+j^2)e_j,\qquad j\ge 1.\]
   Combining this with \eqref{NDGWP20}, one obtains
   \begin{align*}
       \fh_{k,\lambda}=\ssum_{j\ge 1} \frac{4\pi^2(|k|^2+j^2)}{\lambda+4\pi^2(|k|^2+j^2)}\langle \fh_{k,0},e_j \rangle e_j
   \end{align*}
   Plugging the above identity into \eqref{NDGWP19} and using the fact $\fh_{k,0}$ is a real-valued function, it follows 
   \begin{align*}
        \frac{\fm_k(\lambda)-\fm_k(0)}{\lambda}= \ssum_{j\ge 1} \frac{4\pi^2(|k|^2+j^2)}{\lambda+4\pi^2(|k|^2+j^2)}|\langle \fh_{k,0},e_j \rangle|^2,
    \end{align*}
    which, together with \eqref{NDGWP18}, implies
    \begin{align*}
        \left|\frac{\fm_k(\lambda)-\fm_k(0)}{\lambda}\right|\le \ssum_{j\ge 1} \left|\frac{4\pi^2(|k|^2+j^2)}{\lambda+4\pi^2(|k|^2+j^2)}\right||\langle \fh_{k,0},e_j \rangle|^2\le \frac{4}{3}\|\fh_{k,0}\|^2_{L^2_y}.
    \end{align*}
    Since by taking $\lambda\to0$ in \eqref{NDGWP19}, it follows
    \[\|\fh_{k,0}\|^2_{L^2_y}=\fm_k'(0),\]
    one gets
    \begin{align}\label{NDGWP21}
        \fn_k\left|\frac{\fm_k(\lambda)-\fm_k(0)}{\lambda}\right|\le \frac{8\pi |k|}  {3}\fm_k'(0)\tanh(\pi|k|)=\frac{2}{3}\left(1-\frac{2\pi|k|}{\sinh(2\pi|k|)}\right)\le \frac{2}{3}.
    \end{align}
    Combining this with \eqref{NDGWP17-1} and the fact $\fc\in[0,\frac{7\alpha}{6}]$, one gets
    \begin{align*}
        |\FD(\lambda)| \ge |\fm_k(\lambda)+\alpha|\left(1-\frac{2\fc}{3|\fm_k(\lambda)+\alpha|}\right)\ge \frac{2|\fm_k(\lambda)+\alpha|}{9}\ge \frac{\sqrt{2}(\alpha+|\fm_k(\lambda)|)}{9}.
    \end{align*}
    Since \eqref{NDGWP18} ensures
    \[\realpart \rho_{k,\lambda}\ge \sqrt{2}\pi>0,\]
    it follows 
    \begin{align*}
        |\fm_k(\lambda)|=|\rho_{k,\lambda}|\left|\frac{1+e^{-\rho_{k,\lambda}}}{1-e^{-\rho_{k,\lambda}}}\right|\ge |\rho_{k,\lambda}|\frac{1-e^{-\realpart\rho_{k,\lambda}}}{1+e^{-\realpart\rho_{k,\lambda}}}\ge \frac{1-e^{-\sqrt{2}\pi}}{2}|\rho_{k,\lambda}|.
    \end{align*}
    Therefore, one has
    \begin{align}
        \label{NDGWP22} |\FD(\lambda)| \ge C_\alpha (1+|\rho_{k,\lambda}|).
    \end{align}

    \noindent\textbf{Step I\!V.} We are in a position to close the estimate for the resolvent equation \eqref{NDGWP8}. Indeed, notice that the decomposition \eqref{NDGWP11} holds with $\fa_k=-\FD^{-1}(\lambda)\fq_k'(1/2)$, where 
    \[\fq_k(y)=\int_0^{1/2}\frac{\sinh(\rho_{k,\lambda}\min\{y,y'\})\sinh(\rho_{k,\lambda}(1/2-\max\{y,y'\}))}{\rho_{k,\lambda}\sinh(\rho_{k,\lambda}/2)}\hf_{3,k}(y')\ddd y'.\]
    This, together with \eqref{NDGWP22} and the Cauchy--Schwarz inequality, implies 
    \begin{align}\label{NDGWP22-0}
        |\fa_k|&\lesssim_\alpha (1+|\rho_{k,\lambda}|)^{-1}|\fq_k'(1/2)|\notag\\&\lesssim_\alpha (1+|\rho_{k,\lambda}|)^{-1}\int_0^{1/2}\left|\frac{\sinh(\rho_{k,\lambda} y')}{\sinh (\rho_{k,\lambda}/2)}\hf_{3,k}(y')\right|\ddd y'\notag\\&\lesssim_\alpha (1+|\rho_{k,\lambda}|)^{-1}\left\|\frac{\sinh(\rho_{k,\lambda} \cdot)}{\sinh (\rho_{k,\lambda}/2)}\right\|_{L^2_y}\|\hf_{3,k}\|_{L^2_y}\lesssim_\alpha (1+|\rho_{k,\lambda}|)^{-\frac{3}{2}}\|\hf_{3,k}\|_{L^2_y}.
    \end{align}
    We turn to the estimate of $\fu_k$. Suppose that $\lambda\neq0$. Then,
    \begin{align}\label{NDGWP22-1}
        \fu_k&=\frac{\sinh \rho_{k,\lambda}y}{\sinh (\rho_{k,\lambda}/2)}+\frac{\fc\fn_k}{\lambda}\left(\frac{\sinh 2\pi|k|y}{\sinh \pi|k|}-\frac{\sinh \rho_{k,\lambda}y}{\sinh (\rho_{k,\lambda}/2)}\right)\notag\\&=\ff_{\rho_{k,\lambda}}(y)-\frac{\fc\fn_k}{\rho_{k,\lambda}+2\pi|k|}\frac{\ff_{\rho_{k,\lambda}}-\ff_{2\pi|k|}}{\rho_{k,\lambda}-2\pi|k|},
    \end{align}
    where 
    \[\ff_{\llll}(y):=\frac{\sinh \llll y}{\sinh (\llll/2)}.\]
    As $\fn_k=2\pi|k|\tanh (\pi|k|)\sim |k|$, it remains to bound the function $\ff_{\llll}(y)$. Notice that 
    \begin{align}\label{NDGWP22-1-1}\realpart\rho_{k,\lambda}^2\ge 2\pi^2>0 \qquad\Rightarrow\qquad |\arg \rho_{k,\lambda}|<\frac{\pi}{4},
    \end{align}
    it suffices to consider $\llll \in \{|\arg \llll|<\frac{\pi}{4},\realpart \llll\ge \sqrt{2}\pi\}$. Since in this case, one has 
    \begin{align}
        \label{NDGWP23}\realpart \llll \ge \frac{\sqrt{2}}{2}|\llll|,
    \end{align}
    it follows
    \begin{align}\label{NDGWP24}
        |\ff_{\llll}(y)|=\left|\frac{e^{-(1/2-y)\llll}-e^{-(1/2+y)\llll}}{1-e^{-\llll}}\right|\le \frac{e^{-\frac{\sqrt{2}}{2}(1/2-y)|\llll|}+e^{-\frac{\sqrt{2}}{2}(1/2+y)|\llll|}}{1-e^{-\sqrt{2}\pi}}\lesssim e^{-\frac{\sqrt{2}}{2}(1/2-y)|\llll|}.
    \end{align}
    Similarly, for $j=1,2$,
    \begin{align}\label{NDGWP25}
        |\partial_y^j\ff_{\llll}(y)|\lesssim |\llll|^je^{-\frac{\sqrt{2}}{2}(1/2-y)|\llll|}.
    \end{align}
    To estimate the difference quotient in \eqref{NDGWP22-1}, one needs to address
    \[\partial_\llll \ff_\llll=\frac{-(1/2-y)e^{-(1/2-y)\llll}+(1/2+y)e^{-(1/2+y)\llll}}{1-e^{-\llll}}-e^{-\llll}\frac{e^{-(1/2-y)\llll}-e^{-(1/2+y)\llll}}{(1-e^{-\llll})^2},\]
    which is bounded in the following way:
    \begin{align*}
        |\partial_\llll \ff_\llll|&\lesssim (1/2-y)e^{-\frac{\sqrt{2}}{2}(1/2-y)|\llll|}+(1/2+y)e^{-\frac{\sqrt{2}}{2}(1/2+y)|\llll|}+e^{-\frac{\sqrt{2}}{2}|\llll|}e^{-\frac{\sqrt{2}}{2}(1/2-y)|\llll|}
        \\&\lesssim |\llll|^{-1}e^{-\frac{\sqrt{2}}{4}(1/2-y)|\llll|}\!+|\llll|^{-1}e^{-\frac{\sqrt{2}}{4}(1/2+y)|\llll|}\!+|\llll|^{-1}e^{-\frac{\sqrt{2}}{2}(1/2-y)|\llll|}\lesssim |\llll|^{-1}e^{-\frac{\sqrt{2}}{4}(1/2-y)|\llll|}.
    \end{align*}
    This implies 
    \begin{align}\label{NDGWP26}
        |\partial_y^j\partial_\llll \ff_\llll|\lesssim |\llll|^{j-1}e^{-\frac{\sqrt{2}}{4}(1/2-y)|\llll|}
    \end{align}
    for $j=0,1,2$. Now by applying \eqref{NDGWP24}--\eqref{NDGWP26} combined with \eqref{NDGWP22-1}, \eqref{NDGWP22-1-1}, and the mean value theorem, one has 
    \begin{align*}
        \|\partial_y^j\fu_k\|&\lesssim \|\partial_y^j\ff_{\rho_{k,\lambda}}\|+\fc\left\|\frac{\partial_y^j\ff_{\rho_{k,\lambda}}-\partial_y^j\ff_{2\pi|k|}}{\rho_{k,\lambda}-2\pi|k|}\right\|\\&\lesssim\|\partial_y^j\ff_{\rho_{k,\lambda}}\|+\alpha \int_{0}^{1}\|\partial_y^j\partial_\llll \ff_\llll|_{\llll= 2\pi|k|+(\rho_{k,\lambda}-2\pi|k|)s} \|\ddd s\\&\lesssim_\alpha|\rho_{k,\lambda}|^{j-\frac{1}{2}}+\int_{0}^{1}|2\pi|k|+(\rho_{k,\lambda}-2\pi|k|)s|^{j-\frac{3}{2}}\ddd s\lesssim_\alpha|\rho_{k,\lambda}|^{j-\frac{1}{2}},
    \end{align*}
    which yields
    \begin{align}
        \label{NDGWP27}
        (|\lambda|+4\pi^2|k|^2)\|\fu_k\|+2\pi|k|\|\fu_k'\|+\|\fu_k''\|\lesssim_\alpha |\rho_{k,\lambda}|^{\frac{3}{2}}.
    \end{align}
    The estimate for $\fu_k$ with $\lambda=0$ can be obtained by taking $\lambda\to0$ in \eqref{NDGWP22-1} and repeating the above argument, so we omit the details. Moreover, the standard resolvent estimate for the operator $\Delta_k$ defined in \eqref{NDGWP19-1} gives the following estimate for $\fq_k$
    \begin{align*}
        (|\lambda|+4\pi^2|k|^2)\|\fq_k\|+2\pi|k|\|\fq_k'\|+\|\fq_k''\|\lesssim \|\hf_{3,k}\|_{L^2_y},
    \end{align*}
    provided that \eqref{NDGWP18} is satisfied. Combining this with \eqref{NDGWP11}, \eqref{NDGWP22-0}, and \eqref{NDGWP27}, one gets
    \begin{align}
        \label{NDGWP28}(|\lambda|+4\pi^2|k|^2)\|\hw_{3,k}\|+2\pi|k|\|\hw_{3,k}'\|+\|\hw_{3,k}''\|\lesssim_{\alpha} \|\hf_{3,k}\|_{L^2_y},
    \end{align}
    when $\hf_{3,k}$ is odd. 
    
    The even case can be treated in the same way. The auxiliary problems are now defined by 
    \begin{align*}
        (\lambda+4\pi^2|k|^2)\fq_k-\fq_k''=\hf_{3,k},\qquad \fq'_k(0)=\fq_k(1/2)=0,
    \end{align*}
    and
    \begin{align*}
        (\lambda+4\pi^2|k|^2)\fu_k-\fu_k''-\fc \HH_k'[n_3\fu_k]=0,\qquad \fu'_k(0)=0,\qquad \fu_k(1/2)=1.
    \end{align*}
    Therefore, the decomposition \eqref{NDGWP11} holds with $\fa_k$ given by
    \[\fa_k:=-\FD^{-1}(\lambda)\fq_k'(1/2),\]
    where
    \[\FD(\lambda):=\fm_k(\lambda)-\fn_k\fc\frac{\fm_k(\lambda)-\fm_k(0)}{\lambda}+\alpha\]
    with 
    \begin{align*}\fm_k(\lambda):=\rho_{k,\lambda}\tanh(\rho_{k,\lambda}/2),\qquad \fn_k:=2\pi |k|\coth(\pi|k|).
    \end{align*}
    Repeating the derivation of \eqref{NDGWP21}, one gets 
    \begin{align*}
        \fn_k\left|\frac{\fm_k(\lambda)-\fm_k(0)}{\lambda}\right|\le \frac{5}  {3}\fm_k'(0)\fn_k=\frac{5}{6}\!\left(1+\frac{2\pi|k|}{\sinh(2\pi|k|)}\right)\!\le\! \frac{5}{6}\left(1+\frac{2\pi}{\sinh(2\pi)}\right)=:c_{\even}.
    \end{align*}
    Since $c_{\even}<\frac{6}{7}$, one has the estimate
    \[|\FD(\lambda)| \ge |\fm_k(\lambda)+\alpha|\left(1-\frac{c_{\even}\fc}{|\fm_k(\lambda)+\alpha|}\right)\ge\left(1-\frac{7c_{\even}}{6}\right) |\fm_k(\lambda)+\alpha|,\]
    which implies
    \[|\FD(\lambda)|\ge C_\alpha(1+|\rho_{k,\lambda}|).\]
    This, together with the arguments used in the derivation of \eqref{NDGWP22-0} and \eqref{NDGWP27}, leads to \eqref{NDGWP28} in the even case. 
    
    Hence, the following estimate 
    \begin{align}
        \label{NDGWP29}
        |\lambda|\|w_3\|+\| w_3\|_{H^2}\lesssim_{\alpha} \|f_{3}\|
    \end{align}
    holds for the resolvent equation \eqref{NDGWP5} with $\lambda$ satisfying \eqref{NDGWP18}.

    \noindent\textbf{Step V.} Finally, we turn to the horizontal components $w_h$, which satisfy
    \[\begin{cases}
        \lambda w_h-\Delta w_h-\fc\nabla_x \HH[n_3w_3]=f_h,\qquad u=\curl^{-1} w,\\
        (w_h-\alpha 
    u_h^\perp)|_{y=0}=(w_h+\alpha 
    u_h^\perp)|_{y=1}=0.
    \end{cases}\]
    Applying $I-\BB$ to both sides, where $\BB$ is the boundary correction operator defined in \eqref{LS3}, the above equation reduces to 
    \begin{align}
        \label{NDGWP30}\lambda \tw_h-\Delta \tw_h+([\BB,\partial_y^2]w)_h=\tf_h+\fc(I-\BB)\nabla_x \HH[n_3w_3],\qquad \tw_h|_{y=0,1}=0,
    \end{align}
    where $\tw:=(I-\BB)w$ and $\tf:=(I-\BB)f$. We claim that there exist constants $c_\alpha, C_\alpha>0$ such that for any $\lambda$ satisfying $\realpart \lambda\ge -c_\alpha$ and $\fc\in[0,\frac{7\alpha}{6}]$, it follows
    \begin{align}\label{NDGWP30-1}
        \| w_h\|\le C_{\alpha} \|f\|.
    \end{align}
    By contradiction, suppose that there are $\lambda_n,\fc_n$ such that 
    \[\realpart \lambda_n\ge -c_\alpha,\qquad \fc_n\in\left[0,\frac{7\alpha}{6}\right],\]
    and functions $w^n_h,f^n$ satisfying \eqref{NDGWP30} and
    \begin{align}\label{NDGWP31}
        \|w^n_h\|=1,\qquad \lim_{n\to\infty}\|f^n\|=0.
    \end{align}
    Then, by applying \eqref{NDGWP29} and \eqref{LE6-1}, one has
    \begin{align}
        \label{NDGWP32}w_3^n\to0,\qquad \fc_n(I-\BB)\nabla_x \HH[n_3w^n_3]\to0,\qquad\mbox{in $L^2$,}
    \end{align}
    as $n\to\infty$. We split the rest of the argument into two cases.

    \noindent \textbf{Case I.} Suppose that the sequence $\{\lambda_n\}_{n\ge 1}$ is bounded. Applying the standard elliptic estimate to \eqref{NDGWP30} and using \eqref{NDGWP31}, \eqref{NDGWP32}, and Proposition \ref{propB2}, one has 
    \begin{align*}
        \|\tw^n_h\|_{H^2}&\le |\lambda_n|\|\tw^n_h\|+\|[\BB,\partial_y^2]w^n\|+\|\tf^n_h+\fc_n(I-\BB)\nabla_x \HH[n_3w_3^n]\|\\&\lesssim_\alpha \sup_{n\ge 1}\left(\|w^n\|+\|\tf^n_h+\fc_n(I-\BB)\nabla_x \HH[n_3w_3^n]\|\right)<\infty.
    \end{align*}
    Combining this with the compact Sobolev embedding theorem and passing to a subsequence if necessary, it follows
    \[\fc_n\to\fc_\infty,\qquad\lambda_n\to\lambda_\infty,\qquad w^n\to w^\infty,\]
    as $n\to\infty$, where $\fc_\infty,\lambda_\infty$ satisfy
    \[\fc_\infty\in [0,\frac{7\alpha}{6}],\qquad \realpart\lambda_\infty\ge -c_\alpha,\]
    and the convergence of $w^n$ holds in the strong topology of $H^1$ and the weak topology of $H^2$. Therefore, $w^\infty$ solves the limiting equation
    \[\lambda_\infty \tw^\infty_h-\Delta \tw^\infty_h+([\BB,\partial_y^2]w^\infty)_h=0,\qquad w_3^\infty=0,\qquad\tw^\infty_h|_{y=0,1}=0,\]
    which is equivalent with
    \begin{align*}
        \begin{cases}
            \lambda_\infty w^\infty_h-\Delta w^\infty_h=0,\qquad w_3^\infty=0,\qquad u^\infty=\curl^{-1} w^\infty,\\
        (w^\infty_h-\alpha 
    (u^\infty_h)^\perp)|_{y=0}=(w^\infty_h+\alpha 
    (u^\infty_h)^\perp)|_{y=1}=0.
    \end{cases}
    \end{align*}
    By applying Lemma \ref{lemmaA4}, from the above equation, one has
    \begin{align*}
    \begin{cases}
        \lambda_\infty u^\infty-\Delta u^\infty+\nabla p_h+(C_1,C_2,0)=0,\\
        (\partial_y u^\infty_h-\alpha 
    u^\infty_h)|_{y=0}=(\partial_y u^\infty_h+\alpha 
    u^\infty_h)|_{y=1}=u^\infty_3|_{y=0,1}=0
    \end{cases}
    \end{align*}
    for some constants $C_1,C_2$ and function $p_h$. Notice that by the definition \eqref{A13} of the Biot--Savart operator $\curl^{-1}$, 
    \[\int_Du_h^\infty\ddd x=0.\]
    Then, by the standard energy method, it follows
    \begin{align*}
        \realpart \lambda_\infty\|u^\infty\|^2+\|\nabla u^\infty\|^2+\alpha\|u^\infty_h\|^2_{L^2(\partial D)}=0,
    \end{align*}
    which, together with choosing 
    \begin{align}\label{NDGWP33}c_\alpha\le\frac{C_p}{2},
    \end{align}
    where $C_p$ is the constant in the Poincar\'e inequality, leads to $u^\infty\equiv 0$ and $w_h^\infty\equiv 0$. Since $\|w^n_h\|=1$ and $w^n$ converges to $w^\infty$ in the strong topology of $H^1$, this is a contradiction.

    \noindent\textbf{Case I\!I.} Now we assume that $|\lambda_n|\to\infty$, as $n\to\infty$. From \eqref{NDGWP30}, it follows 
    \begin{align*}
        \tw_h^n=(\lambda_n-\Delta_D)^{-1}\left(\tf_h^n+\fc_n(I-\BB)\nabla_x \HH[n_3w_3^n]-([\BB,\partial_y^2]w^n)_h\right),
    \end{align*}
    where $\Delta_D$ denotes the Dirichlet--Laplacian operator. By the spectral theorem for the positive self-adjoint operator $-\Delta_D$, one has 
    \[\|(\lambda-\Delta_D)^{-1}\|_{\LL(L^2)}=\sup_{j\ge 1}|\lambda+\fb_j|^{-1},\]
    where $\fb_j>0$ are the eigenvalues of $-\Delta_D$. Combining the above results with Proposition \ref{propB2} and \eqref{NDGWP32}, one obtains
    \begin{align*}
        \|\tw_h^n\|&\lesssim_\alpha |\lambda_n|^{-1}\left(\|\tf_h^n+\fc_n(I-\BB)\nabla_x \HH[n_3w_3^n]\|+\|w^n_3\|+\|\tw^n_h\|\right)\\&\lesssim_\alpha |\lambda_n|^{-1} \|\tw^n_h\|+|\lambda_n|^{-1}\sup_{n\ge 1}\left(\|\tf_h^n+\fc_n(I-\BB)\nabla_x \HH[n_3w_3^n]\|+\|w^n_3\|\right).
    \end{align*}
    By taking $n$ sufficiently large, it follows
    \begin{align*}
        \|\tw_h^n\|\lesssim_\alpha|\lambda_n|^{-1}\sup_{n\ge 1}\left(\|\tf_h^n+\fc_n(I-\BB)\nabla_x \HH[n_3w_3^n]\|+\|w^n_3\|\right)\to0,
    \end{align*}
    as $n\to\infty$. Together with Proposition \ref{propB1}, this contradicts $\|w^n_h\|=1$. Hence, the estimate \eqref{NDGWP30-1} holds as claimed.

    Now set $c_\alpha:=\min \{2\pi^2,C_p/2\}$. By combining \eqref{NDGWP18}, \eqref{NDGWP29}, \eqref{NDGWP30-1}, and \eqref{NDGWP33}, one has 
    \[\|(\lambda-\nu_\kappa^{-1}\LLL_{\eff})^{-1}f\|=\|w\|\lesssim_\alpha \|f\|.\]
    This implies \eqref{NDGWP3}, and by rescaling, completes the proof.
    \end{proof}

    We now turn the resolvent estimate into an exponential decay estimate for the linear semigroup. To this end, we use the following exponential stability criterion, cf. Theorem 1.4 in \cite{HSV24} and references therein.
    \begin{theorem}\label{TheoremES}
        Let $S(t)$ be a strongly continuous semigroup on a Hilbert space $H$ with the generator $A$. Suppose that there is $\lambda_0\in \R$ such that 
        \[r_0^{-1}:=\sup_{\realpart \lambda>\lambda_0}\|(\lambda-A)^{-1}\|_{\LL(H)}<\infty.\]
        Assume that 
        \[\|S(t)\|_{\LL(H)}\le m(t),\qquad \forall t\ge 0\]
        for some continuous positive function $m(t)$. Then, for any $t,a,b>0$ such that $t\ge a+b$,
        \[\|S(t)\|_{\LL(H)}\le \frac{e^{\lambda_0t-r_0(t-a-b)}}{r_0\|\frac{1}{m}\|_{e^{-\lambda_0}L^2(0,a)}\|\frac{1}{m}\|_{e^{-\lambda_0}L^2(0,b)}},\]
        where 
        \[\|f\|_{e^{-\lambda_0}L^2(0,a)}^2:=\int_0^a|f(t)|^2e^{2\lambda_0 t}\ddd t.\]
    \end{theorem}
    Now we are in a position to present the following linear decay estimate.
    \begin{prop}\label{GWPprop3}
        Let $\alpha,\nu,\kappa>0$. Then, there are constants $C_{\alpha},c_\alpha>0$ depending only on $\alpha$ such that 
        \begin{align}
            \label{NDGWP34}
            \|e^{t\LLL_{\eff}}\|_{\LL(H_{\curl})}\le C_\alpha e^{-c_\alpha \nu_\kappa t}.
        \end{align}
    \end{prop}   
    \begin{proof}
        \textbf{Step I.} In this step, we derive a rough semigroup bound. Consider the linear evolutionary problem 
        \begin{align*}
        \begin{cases}
        \partial_t w=\LLL_{\eff} w,\\w|_{t=0}=w_{\ini}.
        \end{cases}
    \end{align*}
    Let $\tw:=(I-\BB)w$, where $\BB$ is the boundary correction operator defined in \eqref{LS3}. Applying $(I-\BB)$ on both sides of the above equation, it follows
     \begin{align}\label{NDGWP35}
        \begin{cases}
            \partial_t \tw=\nu_\kappa \Delta \tw-\nu_\kappa [\BB,\partial_y^2]w+\beta_\kappa(I-\BB)\nabla\HH[n_3w_3],
            \\(\partial_y \tw_3-\alpha\tw_3)|_{y=0}=(\partial_y \tw_3+\alpha\tw_3)|_{y=1}=\tw_h|_{y=0,1}=0,\\
            \tw|_{t=0}=\tw_{\ini}.
        \end{cases}
    \end{align}
    By the standard energy estimate combined with Proposition \ref{propB2} and \eqref{LE6-1}, it follows that 
    \begin{align*}
    \frac{\ddd }{\ddd t}\|\tw\|^2+\nu_\kappa\left(\|\nabla\tw\|^2+\alpha\|\tw_3\|^2_{L^2(\partial D)}\right)\lesssim_\alpha \nu_\kappa\|\tw\|^2,
    \end{align*}
     which implies
    \begin{align*}
        \|\tw(t)\| \le e^{C_\alpha \nu_\kappa t}\|\tw_{\ini}\|,
    \end{align*}
    and thus 
    \begin{align}
    \label{NDGWP36}\|e^{t\LLL_{\eff}}\|_{\LL(H_{\curl})}\le C_\alpha e^{C_\alpha\nu_\kappa t}.
    \end{align}

    \noindent\textbf{Step I\!I.} Now we apply Theorem \ref{TheoremES} with
    \[\lambda_0:=-c_\alpha \nu_\kappa,\qquad m(t):=C_\alpha e^{C_\alpha \nu_\kappa  t},\qquad a=b=r_0^{-1},\]
    where $c_\alpha$ is given in Proposition \ref{GWPprop1}. Notice that 
    \begin{align*}
        r_0\|m^{-1}\|_{e^{-\lambda_0}L^2(0,r_0^{-1})}^2&=C_\alpha^{-2}r_0\int_{0}^{r_0^{-1}}\exp\left(-(2c_\alpha\nu_\kappa +2C_\alpha\nu_\kappa )t\right)\ddd t\\&=\frac{r_0}{2C_\alpha^2\nu_\kappa(c_\alpha + C_\alpha) }\left(1-\exp\left(-\frac{2\nu_\kappa(c_\alpha + C_\alpha) }{r_0}\right)\right),
    \end{align*}
    where, by using \eqref{NDGWP2}, one has
    \[\frac{2\nu_\kappa(c_\alpha + C_\alpha) }{r_0}\le 2(c_\alpha + C_\alpha)C_\alpha=:C'_\alpha.\]
    Since the function $t\mapsto t^{-1}(1-e^{-t})$ is decreasing on $\R_+$, this gives
     \begin{align*}
        r_0\|m^{-1}\|_{e^{-\lambda_0}L^2(0,r_0^{-1})}^2\ge\frac{1-e^{-C_\alpha'}}{C_{\alpha}^2C_\alpha'},
    \end{align*}
    and thus
    \begin{align}\label{NDGWP37}
        \|e^{t\LLL_{\eff}}\|_{\LL(H_{\curl})}\le \frac{C_{\alpha}^2C_\alpha'}{1-e^{-C_\alpha'}} e^{-c_\alpha\nu_\kappa t},\qquad t\ge \frac{2}{r_0}.
    \end{align}
    For $t\in [0,2/r_0)$, notice that from \eqref{NDGWP2}, one has
    \[\frac{\nu_\kappa}{r_0}\le C_\alpha.\]
    This, together with the rough bound \eqref{NDGWP36}, gives
    \[\|e^{t\LLL_{\eff}}\|_{\LL(H_{\curl})}\le C_\alpha \exp((C_\alpha+c_\alpha)\nu_\kappa t)e^{-c_\alpha\nu_\kappa t}\le  C_\alpha \exp(2(C_\alpha+c_\alpha)C_\alpha)e^{-c_\alpha\nu_\kappa t}.\]
    Combining the above estimate with \eqref{NDGWP37}, we obtain \eqref{NDGWP34}.
    \end{proof}
    Let $\XX^m$ denote the anisotropic Sobolev space defined in \eqref{INTRO7}. The following linear estimate in $\XX^m$ follows directly from the above semigroup estimate; for completeness, its proof is postponed to Appendix \ref{Duhamel}.
     \begin{coro}\label{NDGWPcoro1}
        Let $m$ be a positive integer and $f\in L^2_{\loc}(0,\infty;H_{\curl})$. Suppose that $w$ solves 
        \begin{align}\label{NDGWP38}
        \begin{cases}
            \partial_t w=\nu_\kappa \Delta w+\beta_\kappa\nabla\HH[n_3w_3]+f,\qquad u=\curl^{-1}w,
            \\(\partial_y w_3-\alpha w_3)|_{y=0}=(w_h-\alpha 
            u_h^\perp)|_{y=0}=0,\\(\partial_y w_3+\alpha w_3)|_{y=1}=(w_h+\alpha
            u_h^\perp)|_{y=1}=0,\\
            w|_{t=0}=w_{\ini}.
        \end{cases}
    \end{align}
    Then, there is a constant $C_{\alpha}>0$ depending only on $\alpha$ such that for any $T>0$,
    \[\sup_{t\le T}\|w(t)\|^2_{m,0}+\nu_{\kappa}\int_0^T\left(\|\nabla w\|^2_{m,0}+\| w\|^2_{m,0}\right)\ddd t\le C_\alpha\|w_{\ini}\|^2_{m,0}+\frac{C_\alpha}{\nu_\kappa}\int_0^T\|f\|^2_{m-1,0}\ddd t.\]
    \end{coro}
    \subsubsection{Nonlinear estimate}
    In this subsection, we prove the global well-posedness of \eqref{NDGWP1} with small initial data. The main result is given below.
    \begin{prop}\label{GWPprop5}
        Let $\alpha,\nu,\kappa>0$ and $m\ge 2$ be an integer. Then, there is a constant $C_{\alpha,m}>0$ such that if 
        \begin{align}
            \label{NDGWP39}w_{\ini}\in H_{\curl}\cap \XX^m,\qquad\|w_{\ini}\|_{m,0}\le \frac{\nu_\kappa}{C_{\alpha,m}},
        \end{align}
        the equation \eqref{NDGWP1} admits a unique global solution satisfying 
        \begin{align}\label{NDGWP40}
            \sup_{t\ge 0}\|w(t)\|^2_{m,0}+\nu_{\kappa}\int_0^\infty\|\nabla w(t)\|^2_{m,0}\ddd t\lesssim_{\alpha,m}\|w_{\ini}\|^2_{m,0}.
        \end{align}
    \end{prop}
     \begin{proof}\textbf{Step I.} In this step, we estimate the nonlinear terms. By applying Corollary \ref{coroA4}, the Sobolev embedding $H_y^1\hookrightarrow L^\infty_y$, and the divergence-free condition of $w$, it follows 
        \begin{align}
            \label{NDGWP41}
            \|(w\cdot\nabla)u\|_{m-1,0}&\lesssim_{m}\|\|w_h\|_{H^{m-1}_x}\|\nabla_x u\|_{H^{m-1}_x}\|_{L^2_y}+\|\|w_3\|_{H^{m-1}_x}\|\partial_y u\|_{H^{m-1}_x}\|_{L^2_y}\notag\\&\lesssim_{m}\|w_h\|_{m-1,0}\|\nabla_xu\|_{H^1_yH_x^{m-1}}+\|w_3\|_{H^1_yH^{m-1}_x}\|\partial_y u\|_{m-1,0}\notag\\&\lesssim_m\|w\|_{m,0}^2+\|w_3\|_{m-1,0}\|w\|_{m-1,0}+\|\partial_yw_3\|_{m-1,0}\|w\|_{m-1,0}\notag\\&\lesssim_m\|w\|_{m,0}^2+\|\dive w_h\|_{m-1,0}\|w\|_{m-1,0}\notag\\
            &\lesssim_m\|w\|_{m,0}^2.
        \end{align}
    Similarly,
    \begin{align}
        \label{NDGWP42}
        \|(u\cdot \nabla)w\|_{m-1,0}\lesssim_m\|u\|_{H^1_yH^{m-1}_x}\|\nabla w\|_{m-1,0}\lesssim_m\|w\|_{m,0}\|\nabla w\|_{m,0}.
    \end{align}

    \noindent\textbf{Step I\!I.} Define 
    \[\EE_m(T):=\sup_{t\le T}\|w(t)\|^2_{m,0}+\nu_{\kappa}\int_0^T(\|\nabla w\|^2_{m,0}+\| w\|^2_{m,0})\ddd t.\]
    By applying Corollary \ref{NDGWPcoro1} combined with \eqref{NDGWP41} and \eqref{NDGWP42}, it follows
    \begin{align*}
        \EE_m(T)&\le C_\alpha \|w_{\ini}\|^2_{m,0}+\frac{C_{\alpha,m}}{\nu_\kappa}\int_0^T\|w\|_{m,0}^2(\|\nabla w\|_{m,0}^2+\|w\|_{m,0}^2)\ddd t\\&\le C_\alpha \|w_{\ini}\|^2_{m,0}+\frac{C_{\alpha,m}}{\nu^2_\kappa}\EE_m^2(T).
    \end{align*}
    Let 
    \[T^*:=\sup\{T>0| \EE_m(T)\le 2C_\alpha \|w_{\ini}\|^2_{m,0}\}>0.\]
    For any $T<T^*$, one has
    \[\EE_m(T)\le C_\alpha \|w_{\ini}\|^2_{m,0}+\frac{4C_{\alpha,m}C^2_\alpha}{\nu^2_\kappa}\|w_{\ini}\|^4_{m,0}.\]
    Hence, if 
    \[\|w_{\ini}\|^2_{m,0}\le\frac{\nu^2_\kappa}{8C_{\alpha,m}C_\alpha},\]
    one obtains 
    \[\EE_m(T)\le \frac{3C_\alpha}{2} \|w_{\ini}\|^2_{m,0}<2C_\alpha \|w_{\ini}\|^2_{m,0},\]
    which implies $T^*=\infty$ and \eqref{NDGWP40}.
    \end{proof}

    \subsection{The degenerate case}
    Now we consider the degenerate case. After correcting the boundary conditions, the limiting equation without cut-off is given by
    \begin{align}
        \label{DGWP1}\begin{cases}
            \partial_t\tw+(u\cdot \nabla)\tw-(\tw\cdot\nabla)u=\nu \FJ\Delta\FJ^{-1}\tw+\kappa\FJ\FA_{\eff}\FJ^{-1}\tw+ \RR[w],
            \\(\partial_y \tw_3-\alpha \tw_3)|_{y=0}=(\partial_y \tw_3+\alpha\tw_3)|_{y=1}=\tw_h|_{y=0,1}=0,\\
    \tw|_{t=0}=\tw_{\ini},
        \end{cases}
    \end{align}
    where $\FJ:=I-\BB$, $\tw:=\FJ w$, $\FA_{\eff}$ is defined in \eqref{SL35}, and
    \begin{align}
        \label{DGWP2}\RR[w]:=\BB(u\cdot\nabla)w-\BB(w\cdot\nabla)u-(u\cdot\nabla)\BB w+(\BB w\cdot\nabla)u.
    \end{align}
    Introduce the effective dissipation operator
    \[\LLL_{\eff}\tw:=\nu \FJ\Delta\FJ^{-1}\tw+\kappa\FJ\FA_{\eff}\FJ^{-1}\tw\]
    on the domain
    \[\DDDD(\LLL_{\eff}):=\{\tw\in H^2|\mbox{ $\tw$ satisfies the corrected boundary condition \eqref{DGWP1}$_2$}\}.\]
    Consider the tangential Fourier expansion 
    \[\tw=\ssum_{k}e^{2\pi ik\cdot x}\htw(k,y),\]
    and decompose $\tw$ into its zeroth tangential mode and non-zero modes
    \[\tw_=:=\htw(0,y),\qquad \tw_{\neq}:=\ssum_{k\neq0}e^{2\pi ik\cdot x}\htw(k,y).\]
    Since 
    \[\FJ\FA_{\eff}\FJ^{-1}\tw_==0,\]
    the effective dissipation operator $\mathcal L_{\eff}$ does not produce enhanced dissipation on the zeroth mode $\tw_=$. Therefore, we shall treat the zeroth tangential mode and the non-zero modes separately. Notice that the zeroth modes of $u$, $w$, and $\tw$ are shear flows. Consequently, the nonlinearity 
    \[\NN:=(u\cdot \nabla)\tw-(\tw\cdot\nabla)u\]
    and the remainder $\RR[w]$ in equation \eqref{DGWP1} contain no zero--zero nonlinear interaction. Thus, the possible growth of $\tw_=$ is driven only by interactions involving the non-zero modes $\tw_{\neq}$. This structure allows us to control $\tw_=$ through the exponential stability estimate for the non-zero modes $\tw_{\neq}$.

    \subsubsection{Linear decay estimate for the non-zero modes}
    We establish a semigroup estimate for $e^{t\LLL_{\eff}}\PPP_{\neq}$, where $\PPP_{\neq}$ denotes the projection onto the non-zero modes
    \[\PPP_{\neq}f:=f-\PPP_= f,\qquad \PPP_= f:=\int_{\T^2} f(x,y)\ddd x.\]
    To this end, for each $k\neq0$, we define 
    \[H_{\curl,k}:=\{f=f(y)\in L^2| 2\pi ik\cdot f_h+f_3'=0\},\]
    and 
    \[\TIH_{\curl,k}:=\{f=f(y)\in L^2| 2\pi ik\cdot f_h+f_3'=\alpha(1-2y)f_3\}.\]
    Let $\FJ_k,\Pi_k,\BB_k,\Delta_k$ be the operators induced by $\FJ,\Pi,\BB,\Delta$ on the $k$-th tangential Fourier mode, respectively. More precisely, they are defined by
    \[\FJ(fe^{2\pi ik\cdot x})=e^{2\pi ik\cdot x}\FJ_k f,\qquad \Pi(fe^{2\pi ik\cdot x})=e^{2\pi ik\cdot x}\Pi_k f,\qquad \BB(fe^{2\pi ik\cdot x})=e^{2\pi ik\cdot x}\BB_k f,\]
    and 
    \[\Delta_k f:=-4\pi^2|k|^2f+f''.\]
    These operators are well-defined, since $\FJ$, $\Pi$, $\BB$, $\Delta$ commute with the tangential Fourier transform. In particular, they inherit the bounds of the corresponding full operators, e.g.
    \begin{align}
        \label{DGWP4-1}\sup_{k\neq0}&\left(\|\FJ_k\|_{\LL(H_{\curl,k};\TIH_{\curl,k})}+\|\FJ_k^{-1}\|_{\LL(\TIH_{\curl,k};H_{\curl,k})}\right)\notag\\&\qquad\qquad\qquad\qquad\qquad+\sup_{k\neq0}\left(\|\Pi_k\|_{\LL(L^2_y;H_{\curl,k})}+\|\BB_k\|_{\LL(H_{\curl,k};H^1_y)}\right)<\infty.
    \end{align}
    Let
    \[\MM_k:=\diag\left\{\frac{1}{4}I_2+\frac{1}{2}\frac{k\otimes k}{|k|^2},1\right\},\]
    and 
    \[\LLL_{\eff,k}:=\nu\FJ_k\Delta_k\FJ_k^{-1}-4\pi^2|k|^2\kappa\FJ_k\Pi_k\MM_k\FJ_k^{-1}.\]
    By construction, $\LLL_{\eff,k}$ is the restriction of $\LLL_{\eff}$ to the $k$-th mode. The following resolvent estimate holds for $\LLL_{\eff,k}$.
    \begin{prop}
        There are $C_\alpha,c_\alpha>0$ depending only on $\alpha$ such that for any $\kappa\ge C_\alpha\nu$ and $k\neq 0$, 
        \begin{align}
            \label{DGWP5}\sup_{\realpart\lambda\ge -c_\alpha\kappa|k|^2}\| (\lambda-\LLL_{\eff,k})^{-1}\|_{\LL(\TIH_{\curl,k})}\le \frac{C_\alpha}{\kappa}.
        \end{align}
    \end{prop}
    \begin{proof}
        By rescaling, it suffices to prove 
        \begin{align}
            \label{DGWP6}\sup_{\realpart\lambda\ge -c_\alpha|k|^2}\| (\lambda-\FG_{\epsilon,k})^{-1}\|_{\LL(\TIH_{\curl,k})}\le C_\alpha
        \end{align}
        for any $\epsilon\in(0,C^{-1}_\alpha]$ and $k\neq0$, where
        \[\FG_{\epsilon,k}:=\epsilon\FJ_k\Delta_k\FJ_k^{-1}-4\pi^2|k|^2\FJ_k\Pi_k\MM_k\FJ_k^{-1}\]
        is the normalized effective dissipation operator. The proof is divided into the following steps.

        \noindent\textbf{Step I.} First, we establish an auxiliary estimate. Let $\{g_n\}_{n\ge 1}\subset \TIH_{\curl,k}$ satisfy \begin{align}
            \label{DGWP7}\|g_n\|=1 \qquad\mbox{and}\qquad g_n\rightharpoonup 0,\qquad \mbox{as $n\to\infty$.}
        \end{align}
        Then,
        \begin{align}
            \label{DGWP8}\liminf_{n\to\infty}\realpart\langle \FJ_k\Pi_k\MM_k\FJ_k^{-1}g_n,g_n\rangle\ge \frac{1}{4}.
        \end{align}
        To see this, notice that $\BB_k$ is a compact linear operator from $H_{\curl,k}$ to $L^2$. Consequently, for $\bg_n:=\FJ_k^{-1}g_n$, one has
        \[\bg_n-g_n=(I-\FJ_k)\FJ_k^{-1}g_n=\BB_k\FJ_k^{-1}g_n\to0\]
        strongly in $L^2$, as $n\to\infty$. This implies
        \begin{align*}
            \lim_{n\to\infty}\|\bg_n\|=1.
        \end{align*}
        Combining this with the compactness of $\BB_k$ and the fact that $g_n\rightharpoonup 0$ as $n\to\infty$,  it follows
        \begin{align}\label{DGWP8-1}
           \notag \langle \FJ_k\Pi_k\MM_k\FJ_k^{-1}g_n,g_n\rangle&=\langle \FJ_k\Pi_k\MM_k\bg_n,\FJ_k\bg_n\rangle\\\notag&=\langle \Pi_k\MM_k\bg_n,\bg_n\rangle-\langle \BB_k\Pi_k\MM_k\bg_n,\bg_n\rangle-\langle \FJ_k\Pi_k\MM_k\bg_n,\BB_k\bg_n\rangle\\&=\langle \MM_k\bg_n,\bg_n\rangle+o(1).
        \end{align}
        As $\MM_k$ is orthogonally similar to $\diag\{\frac{3}{4},\frac{1}{4},1\}$, this yields \eqref{DGWP8}.

        \noindent\textbf{Step I\!I.} In this step, we establish the resolvent estimate \eqref{DGWP6} with large $|\lambda|$. Consider the resolvent equation 
        \begin{align}
            \label{DGWP9}
            f&=(\lambda-\FG_{\epsilon,k})g=\lambda g-\epsilon\FJ_k\Delta_k\FJ_k^{-1}g+4\pi^2|k|^2\FJ_k\Pi_k\MM_k\FJ_k^{-1}g\notag\\&=\lambda g-\epsilon\Delta_kg-\epsilon[\FJ_k,\partial_y^2]\FJ_k^{-1}g+4\pi^2|k|^2\FJ_k\Pi_k\MM_k\FJ_k^{-1}g.
        \end{align}
        Here, we view $\Delta_k$ as an operator from 
        \[\DDDD(\Delta_k):=\{g\in H^2_y| \mbox{$g$ satisfies the homogeneous boundary condition \eqref{DGWP1}$_2$}\}\]
        to $L^2_y$. Since
        \begin{align*}
            \langle -\Delta_k g,g\rangle=4\pi^2|k|^2 \|g\|^2+\|g'\|^2+\alpha|g_3(0)|^2+\alpha|g_3(1)|^2\ge (4\pi^2|k|^2 +C_\alpha)\|g\|^2,
        \end{align*}
        one has
        \begin{align}
            \label{DGWP10}\sigma(-\Delta_k)\subset (4\pi^2,\infty).
        \end{align}
        On the other hand, by the spectral theorem for the positive self-adjoint operator $-\Delta_k$, one has 
        \[\|(\lambda-\epsilon\Delta_k)^{-1}\|_{\LL(L_y^2)}=\sup_{j\ge 1}|\lambda+\epsilon\fb_j|^{-1},\]
        where $\fb_j$ are the eigenvalues of $-\Delta_k$. This, together with \eqref{DGWP10}, gives
        \begin{align}
            \label{DGWP11}
            \|(\lambda-\epsilon\Delta_k)^{-1}\|_{\LL(L_y^2)}\le \sup_{\fb>4\pi^2}|\lambda+\epsilon\fb|^{-1}.
        \end{align}
        Now assume that $\realpart \lambda\ge-\frac{\pi^2}{2}|k|^2$ and $|\lambda|\ge 2\pi^2|k|^2.$ Notice that if $|\imaginarypart \lambda|>\frac{1}{2}|\lambda|$, then 
        \begin{align*}
            |\lambda+\epsilon\fb|^{-1}\le |\imaginarypart \lambda|^{-1}\le 2|\lambda|^{-1};
        \end{align*}
        if $|\realpart \lambda|>\frac{1}{2}|\lambda|$, then $|\realpart \lambda|>\pi^2|k|^2$, which implies $\realpart \lambda>0$ and thus
        \begin{align*}
            |\lambda+\epsilon\fb|^{-1}\le (\realpart \lambda+\epsilon\fb)^{-1}\le 2|\lambda|^{-1}.
        \end{align*}
        As $|\realpart\lambda|+|\imaginarypart\lambda|\ge |\lambda|$, one of the above two cases must hold. Consequently, from \eqref{DGWP11}, it follows
        \begin{align*}
            \|(\lambda-\epsilon\Delta_k)^{-1}\|_{\LL(L_y^2)}\le 2|\lambda|^{-1},
        \end{align*}
        which, together with Propositions \ref{propB1} and \ref{propB2}, leads to
        \begin{align*}
            \|g\|&=\|(\lambda-\epsilon\Delta_k)^{-1}(f+\epsilon[\FJ_k,\partial_y^2]\FJ_k^{-1}g-4\pi^2|k|^2\FJ_k\Pi_k\MM_k\FJ_k^{-1}g)\|\\&\le 2|\lambda|^{-1}\|f\|+ 2|\lambda|^{-1}\|\epsilon[\FJ_k,\partial_y^2]\FJ_k^{-1}g-4\pi^2|k|^2\FJ_k\Pi_k\MM_k\FJ_k^{-1}g\|\\&\le 2|\lambda|^{-1}\|f\|+ C_\alpha|\lambda|^{-1}(|k|^2+1)\|g\|.
        \end{align*}
        By taking
        \begin{align}
            \label{DGWP12}\realpart \lambda\ge-\frac{\pi^2}{2}|k|^2,\qquad |\lambda|\ge (2\pi^2+4C_\alpha)|k|^2,
        \end{align}
        in the above estimate, one gets the a priori estimate
        \begin{align*}\|g\|\le 4|\lambda|^{-1}\|f\|,
        \end{align*}
        which ensures 
        \begin{align}
            \label{DGWP13}\|(\lambda-\FG_{\epsilon,k})^{-1}\|_{\LL(\TIH_{\curl,k})}\le 4|\lambda|^{-1}
        \end{align}
        for any $\epsilon\in(0,1]$ and $k\neq0$.

        \noindent\textbf{Step I\!I\!I.} By contradiction, suppose that \eqref{DGWP6} fails to hold. Then, for any $c_\alpha\in(0,\frac{\pi^2}{2})$ and $n\ge 1$, there are $\epsilon_n,k_n,\lambda_n$ and functions $f_n,g_n$ satisfying
        \begin{align}
            \label{DGWP14}\epsilon_n\in\left(0,\frac{1}{2n}\right],\qquad k_n\neq0,\qquad \realpart \lambda_n \ge -c_\alpha|k_n|^2,
        \end{align}
        and
        \begin{align}
            \label{DGWP15} \|f_n\|\le\frac{1}{n},\qquad \|g_n\|=1,\qquad (\lambda_n-\FG_{\epsilon_n,k_n})g_n=f_n.
        \end{align}
        We claim that $\sup_{n\ge 1}\frac{|\lambda_n|}{|k_n|^2}<\infty$. If not, passing to a subsequence if necessary, one has
        \[\lim_{n\to\infty}|\lambda_n|\ge\lim_{n\to\infty}\frac{|\lambda_n|}{|k_n|^2}=\infty.\]
        Together with \eqref{DGWP13}, this implies 
        \[\| g_n\|\le \frac{4}{n}|\lambda_n|^{-1}\to0,\]
        which contradicts \eqref{DGWP15}. Therefore, without loss of generality, one may assume $\frac{\lambda_n}{|k_n|^2}\to \mu$ for some $\mu$ satisfying
        \begin{align}
            \label{DGWP16}
            \realpart \mu\ge -c_\alpha.
        \end{align}
        
        Next, we show that $\sup_{n\ge 1}|k_n|<\infty.$ To see this, we need the following auxiliary estimate 
        \begin{align}
            \label{DGWP17}\|\BB_k\|_{\LL(H_{\curl,k};L^2_y)}\lesssim_\alpha |k|^{-1}.
        \end{align}
        Indeed, since $\BB_k$ is the restriction of $\BB$ to the $k$-th mode, by applying Corollary \ref{coroA4}, one has 
        \begin{align*}
            |k|\|\BB_k \tilde g\|\lesssim\|\BB(\nabla_x(\tilde ge^{2\pi i k\cdot x}))\|\lesssim_\alpha \|\tilde g\|,
        \end{align*}
        which ensures \eqref{DGWP17}. Now let us suppose by contradiction that $|k_n|\to\infty$. Then, from the resolvent equation
        \begin{align}
            \label{DGWP17-1}\lambda_n g_n-\epsilon_n\Delta_kg_n-\epsilon_n[\FJ_k,\partial_y^2]\FJ_k^{-1}g_n+4\pi^2|k_n|^2\FJ_k\Pi_k\MM_k\FJ_k^{-1}g_n=f_n,
        \end{align}
        it follows
        \begin{align}
            \label{DGWP18}\frac{\lambda_n}{|k_n|^2}  g_n-\frac{\epsilon_n}{|k_n|^2}\Delta_kg_n-\frac{\epsilon_n}{|k_n|^2}[\FJ_k,\partial_y^2]\FJ_k^{-1}g_n+4\pi^2\FJ_k\Pi_k\MM_k\FJ_k^{-1} g_n=\frac{f_n}{|k_n|^2}.
        \end{align}
        Here, to avoid double subscripts, we omit the dependence of $k$ in $\FJ_k,\Delta_k,\Pi_k,\MM_k$ on $n$. Repeating the derivation of \eqref{DGWP8-1} and using \eqref{DGWP4-1}, \eqref{DGWP15}, and \eqref{DGWP17}, it follows
        \begin{align*}
            \realpart\langle\FJ_k\Pi_k\MM_k\FJ_k^{-1} g_n, g_n\rangle=\realpart\langle\MM_k\FJ_k^{-1} g_n,\FJ_k^{-1} g_n\rangle+o(1),
        \end{align*}
        which implies
        \begin{align}
            \label{DGWP19}
            \liminf_{n\to\infty}4\pi^2\realpart\langle\FJ_k\Pi_k\MM_k\FJ_k^{-1} g_n, g_n\rangle\ge\pi^2\|\FJ_k^{-1} g_n\|^2\ge c^*_\alpha
        \end{align}
        for some constant $c^*_\alpha>0$, where the last inequality follows from the uniform boundedness of $\FJ_k$ and $\FJ_k^{-1}$ in \eqref{DGWP4-1}. Moreover, by applying Proposition \ref{propB2}, one has
        \begin{align}\label{DGWP20}
        \frac{\epsilon_n}{|k_n|^2}|\langle[\FJ_k,\partial_y^2]\FJ_k^{-1}g_n,g_n\rangle|\le \frac{C_\alpha\epsilon_n}{|k_n|^2}\|g_n\|^2=\frac{C_\alpha\epsilon_n}{|k_n|^2}\to0.
        \end{align}
        Therefore, by taking inner product with $ g_n$ on both sides of \eqref{DGWP18}, sending $n\to\infty$, and using \eqref{DGWP14}, \eqref{DGWP16}, \eqref{DGWP19}, and \eqref{DGWP20}, it follows
        \[c^*_{\alpha}-c_\alpha\le \realpart\mu+c^*_{\alpha}\le 0.\]
        By taking $c_\alpha<c_\alpha^*$, one gets a contradiction.

        \noindent\textbf{Step I\!V.} Since $\{k_n\}_{n\ge1}$ is bounded, we may pass to a subsequence and assume that $k_n\equiv k$ for some fixed $k\neq0$. Therefore, one has $\lambda_n\to \mu|k|^2$ with $\mu$ satisfying \eqref{DGWP16}. Since $g_n$ is uniformly bounded in $L^2$, there is a function $g_\infty$ such that 
        \[g_n\rightharpoonup g_\infty.\]
        We claim that $g_\infty\neq0$. Indeed, by contradiction, suppose that $g_\infty=0$. Applying the standard energy estimate to the resolvent equation \eqref{DGWP17-1}, it follows
        \begin{align}\label{DGWP21}
            \realpart \lambda_n \|g_n\|^2+\epsilon_n\|g'_n\|^2+4\pi^2|k|^2\epsilon_n\|g_n\|^2+&4\pi^2|k|^2\realpart\langle\FJ_k\Pi_k\MM_k\FJ_k^{-1}g_n,g_n\rangle\notag\\&\le (\|f_n\|+\epsilon_n\|[\FJ_k,\partial_y^2]\FJ_k^{-1}g_n\|)\|g_n\|,
        \end{align}
        which, together with sending $n\to\infty$, and using \eqref{DGWP8}, \eqref{DGWP14}, \eqref{DGWP15}, and Proposition \ref{propB2}, yields
        \[(\pi^2-c_\alpha)|k|^2=\realpart \mu|k|^2+\pi^2|k|^2\le 0.\]
        As $c_\alpha<\frac{\pi^2}{2}\wedge c^*_\alpha$, one gets a contradiction. 

        Now we are in a position to conclude the proof. Indeed, from \eqref{DGWP21}, it follows
        \[\epsilon_n\|g'_n\|^2\le 4\pi^2|k|^2|\langle\FJ_k\Pi_k\MM_k\FJ_k^{-1}g_n,g_n\rangle|+(\|f_n\|+\epsilon_n\|[\FJ_k,\partial_y^2]\FJ_k^{-1}g_n\|)\|g_n\|\le C_{\alpha,k}.\]
        Thus, for any test function $\varphi\in C_0^\infty(0,1)$
        \begin{align*}
            \epsilon_n|\langle -\Delta_k g_n, \varphi\rangle|&\le \epsilon_n(4\pi^2|k|^2\|g_n\|\|\varphi\|+\|g'_n\|\|\varphi'\|)\\&\le 4\pi^2|k|^2\epsilon_n\|\varphi\|+\sqrt{\epsilon_n}C_{\alpha,k}\|\varphi'\|\to0.
        \end{align*}
        Notice that 
        \[f_n+\epsilon_n[\FJ_k,\partial_y^2]\FJ_k^{-1}g_n\to0\qquad\mbox{in $L^2$},\]
        and 
        \[\lambda_n g_n+4\pi^2|k|^2\FJ_k\Pi_k\MM_k\FJ_k^{-1}g_n\rightharpoonup\mu|k|^2g_\infty+4\pi^2|k|^2\FJ_k\Pi_k\MM_k\FJ_k^{-1}g_\infty\qquad\mbox{in $L^2$}.\]
        By taking $n\to\infty$ in the resolvent equation \eqref{DGWP17-1}, it follows
        \[\mu|k|^2g_\infty+4\pi^2|k|^2\FJ_k\Pi_k\MM_k\FJ_k^{-1}g_\infty=0,\]
        where the equality holds in the sense of distributions. Moreover, by density argument, the above equation holds in $\TIH_{\curl,k}$, as the left-hand side belongs to $\TIH_{\curl,k}$. This yields
        \[\mu|k|^2\FJ_k^{-1}g_\infty+4\pi^2|k|^2\Pi_k\MM_k\FJ_k^{-1}g_\infty=0,\]
        and thus 
        \[\mu|k|^2\in\sigma(-4\pi^2|k|^2\Pi_k\MM_k).\]
        However, for any $\phi\in H_{\curl,k}$, one has 
        \[\realpart\langle -4\pi^2|k|^2\Pi_k\MM_k\phi,\phi\rangle\le -\pi^2|k|^2\|\phi\|^2,\]
        which yields
        \[\mu|k|^2\in\sigma(-4\pi^2|k|^2\Pi_k\MM_k)\subset (-\infty,-\pi^2|k|^2].\]
        This contradicts the fact $\realpart \mu \ge -c_\alpha>-\frac{\pi^2}{2}$ and thus completes the proof.
    \end{proof}
    Together with the Parseval theorem, the above resolvent estimate ensures
    \[\sup_{\realpart\lambda\ge -c_\alpha\kappa}\|(\lambda-\LLL_{\eff})^{-1}\PPP_{\neq}\|_{\LL(\TIH_{\curl})}\le \frac{C_\alpha}{\kappa}.\]
    Combining this  with the exponential stability criterion in Theorem~\ref{TheoremES}, the following decay estimate for the semigroup $e^{t\LLL_{\eff}}\PPP_{\neq}$ holds. The proof is identical to that of Proposition~\ref{GWPprop3}, and is therefore omitted.
    \begin{prop}Let $\alpha,\nu>0$. There exist constants $C_\alpha,c_\alpha>0$ depending only on $\alpha$ such that for any $\kappa\ge C_\alpha\nu$,
    \[\|e^{t\LLL_{\eff}}\PPP_{\neq}\|_{\LL(\TIH_{\curl})}\le C_\alpha e^{-c_\alpha\kappa t}.\]
    \end{prop}
    As a direct consequence, we obtain the following linear estimate for the non-zero modes in the anisotropic space $\XX^m$. The proof is identical to that of Corollary~\ref{NDGWPcoro1}; we therefore omit the details and refer the reader to Appendix~\ref{Duhamel}.
    \begin{coro}\label{GWPcoro8}
        Let $\alpha,\nu>0$, $m$ be a positive integer, and $f\in L^2_{\loc}(0,\infty;\TIH_{\curl})$. Suppose that $\tw$ solves 
        \begin{align*}
        \begin{cases}
            \partial_t \tw=\LLL_{\eff} \tw+\PPP_{\neq} f,
            \\(\partial_y \tw_3-\alpha \tw_3)|_{y=0}=(\partial_y \tw_3+\alpha\tw_3)|_{y=1}=\tw_h|_{y=0,1}=0,\\
            \tw|_{t=0}=\PPP_{\neq} \tw_{\ini}.
        \end{cases}
    \end{align*}
    Then, there is a constant $C_{\alpha}>0$ depending only on $\alpha$ such that for any $\kappa\ge C_\alpha \nu$ and $T>0$,
    \begin{align*}
        \sup_{t\le T}\|\tw(t)\|^2_{m,0}+\kappa\int_0^T\|\nabla_x \tw\|^2_{m,0}\ddd t+&\nu \int_0^T\|\partial_y \tw\|^2_{m,0}\ddd t\\&\le C_\alpha\|\PPP_{\neq} \tw_{\ini}\|^2_{m,0}+\frac{C_\alpha}{\kappa}\int_0^T\|\PPP_{\neq} f\|^2_{m-1,0}\ddd t.
    \end{align*}
    \end{coro}
    
    \subsubsection{Nonlinear estimate}
    In this subsection, we prove the global well-posedness of \eqref{DGWP1} with small initial data. The main result is given below.
    \begin{prop}
        Let $\alpha,\nu>0$ and $m\ge 2$ be an integer. Then, there is a constant $C_{\alpha,\nu,m}>0$ such that if $\kappa\ge C_{\alpha,\nu,m}$ and the initial data $\tw_{\ini}:=(I-\BB)w_{\ini}$ satisfies
        \begin{align*}w_{\ini}\in H_{\curl}\cap \XX^m,\qquad\|w_{\ini}\|_{m,0}\le \frac{\sqrt{\kappa}}{C_{\alpha,\nu,m}},
        \end{align*}
        the equation \eqref{DGWP1} admits a unique global solution satisfying 
        \begin{align*}
            \sup_{t\ge 0}\|\tw(t)\|^2_{m,0}+\kappa\int_0^\infty\|\nabla_x \tw\|^2_{m,0}\ddd t+&\nu \int_0^\infty\|\partial_y \tw\|^2_{m,0}\ddd t\lesssim_{\alpha,\nu,m}\|w_{\ini}\|^2_{m,0}.
        \end{align*}
    \end{prop}
    \begin{proof}
        The proof follows the same strategy as that of Proposition~\ref{GWPprop5}, so we only outline the main steps.

        \noindent\textbf{Step I.} In this step, we estimate the nonlinear terms. First, we decompose the convection terms into
        \begin{align}\label{DGWP22}
            \PPP_{\neq}\NN=\PPP_{\neq}[(u\cdot \nabla)\tw-(\tw\cdot\nabla)u]=\PPP_{\neq}\NN_1+\PPP_{\neq}\NN_2+\PPP_{\neq}\NN_3,
        \end{align}
        where
        \begin{align*}
            \NN_1&:=(u_{\neq}\cdot \nabla)\tw_{\neq}-(\tw_{\neq}\cdot\nabla)u_{\neq},\\
            \NN_2&:=(u_{=}\cdot \nabla)\tw_{\neq}-(\tw_{=}\cdot\nabla)u_{\neq},\\\NN_3&:=(u_{\neq}\cdot \nabla)\tw_{=}-(\tw_{\neq}\cdot\nabla)u_{=}.
        \end{align*}
        Repeating the derivation of \eqref{NDGWP41} and \eqref{NDGWP42}, for $\NN_1$, one has
        \begin{align}
            \label{DGWP23}\|\NN_1\|_{m-1,0}\lesssim_{\alpha,m}\|\tw_{\neq}\|_{m,0}\left(\|\tw_{\neq}\|_{m,0}+\|\partial_y\tw_{\neq}\|_{m,0}\right).
        \end{align}
        Notice that $\tw_{=}$ and $u_=$ are shear flows, then it follows
        \begin{align}
            \label{DGWP24}\|\NN_2\|_{m-1,0}\lesssim_{\alpha,m}\|\tw_=\|_{L^2_y}\|\tw_{\neq}\|_{m,0},
        \end{align}
        and
        \begin{align}
            \label{DGWP25}\|\NN_3\|_{m-1,0}\lesssim_{\alpha,m}\|\tw_=\|_{H^1_y}\|\tw_{\neq}\|_{m,0}.
        \end{align}
        Combining \eqref{DGWP22}--\eqref{DGWP25}, one gets
        \begin{align}
            \label{DGWP25-1}
             \|\PPP_{\neq}\NN\|_{m-1,0} \lesssim_{\alpha,m}\|\tw_{\neq}\|_{m,0}\left(\|\tw_{\neq}\|_{m,0}+\|\partial_y\tw_{\neq}\|_{m,0}\right)+\|\tw_=\|_{H^1_y}\|\tw_{\neq}\|_{m,0}.
        \end{align}
        Next, we turn to the remainder term
        \begin{align}
        \label{DGWP26}\RR[w]&=\BB(u\cdot\nabla)w-\BB(w\cdot\nabla)u-(u\cdot\nabla)\BB w+(\BB w\cdot\nabla)u\notag\\&=\RR_1+\RR_2+\RR_3+\RR_4.
        \end{align}
        For $\RR_1$, by using the divergence-free condition for $u_{\neq},u_=$ and the fact that $w_=,u_=$ are shear, it follows
        \begin{align}
            \label{DGWP27}
            &\notag\|\RR_1\|_{m-1,0}\le \|\BB(u_{\neq}\cdot\nabla)w_{\neq}\|_{m-1,0}+\|\BB(u_{=}\cdot\nabla)w_{\neq}\|_{m-1,0}+\|\BB(u_{\neq}\cdot\nabla)w_{=}\|_{m-1,0}\\&\qquad\le \|\BB\dive(u_{\neq}\otimes w_{\neq})\|_{m-1,0}\!+\!\|\BB\dive(u_{=}\otimes w_{\neq})\|_{m-1,0}\!+\!\|\BB(u_{\neq,3}\partial_yw_{=})\|_{m-1,0}\notag\\&\qquad\lesssim_{\alpha,m} \|\tw_{\neq}\|_{m,0}^2+\|\tw_=\|_{L^2_y}\|\tw_{\neq}\|_{m,0}+\|u_{\neq,3}w_{=}\|_{m-1,0}\!+\!\|\BB\dive_h( u_{\neq,h}\otimes w_=)\|_{m-1,0}\notag\\&\qquad\lesssim_{\alpha,m} \|\tw_{\neq}\|_{m,0}^2+\|\tw_=\|_{L^2_y}\|\tw_{\neq}\|_{m,0}.
        \end{align}
        Similarly, 
        \begin{align}
            \label{DGWP28}
            \|\RR_2\|_{m-1,0}\lesssim_{\alpha,m} \|\tw_{\neq}\|_{m,0}^2+\|\tw_=\|_{L^2_y}\|\tw_{\neq}\|_{m,0}.
        \end{align}
        Since $\BB w_=$ is also a shear flow, for $\RR_3$, one has 
        \begin{align}
            \label{DGWP29}
            \notag\|\RR_3\|_{m-1,0}&\le \|(u_{\neq}\cdot\nabla)\BB w_{\neq}\|_{m-1,0}+\|(u_{=}\cdot\nabla)\BB w_{\neq}\|_{m-1,0}+\|(u_{\neq}\cdot\nabla)\BB w_{=}\|_{m-1,0}\\&\lesssim_{\alpha,m} \|\tw_{\neq}\|_{m,0}^2+\|\tw_=\|_{L^2_y}\|\tw_{\neq}\|_{m,0}+\|u_{\neq,3}\partial_y\BB w_=\|_{m-1,0}\notag\\&\lesssim_{\alpha,m} \|\tw_{\neq}\|_{m,0}^2+\|\tw_=\|_{L^2_y}\|\tw_{\neq}\|_{m,0}.
        \end{align}
        Proceeding similarly, for $\RR_4$, one gets
        \begin{align}
            \label{DGWP30}
            \|\RR_4\|_{m-1,0}\lesssim_{\alpha,m} \|\tw_{\neq}\|_{m,0}^2+\|\tw_=\|_{L^2_y}\|\tw_{\neq}\|_{m,0}.
        \end{align}
        By summarizing the estimates \eqref{DGWP26}--\eqref{DGWP30}, it follows
        \begin{align}
            \label{DGWP31}
            \|\PPP_{\neq} \RR[w]\|_{m-1,0}\lesssim_{\alpha,m} \|\tw_{\neq}\|_{m,0}^2+\|\tw_=\|_{L^2_y}\|\tw_{\neq}\|_{m,0}.
        \end{align}

        \noindent\textbf{Step I\!I.} Now we estimate the zeroth mode. Since 
        \begin{align}
            \label{DGWP31-1}\tw_{=}=\FJ_0 w_{=}=\FJ_0(\partial_yu_{=,h}^\perp,0),
        \end{align}
        we establish the estimate for $\tw_{=}$ through the velocity formulation of \eqref{DGWP1}
        \begin{align*}\begin{cases}
            \partial_tu+(u\cdot \nabla)u+\nabla p+(c_1(t),c_2(t),0)=\nu \Delta u+\kappa\curl^{-1}\FA_{\eff}\curl u,
            \\(\partial_y u_h-\alpha u_h)|_{y=0}=(\partial_y u_h+\alpha u_h)|_{y=1}=u_3|_{y=0,1}=0,\\
            u|_{t=0}=u_{\ini}.
        \end{cases}
        \end{align*}
        In particular, by applying the projection $\PPP_=$ and using the fact that $\curl^{-1}\FA_{\eff}$ commutes with $\PPP_=$, the equation for $u_{=,h}$ is given by 
        \begin{align}\label{DGWP32}\begin{cases}
            \partial_tu_{=,h}+\PPP_=\partial_y(u_{\neq,3}u_{\neq,h})+(c_1(t),c_2(t))=\nu \partial_y^2u_{=,h},
            \\(\partial_y u_{=,h}-\alpha u_{=,h})|_{y=0}=(\partial_y u_{=,h}+\alpha u_{=,h})|_{y=1}=\int_0^1 u_{=,h}\ddd y=0,\\
            u_{=,h}|_{t=0}=\PPP_=u_{\ini,h}.
        \end{cases}
        \end{align}
        By taking $L^2_y$ inner product with $u_{=,h}$ on both sides and using the mean-zero condition, it follows
        \begin{align*}
            \frac{1}{2}\frac{\ddd}{\ddd t}\|u_{=,h}\|^2+\nu\|\partial_y u_{=,h}\|^2+\alpha\nu|u_{=,h}|_{y=0}|^2+\alpha\nu|u_{=,h}|_{y=1}|^2=-\langle \partial_y(u_{\neq,3}u_{\neq,h}),u_{=,h}\rangle,
        \end{align*}
        which gives 
        \begin{align}
            \label{DGWP33}
            \nu\int_0^t\|\partial_y u_{=,h}\|^2 \ddd s\le\|\PPP_=u_{\ini,h}\|^2+C_\nu\int_0^t\| \PPP_=(u_{\neq,3}u_{\neq,h})\|^2\ddd s.
        \end{align}
        On the other hand, by taking integration with respect to $y\in(0,1)$ on both sides of \eqref{DGWP32} and using the identity
        \[ \int_0^1\PPP_=\partial_y(u_{\neq,3}u_{\neq,h})\ddd y=\int_{\T^2}\int_0^1\partial_y(u_{\neq,3}u_{\neq,h})\ddd y\ddd x=0,\]
         one has 
        \[(c_1(t),c_2(t))=\nu \int_0^1\partial_y^2u_{=,h}\ddd y=-\alpha\nu(u_{=,h}|_{y=1}-u_{=,h}|_{y=0}).\]
        Therefore, by testing \eqref{DGWP32} against $-\partial_y^2u_{=,h}$ and using the Poincar\'e inequality, it follows
        \begin{align*}
            \frac{1}{2}\frac{\ddd}{\ddd t}&\left(\|\partial_y u_{=,h}\|^2+\alpha|u_{=,h}|_{y=0}|^2+\alpha|u_{=,h}|_{y=1}|^2\right)+\nu\|\partial_y^2u_{=,h}\|^2\\&\qquad\qquad=\langle \partial_y(u_{\neq,3}u_{\neq,h}), \partial^2_y u_{=,h}\rangle+\alpha^2\nu|u_{=,h}|_{y=1}-u_{=,h}|_{y=0}|^2\\&\qquad\qquad\le |\langle \partial_y(u_{\neq,3}u_{\neq,h}), \partial^2_y u_{=,h}\rangle|+C_{\alpha,\nu}\|\partial_y u_{=,h}\|^2,
        \end{align*}
        which, together with \eqref{DGWP33}, implies
        \begin{align*}
            \|\partial_y u_{=,h}(t)\|^2+\nu\int_0^t&\|\partial_y^2u_{=,h}\|^2\ddd s\lesssim_{\alpha,\nu}\|\PPP_=u_{\ini,h}\|^2_{H^1_y}+\int_0^t\|\PPP_=(u_{\neq,3}u_{\neq,h})\|_{H^1_y}^2\ddd s.
        \end{align*}
        Combining this, \eqref{DGWP31-1}, and \eqref{DGWP33}, one gets
        \begin{align}\label{DGWP34}
            \|\tw_{=}(t)\|^2+\nu\int_0^t\|\partial_y\tw_{=}\|^2\ddd s&\lesssim_{\alpha,\nu}\|\tw_{\ini}\|^2+\int_0^t\|\PPP_=(u_{\neq,3}u_{\neq,h})\|_{H^1_y}^2\ddd s\notag\\&\lesssim_{\alpha,\nu}\|\tw_{\ini}\|^2+\int_0^t\|\tw_{\neq}\|_{m,0}^4\ddd s.
        \end{align}

        \noindent\textbf{Step I\!I\!I.} Define
        \begin{align*}
            \EE_m(T):=\sup_{t\le T}\|\tw_{\neq}(t)\|^2_{m,0}+&\sup_{t\le T}\|\tw_{=}(t)\|_{L^2_y}^2+\kappa\int_0^T\|\nabla_x \tw_{\neq}\|^2_{m,0}\ddd t\\&\qquad\qquad+\nu\int_0^T\|\partial_y \tw_{\neq}\|^2_{m,0}\ddd t+\nu\int_0^T\|\partial_y \tw_{=}\|_{L^2_y}^2\ddd t.
        \end{align*}
        By applying Corollary \ref{GWPcoro8} combined with \eqref{DGWP25-1}, \eqref{DGWP31}, and the Poincar\'e inequalities
        \[\|\tw_{\neq}\|_{m,0}\le \|\nabla_x \tw_{\neq}\|_{m,0},\qquad \|\tw_{=,h}\|_{L^2_y}\lesssim\|\partial_y\tw_{=,h}\|_{L^2_y},\]
        it follows
        \begin{align*}
        \sup_{t\le T}\|\tw_{\neq}(t)\|^2_{m,0}\!+\kappa\!\int_0^T\!\|\nabla_x \tw_{\neq}\|^2_{m,0}\ddd t+\nu\! \int_0^T\!\|\partial_y \tw_{\neq}\|^2_{m,0}\ddd t\lesssim_{\alpha,\nu,m} \|\tw_{\ini}\|^2_{m,0}\!+\frac{1}{\kappa}\EE^2_m(T).
    \end{align*}
    Similarly, from \eqref{DGWP34}, one has 
    \[\sup_{t\le T}\|\tw_{=}(t)\|_{L^2_y}^2+\nu\int_0^T\|\partial_y\tw_{=}\|_{L^2_y}^2\ddd t\lesssim_{\alpha,\nu,m} \|\tw_{\ini}\|^2_{m,0}+\frac{1}{\kappa}\EE^2_m(T)\]
    Therefore,
    \[\EE_m(T)\lesssim_{\alpha,\nu,m} \|\tw_{\ini}\|^2_{m,0}+\frac{1}{\kappa}\EE^2_m(T).\]
    This, together with a standard bootstrap argument, completes the proof; see Step I\!I of the proof of Proposition \ref{GWPprop5}.
    \end{proof}
    
    \section{Proof of Theorem \ref{MT}}\label{sec5}
    In this section, we combine the results obtained in Sections~\ref{SLofNS} and~\ref{GWPofDNS} to prove Theorem~\ref{MT}. We prove the result only in the non-degenerate case. The degenerate case is treated in the same way, using the corresponding degenerate versions of the scaling limit result and the deterministic global well-posedness result.

    Let 
    \[M:=\|w_{\ini}\|_{L^\infty(\Omega;H^1\cap \XX^m)}.\]
    Choose $\kappa>0$ sufficiently large so that 
    \[M\le \frac{\nu_\kappa}{2 C_{\alpha,m}},\qquad \nu_{\kappa}:=\nu+\frac{4\kappa}{5},\]
    where $C_{\alpha,m}$ is the constant given in Proposition \ref{GWPprop5}. Then, there is a unique global solution $w^{\det}$ of the limiting equation \eqref{NDGWP1} satisfying
    \begin{align*}
        \sup_{t\ge 0}\|w^{\det}(t)\|^2_{m,0}+\nu_{\kappa}\int_0^\infty\|\nabla w^{\det}(t)\|^2_{m,0}\ddd t\le C^2_\alpha M^2.
    \end{align*}
    In particular, since $m\ge 2$, it follows
    \[\sup_{t\ge 0}\|w^{\det}(t)\|\le C_\alpha M.\]
    Let the cut-off parameter $R>0$ be sufficiently large so that
    \[R\ge 2C_\alpha M.\]
    Therefore, the cut-off function $\eta_R$ is inactive along $w^{\det}$, and the deterministic cut-off solution coincides with the deterministic solution without cut-off, that is,
    \[w^{\det}_{\cut}(t)=w^{\det}(t),\qquad \forall t\in [0,T].\]
    Let $w^N_{\cut}$ be the solution of the cut-off equation \eqref{SL1}, corresponding to the noise coefficients $\theta^N_\cdot$ given in \eqref{SL2}. By Proposition \ref{propSL3}, after choosing $N$ sufficiently large, it follows
    \begin{align*}
        \PP\left\{\sup_{t\le T}\|w_{\cut}^N(t)-w_{\cut}^{\det}(t)\|_{H^{\delta}}<R-C_\alpha M\right\}\ge 1-\epsilon.
    \end{align*}
    Notice that on this event, one has
    \[\sup_{t\le T}\|w_{\cut}^N(t)\|\le \sup_{t\le T}\|w_{\cut}^N(t)-w_{\cut}^{\det}(t)\|_{H^{\delta}}+\sup_{t\le T}\|w_{\cut}^{\det}(t)\|<R.\]
    Define 
    \[\tau^N:=\inf\{t|\|w_{\cut}^N(t)\|\ge R\}.\]
    The preceding estimate implies
    \begin{align}
        \label{NDGWP43}\PP\{\tau^N\ge T\}\ge 1-\epsilon.
    \end{align}
    Moreover, on the time interval $[0,\tau^N)$, the cut-off is inactive, so $w^N_{\cut}$ coincides with a local solution of \eqref{INTRO1}. This, together with \eqref{NDGWP43}, ensures the existence of a solution 
    \[(w,\tau):=(w_{\cut}^N,\tau^N)\]
    of \eqref{INTRO1}, which exists up to $T$ with high probability.

    \appendix 
    \section{The curl-admissible class and Biot--Savart operator in the periodic channel}\label{curl-ad}
    In this section, we collect some basic properties of the curl-admissible class $H_{\curl}$ defined in \eqref{INTRO1-1} and the Biot--Savart operator $\curl^{-1}$. Define 
    \[H_{\curl}^\perp:=\{\nabla \pp| \pp\in H^1,\mbox{ $\pp|_{y=0,1}$ are constants}\}.\]
    \begin{prop} \label{propA1}The following orthogonal decomposition holds
    \[L^2=H_{\curl}\oplus H^\perp_{\curl}.\]
    Moreover, for any $f\in L^2$,
    \begin{align}
        \label{A0}\Pi f=f-\nabla \pp,
    \end{align}
    where $\Pi: L^2\to H_{\curl}$ is the projection operator and $\pp$ satisfies the non-homogeneous Dirichlet--Poisson problem
    \begin{align*}
        \begin{cases}
        \Delta \pp=\dive f,\\
        \pp|_{y=0}=0,\qquad \pp|_{y=1}=\int_D f_3\ddd x\ddd y.
        \end{cases}
    \end{align*}
    \end{prop}
    \begin{proof}
        \textbf{Step I.} Since
        \[\int_0^1\left|\int_{\T^2}w_3(x,y)\ddd x\right|^2\ddd y\lesssim \|w\|^2,\]
        the map 
        \[w\mapsto Tw(y):=\int_{\T^2}w_3(x,y)\ddd x\]
        is bounded from $L^2$ to $L^2_y$, which implies that $H_{\curl}$ is a closed subspace of $L^2$. Fix $w\in H_{\curl}$ and $\nabla \pp\in H_{\curl}^\perp$. According to the definition of $H^\perp_{\curl}$,  
        \[\tilde \pp:=\pp-\pp_1y-(1-y)\pp_0\in H^1_0,\]
        where $\pp_0:=\pp|_{y=0}$ and $\pp_1:=\pp|_{y=1}$ are constants. Integrating by parts and using the definition~\eqref{INTRO1-1} of $H_{\curl}$, one has the orthogonality relation
        \begin{align*}
            \int_D w\cdot\nabla \pp\ddd x\ddd y=\int_D w\cdot\nabla \tilde \pp\ddd x\ddd y+(\pp_1-\pp_0)\int_D w_3\ddd x\ddd y=0.
        \end{align*}

        \noindent\textbf{Step I\!I.} It remains to construct the decomposition \eqref{A0}. Define 
        \[\pp:=\tilde \pp+y\int_D f_3\ddd x\ddd y,\]
        where $\tilde \pp$ solves the Dirichlet--Poisson problem
        \begin{align}\label{A1}
            \begin{cases}
            \Delta\tilde \pp=\dive f,\\
            \tilde \pp|_{y=0,1}=0.
            \end{cases}
        \end{align}
        Then, it follows that $\nabla \pp\in H^\perp_{\curl}$ and $w:=f-\nabla \pp$ satisfies 
        \[\dive w=\dive f-\Delta \pp=\dive f-\Delta \tilde \pp=0.\]
        To verify the horizontal zero-mean condition for $w_3$, notice that for any $\varphi=\varphi(y)\in C_0^\infty(0,1)$, one has 
        \begin{align*}
            0=\int_D \dive w \varphi \ddd x\ddd y=-\int_D w_3\varphi' \ddd x\ddd y=\int_D \partial_y\left(\int_{\T^2}w_3(x,y)\ddd x\right)\varphi(y) \ddd y,
        \end{align*}
        which implies 
        \[\int_{\T^2}w_3(x,y)\ddd x=c,\qquad a.e.\ y\in (0,1)\]
        for some constant $c$. On the other hand, by the Newton--Leibniz formula,
        \[0=\int_D (f_3-\partial_y\pp)\ddd x\ddd y=\int_D w_3\ddd x\ddd y=c.\]
        Therefore, $w:=f-\nabla \pp\in H_{\curl}$ and the proof is complete.
    \end{proof}
    Let $H^{-1,q}$ be defined in \eqref{notation0}. The following estimate plays an important role in the a priori estimate.
    \begin{coro}\label{coroA2}
    Let $q\in(1,\infty)$ and $m\ge 0$ be an integer. For any smooth vector field $f$, it follows that 
    \[\|\Pi f\|_{H^{m,q}}\lesssim_{m,q} \|f\|_{H^{m,q}},\]
    and
    \begin{align}\label{A1-0}
        \|\Pi f\|_{H^{-1,q}}\lesssim_{q} \left(\|f\|_{H^{-1,q}}+\left|\int_D f_3 \ddd x\ddd y\right|\right).
    \end{align}
    \end{coro}
    \begin{proof}\textbf{Step I.} We first address the case $m=0.$ Let $\tilde \pp$ be the solution of \eqref{A1}. By applying odd extension on $\tilde \pp, f_1, f_2$ and even extension on $f_3$, periodizing the resulting equation, and using the Mikhlin multiplier theorem, one has 
    \[\|\nabla \tilde \pp\|_{q}\lesssim_q \|f\|_{q}.\]
    This, together with the decomposition
    \begin{align}\label{A1-1}
        f=\Pi f+\nabla \tilde{\pp}+\left(\int_D f_3 \ddd x\ddd y\right)e_3,
    \end{align}
    implies that
    \[\|\Pi f\|_{q}\le\|f\|_{q}+\|\nabla \tilde \pp\|_{q}+\|f_3\|_{L^1}\lesssim_q \|f\|_{q}.\]

    \noindent\textbf{Step I\!I.} The higher regularity can be obtained through a standard argument. For simplicity, we only discuss the case $m=1$. Taking $\partial_{x_j}$ on both sides of \eqref{A0}, it follows that 
    \[\partial_{x_j}\Pi f=\partial_{x_j}f-\nabla\partial_{x_j}\pp,\]
    where $\partial_{x_j}\pp$ satisfies the Dirichlet--Poisson problem
    \begin{align*}
        \begin{cases}
        \Delta \partial_{x_j}\pp=\dive \partial_{x_j}f,\\
        \partial_{x_j}\pp|_{y=0}=\partial_{x_j}\pp|_{y=1}=0.
        \end{cases}
    \end{align*}
    By the standard elliptic estimate, one has 
    \begin{align*}
        \|\partial_{x_j}\Pi f\|_q\lesssim_q \|f\|_{H^{1,q}}+\|\nabla \partial_{x_j}\pp\|_q\lesssim_q\|f\|_{H^{1,q}}.
    \end{align*}
    Notice that 
    \begin{align*}
        \|\nabla\partial_y \pp\|_q\le \|\nabla_x \partial_y \pp\|_q+\|\partial_y^2\pp\|_q\lesssim\|f\|_{H^{1,q}}+\|\Delta_x \pp\|_q+\|\dive f\|_q\lesssim_q \|f\|_{H^{1,q}},
    \end{align*}
    it follows that 
    \[\|\nabla \Pi f\|_q\le \|\nabla_x\Pi f\|_q+ \|\partial_y\Pi f\|_q\lesssim_q \|f\|_{H^{1,q}}.\]

    \noindent\textbf{Step I\!I\!I.} We turn to \eqref{A1-0}. Using \eqref{A1-1} again, one has
    \begin{align}\label{A2}
        \|\Pi f\|_{H^{-1,q}}\le \|f\|_{H^{-1,q}}+\|\nabla \tilde \pp\|_{H^{-1,q}}+\left|\int_D f_3 \ddd x\ddd y\right|.
    \end{align}
    For any $\psi\in L^{q'}$, consider the following Dirichlet--Poisson problem 
    \begin{align}\label{A3}
        \begin{cases}
            -\Delta\phi=\psi,\\
            \phi|_{y=0,1}=0.
        \end{cases}
    \end{align}
    By duality, it follows that
    \[\|\tilde \pp\|_{q}=\sup_{\|\psi\|_{L^{q'}=1}}|\langle \tilde \pp,\psi\rangle|,\]
    where, by integrating by parts and the standard elliptic regularity theory for \eqref{A3}, one has
    \[|\langle \tilde \pp,\psi\rangle|=|\langle \tilde \pp,\Delta\phi\rangle|=|\langle \dive f,\phi\rangle|=|\langle f,\nabla \phi\rangle|\le \|f\|_{H^{-1,q}}\|\nabla \phi\|_{H^{1,q'}}\lesssim_q\|f\|_{H^{-1,q}}.\]
    Hence,
    \begin{align}\label{A4}
    \|\tilde \pp\|_{q}\lesssim_q\|f\|_{H^{-1,q}}.
    \end{align}
    On the other hand, using the boundary conditions of $\tilde \pp$, one has
    \[\|\nabla \tilde \pp\|_{H^{-1,q}}=\sup_{\|\varphi\|_{H^{1,q'}}=1}|\langle \nabla \tilde \pp, \varphi\rangle|\le \|\tilde \pp\|_{q}\|\dive \varphi\|_{L^{q'}}\lesssim  \|\tilde \pp\|_{q}.\]
    This, together with \eqref{A2} and \eqref{A4}, implies \eqref{A1-0}.
    \end{proof}
    Next, we turn to the basic properties of the Biot--Savart operator $\curl^{-1}$ in $\T^2\times (0,1)$.
    \begin{prop}\label{propA3}Let $w\in H_{\curl}$. Then, there is a unique solution $u$ of the div-curl system
    \begin{align}
        \label{A5}
        \begin{cases}
            \curl u=w,\qquad \dive u=0,\\
            u_3|_{y=0,1}=\int_{D}u_1\ddd x=\int_Du_2\ddd x=0.
        \end{cases} 
    \end{align}
    Moreover, it follows that 
    \begin{align}
        \label{A6}\|u\|_{H^1}\lesssim \|w\|.
    \end{align}
    \end{prop}
    We need the following representation of curl-free vector fields.
    \begin{lemma}\label{lemmaA4}
        For any vector field $z\in L^2$ satisfying $\curl z=0$, there are $\varphi\in H^1$ and $c_h\in \R^2$ such that
        \[z=\nabla\varphi+(c_h,0).\]
    \end{lemma}
    \begin{proof} Consider the tangential Fourier expansion
    \[z=\ssum_{k\in\Z^2} \hat z(k,y) e^{2\pi i k\cdot x}.\]
    Since $z$ is curl-free, then for any $k\in \Z^2$, 
    \begin{align}\label{A9}
        \begin{cases}
        \partial_y\hat{z}_2(k,y)-2\pi i k_2\hat z_3(k,y)=0,\\
        \partial_y\hat{z}_1(k,y)-2\pi i k_1\hat z_3(k,y)=0,\\
        k_1\hat z_2(k,y)-k_2\hat z_1(k,y)=0.
    \end{cases}
    \end{align}
    For any $k\neq 0$, by using \eqref{A9}$_3$, one has
    \begin{align}\label{A11}
        \hat z_h(k,y)=\frac{-ik}{2\pi}\phi(k,y)
    \end{align}
    for some scalar function $\phi$, which, together with \eqref{A9}$_1$, yields
    \begin{align*}
        \hat z_3(k,y)=-\frac{1}{4\pi^2}\partial_y\phi(k,y).
    \end{align*}
    Therefore, 
    \begin{align}\label{A11-1}
        \hat z(k,y) e^{2\pi ik\cdot x}=-\frac{1}{4\pi^2}\nabla (\phi(k,y)e^{2\pi ik\cdot x}).
    \end{align}
    For $k=0$, one has
    \[0=\curl \hat z(0,y)=(-\partial_y\hat z_2(0,y),\partial_y\hat z_1(0,y),0),\]
    which implies $\hat z_h\equiv c_h$ for some $c_h\in\R^2$. Let
    \[\varphi:=-\frac{1}{4\pi^2}\ssum_{k\neq 0}\phi(k,y)e^{2\pi ik\cdot x}+\int_0^y \hat z_3(0,y')\ddd y',\]
    then it follows that 
    \[z=(\hat z_h,0)+\hat z_3 e_3+\nabla\left(-\frac{1}{4\pi^2}\ssum_{k\neq 0}\phi(k,y)e^{2\pi ik\cdot x}\right)=(c_h,0)+\nabla \varphi.\]
    Since $\nabla \varphi\in L^2$, by assuming $\int_D\varphi \ddd x\ddd y=0$ if necessary, from the Poincar\'e inequality, one has $\varphi\in H^1$.
    \end{proof}
    
    Now we prove Proposition \ref{propA3}.
    \begin{proof}[Proof of Proposition \ref{propA3}]\textbf{Step I.} We start with an auxiliary problem. Define
    \begin{align}
        \label{A6-1}V:=\left\{u\in H^1\Big| \dive u=0, u_3|_{y=0,1}=\int_Du_1\ddd x=\int_Du_2\ddd x=0\right\}
    \end{align}
    and the bilinear form
    \[a(u,v):=\int_D\nabla u\cdot\nabla v\ddd x\ddd y,\qquad u,v\in V.\]
    For any $v\in V$, since $v_1,v_2$ have a zero mean and $v_3$ vanishes on the boundary, by applying the Poincar\'e inequality, one has
    \[a(v,v)=\int_D\|\nabla v\|^2 \ddd x\ddd y\ge c\|v\|^2.\]
    It is clear to have the boundedness of $a(u,v)$. Therefore, by applying the Lax--Milgram theorem, for any $w\in H_{\curl}$, there exists a unique element $u\in V$ satisfying \eqref{A6} and
    \[a(u,v)=\int_D w\cdot \curl v \ddd x\ddd y,\qquad \forall v\in V.\]
    This, together with the identity
    \[\int_D\nabla u\cdot\nabla v\ddd x\ddd y=\int_D\curl u\cdot\curl v\ddd x\ddd y,\]
    yields
    \begin{align}\label{A7}
        \int_D(\curl u-w)\cdot \curl v\ddd x\ddd y=0,\qquad \forall v\in V.
    \end{align}
    \noindent\textbf{Step I\!I.} Let $z:=\curl u-w$. It follows that
    \[\dive z= \dive \curl u-\dive w=0.\]
    Moreover, we claim that $z$ is curl-free. Indeed, for any vector field $\psi\in C^\infty_0$, by taking
    \[v=\PPP \psi-\left(\int_D(\PPP\psi)_1\ddd x,\int_D(\PPP\psi)_2\ddd x,0\right)\in V\]
    in \eqref{A7}, where $\PPP$ is the usual Leray projection, one has
    \[0=\int_D z\cdot \curl \psi \ddd x\ddd y=-\int_D \curl z \cdot\psi \ddd x\ddd y.\]
    This implies $\curl z=0$ as claimed. 

    \noindent\textbf{Step I\!I\!I.} As $z$ is curl-free and divergence-free, by Lemma \ref{lemmaA4}, 
    \begin{align}\label{A8-0}
        z=\nabla \varphi+(c_h,0)
    \end{align}
    for some harmonic function $\varphi$ and constant vector field $c_h:=(c_1,c_2)$. We claim that 
    \begin{align}
        \label{A8}\varphi|_{y=0}\equiv C_0,\qquad \varphi|_{y=1}\equiv C_1
    \end{align}
    for some constants $C_0,C_1$. To see this, let $\psi\in C^{\infty}(\T^2)$ and $\vartheta\in C^\infty(0,1)$. By plugging \eqref{A8-0} and
    \[v=\vartheta(y)(-\partial_{x_2}\psi(x),\partial_{x_1}\psi(x),0)\]
    into \eqref{A7}, it follows that
    \begin{align*}
        0&=\int_D \nabla \varphi\cdot\curl v\ddd x\ddd y+\int_D c_h\cdot(\curl v)_h\ddd x\ddd y\\&=\int_D \nabla \varphi\cdot\curl v\ddd x\ddd y-\int_D\vartheta'(c_h\cdot \nabla_x)\psi \ddd x\ddd y=\int_D \nabla \varphi\cdot\curl v\ddd x\ddd y.
    \end{align*}
    This, together with integrating by parts and the fact 
    \[\dive \curl v=0,\qquad (\curl v)_3=\vartheta(y)\Delta_x\psi,\]
    gives
    \begin{align*}
        0=\vartheta(1)\int_{\T^2}  \varphi|_{y=1} \Delta_x\psi\ddd x-\vartheta(0)\int_{\T^2}  \varphi|_{y=0} \Delta_x\psi\ddd x.
    \end{align*}
    Since $\vartheta$ are arbitrary smooth functions, one may take $\vartheta(1)=0$ and $\vartheta(0)=1$, to get
    \[\int_{\T^2}  \varphi|_{y=0} \Delta_x\psi\ddd x=0\]
    for any $\psi \in C^\infty(\T^2)$, which implies $\varphi|_{y=0}$ is a constant and thus \eqref{A8} holds as claimed. Now note that $z-(c_h,0)=\nabla \varphi$, where $\varphi$ is a harmonic function satisfying \eqref{A8} and $z$ fulfills $\dive z=0$ and $\int_{\T^2} z_3\ddd x=0$, then an application of Proposition \ref{propA1} yields
    \[z\equiv (c_h,0).\]
    This, together with taking
    \[v=(c_2(y-1/2),-c_1(y-1/2),0)\]
    in \eqref{A7}, implies 
    \[c_1^2+c_2^2=0,\]
    and $z\equiv 0$. As a result, $\curl u=w$ and $u$ is a solution of \eqref{A5}.
    
    \noindent\textbf{Step I\!V.} It remains to show the uniqueness. Suppose that $u$ is a solution of \eqref{A5} with $w=0$. Then, one has 
    \[\int_{D}\|\nabla u\|^2 \ddd x\ddd y= \int_{D}\|\curl u\|^2 \ddd x\ddd y=0,\]
    which implies that $u$ is a constant vector field. This, combined with \eqref{A5}$_2$, gives $u\equiv 0$.
    \end{proof}
    Define the Biot--Savart operator 
    \begin{align}
        \label{A13}\curl^{-1}: H_{\curl}\to V,\qquad w\mapsto u,
    \end{align}
    where $u$ is the unique solution of \eqref{A5}, and $V$ is defined in \eqref{A6-1}. 
    \begin{coro}\label{coroA4}
        Let $q\in(1,\infty)$ and $m\ge -1$ be an integer. For any $w\in H_{\curl}\cap C^\infty$, it follows that 
        \begin{align}\label{A14}
            \|\curl^{-1} w\|_{H^{m,q}}\lesssim_{m,q} \|w\|_{H^{m-1,q}}.
        \end{align}
    \end{coro}
    \begin{proof}\textbf{Step I.} We begin with the case $m=1$. Let $u=\curl^{-1}w$ be the solution of \eqref{A5}. Applying even extension on $u_1,u_2,w_3$ and odd extension on $w_1,w_2,u_3$, periodizing the resulting equation, and using the boundedness of the Riesz transform, it follows that 
    \[\|\nabla u\|_q\lesssim_q\|w\|_q.\]
    This, combined with the Poincar\'e inequality, leads to 
    \[\|u\|_{H^{1,q}}\lesssim_q \|w\|_q.\]

    \noindent\textbf{Step I\!I.} We use a duality argument to address the case $m=0$. For any vector field $\psi\in L^{q'}$, consider the following Neumann--Poisson problem
        \begin{align}\label{A15}
            \begin{cases}
            -\Delta\phi =\psi-\left(\int_D \psi_1\ddd x,\int_D\psi_2\ddd x,0\right)=:\psi-\psi_{h,D},\\
            \partial_y \phi_1|_{y=0,1}=\partial_y \phi_2|_{y=0,1}=\phi_3|_{y=0,1}=0.
        \end{cases}
        \end{align}
        The above problem is solvable, since $(\psi-\psi_{h,D})_h$ has zero mean. By duality and \eqref{A5}$_2$, it follows that
        \[ \|u\|_{q}=\sup_{\|\psi\|_{q'}=1}\left|\int_D u\cdot\psi \ddd x\ddd y\right|=\sup_{\|\psi\|_{q'}=1}\left|\int_D u\cdot(\psi-\psi_{h,D}) \ddd x\ddd y\right|,\]
        where, by applying \eqref{A15}, the identity
        \[-\Delta\phi=\curl (\curl\phi)-\nabla(\dive \phi),\]
        and the fact that $n\times u$ is orthogonal to $\curl\phi$ on $\partial D$, one has 
        \begin{align*}
           \left|\int_D u\cdot(\psi-\psi_{h,D}) \ddd x\ddd y\right|&=|\langle u, \Delta \phi\rangle|=\!\left|\int_D u\cdot\curl (\curl\phi) \ddd x\ddd y-\int_D u\cdot\nabla (\dive \phi) \ddd x\ddd y\right|\\&=\left|\int_D\curl u\cdot \curl \phi \ddd x\ddd y-\int_{\{y=0,1\}}(n\times u)\cdot \curl \phi \ddd x\right|\\&=\left|\int_D w\cdot \curl \phi \ddd x\ddd y\right|\le \|w\|_{H^{-1,q}}\|\phi\|_{H^{2,q'}}.
        \end{align*}
        This, combined with the standard elliptic estimate
        \[\|\phi\|_{H^{2,q'}}\lesssim_q\|\psi-\psi_{h,D}\|_{q'}\lesssim_q\|\psi\|_{q'},\]
        implies \eqref{A14} with $m=0$. The case $m=-1$ can be obtained in the same way.
    
    \noindent\textbf{Step I\!I\!I.} As in the proof of Corollary \ref{coroA2}, we only discuss the case $m=2$. Taking $\partial_{x_j}$ on both sides of \eqref{A5}, it follows that 
    \begin{align*}
        \begin{cases}
            \curl \partial_{x_j}u=\partial_{x_j}w,\qquad \dive \partial_{x_j}u=0,\\
            \partial_{x_j}u_3|_{y=0,1}=\int_{D}\partial_{x_j}u_1\ddd x=\int_D\partial_{x_j}u_2\ddd x=0.
        \end{cases} 
    \end{align*}
    Applying \eqref{A14} with $m=1$, one has 
    \[\|\partial_{x_j}u\|_{H^{1,q}}\lesssim \|\partial_{x_j}w\|_{L^q}\lesssim\|w\|_{H^{1,q}},\]
    which, combined with the fact that
    \[\partial_y^2 u_1=\partial_y w_2+\partial_y\partial_{x_1}u_3,\quad\partial_y^2 u_2=-\partial_y w_1+\partial_y\partial_{x_2}u_3,\quad \partial_y^2u_3=-\partial_y(\partial_{x_1}u_1+\partial_{x_2}u_2),\]
    gives 
    \[\|u\|_{H^{2,q}}\lesssim\|w\|_{H^{1,q}}\]
    as desired.
    \end{proof}

    \section{Basic properties of the boundary correction operator \texorpdfstring{$\BB$}{B}}\label{BCO}
    In this section, we collect some basic properties of the boundary correction operator $\BB$ defined in \eqref{LS3}. First, we prove the invertibility of $I-\BB$. Define the corrected vorticity class
    \begin{align}
        \label{BDYY1}\TIH_{\curl}:=\bigg\{\tw\in L^2\bigg| \dive \tw=\alpha(1-2y)\tw_3,\ \int_{\T^2} \tw_3(x,y)\ddd x=0\mbox{ a.e. $y\in(0,1)$}\bigg\}
    \end{align}
    endowed with the usual $L^2$-norm. Here, the restriction on the divergence of $\tw$ is understood in the sense of distributions.
    \begin{prop}\label{propB1}
        Let $s\ge0$, $q\in(1,\infty)$. Then, $I-\BB$ is a topological isomorphism from $H_{\curl}\cap H^{s,q}$ to $\TIH_{\curl}\cap H^{s,q}$.
    \end{prop}
    \begin{proof}\textbf{Step I.} Let $w\in H_{\curl}$ and $\tw=(I-\BB)w$. Notice that
    \[\dive \BB w=-\alpha(1-2y)w_3,\qquad \tw_3=w_3,\]
    then one has $\dive \tw =\alpha(1-2y)\tw_3$ and $\int_{\T^2}\tw_3\ddd x=0$ for almost every $y\in (0,1)$, which implies $\tw\in \TIH_{\curl}$. Thus, for any $s\ge 0$, $q\in(1,\infty)$,
    \[I-\BB\in\LL(H_{\curl}\cap H^{s,q};\TIH_{\curl}\cap H^{s,q}).\]
    
    \noindent\textbf{Step I\!I.} We show that $I-\BB$ can be expressed as a compact perturbation of an isomorphism. Indeed, for $w\in H_{\curl}$, define $\Phi[w]$ as the solution of the Neumann--Poisson problem
    \begin{align}\label{BDYY2}
            \Delta\Phi[w]-\alpha(1-2y)\partial_y\Phi[w]=\alpha(1-2y)w_3,\qquad \partial_y\Phi[w]|_{y=0,1}=0.
    \end{align}
    Since $\int_{\T^2} w_3\ddd x=0$ for almost every $y\in (0,1)$, the above problem is solvable. Let 
    \[\TT w:=w+\nabla \Phi[w].\]
    Notice that 
    \[\dive (\TT w)=\Delta\Phi[w]=\alpha(1-2y)(w_3+\partial_y\Phi[w])=\alpha(1-2y)(\TT w)_3,\]
    and 
    \[\int_{\T^2} (\TT w)_3\ddd x=\partial_y\int_{\T^2}\Phi[w]\ddd x.\]
    Since $\overline{\Phi[w]}:=\int_{\T^2}\Phi[w]\ddd x$ satisfies
    \[\partial_y^2\overline{\Phi[w]}-\alpha(1-2y)\partial_y\overline{\Phi[w]}=0,\qquad \partial_y\overline{\Phi[w]}|_{y=0,1}=0,\]
    one has $\partial_y\overline{\Phi[w]}\equiv 0$ and thus $\TT w\in \TIH_{\curl}$. Moreover, the standard elliptic estimate for \eqref{BDYY2} gives 
    \begin{align}
        \label{BDYY4}\|\nabla \Phi[w]\|_{H^{s+1,q}}\lesssim_\alpha \|w\|_{H^{s,q}}.
    \end{align}
    Therefore, $\TT$ is a bounded linear operator from $H_{\curl}\cap H^{s,q}$ to $\TIH_{\curl}\cap H^{s,q}$. We claim that $\TT$ is indeed an isomorphism. Indeed, for any $\tw\in \TIH_{\curl}\cap H^{s,q}$, define $\Psi[\tw]$ as the solution of the following Neumann--Poisson problem
    \begin{align}\label{BDYY3}
            \Delta\Psi[\tw]=\alpha(1-2y)\tw_3,\qquad\partial_y\Psi[\tw]|_{y=0,1}=0.
    \end{align}
    Let 
    \[w:=\tw-\nabla\Psi[\tw].\]
    Then, one has $\dive w=0$ and $\int_{\T^2}w_3\ddd x=0$ for almost every $y\in(0,1)$, which implies $w\in H_{\curl}$. Since $\Psi[\tw]$ satisfies
    \[\Delta \Psi[\tw]-\alpha(1-2y)\partial_y\Psi[\tw]=\alpha(1-2y)w_3,\qquad \partial_y\Psi[\tw]|_{y=0,1}=0,\]
    by the unique solvability of \eqref{BDYY2}, one has
    \[\nabla \Psi[\tw]=\nabla \Phi[w].\]
    This yields
    \[\tw=w+\nabla\Phi[w]=\TT w,\]
    and thus $\TT$ is surjective. It remains to show $\TT$ is injective. To see this, suppose that 
    \[w+\nabla\Phi[w]=0.\]
    Then, $\Phi[w]$ satisfies 
    \[\Delta\Phi[w]=0,\qquad \partial_y\Phi[w]|_{y=0,1}=0,\]
    which gives $\nabla \Phi[w]=0$ and thus $w=0$. Consequently, $\TT$ is an isomorphism.

    Now notice that 
    \[\TT w-(I-\BB)w=\nabla\Phi[w]+\BB w,\]
    where, by \eqref{BDYY4}, Corollary \ref{coroA4}, the standard interpolation theory, and the compact Sobolev embedding theorem, the linear operator $w\mapsto \nabla\Phi[w]+\BB w$ is compact. Hence, $I-\BB$ is a compact perturbation of an isomorphism, and is thus a Fredholm operator of index zero.

    \noindent\textbf{Step I\!I\!I.} By the Fredholm alternative theorem, it remains to show $I-\BB$ is injective. To this end, suppose that $w=\BB w$. Then, 
    \begin{align}\label{BDY1}
        \begin{cases}
        \alpha(1-2y)u_1=\partial_yu_1-\partial_{x_1}u_3,\\
        \alpha(1-2y)u_2=\partial_yu_2-\partial_{x_2}u_3,\\
        0=\partial_{x_2}u_1-\partial_{x_1}u_2,\\
        \dive u=0,\\
        \int_D u_1\ddd x\ddd y=\int_D u_2\ddd x\ddd y=0,
    \end{cases}
    \end{align}
    where $u:=\curl^{-1}w$. Applying $\partial_{x_1}$ to \eqref{BDY1}$_1$ and $\partial_{x_2}$ to \eqref{BDY1}$_2$ respectively, summing the resulting equalities, and using \eqref{BDY1}$_4$, one has
    \[\Delta u_3-\alpha(1-2y)\partial_3u_3=0,\qquad u_3|_{y=0}=u_3|_{y=1}=0,\]
    which gives
    \[u_3\equiv 0.\]
    This, together with \eqref{BDY1}$_3$, implies that for each fixed $y\in (0,1)$, the two-dimensional vector field $u_h(x,y)$ is divergence-free and curl-free. Therefore, 
    \[u_h(x,y)=c(y),\]
    for some vector field $c(y)$. Combining this with \eqref{BDY1}$_1$ and \eqref{BDY1}$_2$, one has 
    \[u_h=C e^{\alpha(y-y^2)}.\]
    Plugging the above result into \eqref{BDY1}$_5$ yields $C=0$. Thus, $u\equiv 0$ and $I-\BB$ is injective.
    \end{proof}
    Next, we address some commutator estimates.
    \begin{prop}
        \label{propB2} The following assertions hold.
        \begin{enumerate}
            \item For any $m\in\N^2$, $[\BB,\partial_x^m]=0$. 
            \item For any $w\in H_{\curl}\cap C^\infty$,
            \begin{align}\label{BDY2}
            \|[\BB,\partial_y] w\|\lesssim_\alpha \|w\|,\qquad \|\BB\partial_y w\|\lesssim_\alpha \|w\|.
            \end{align}
            Moreover, if $w$ satisfies the Navier boundary conditions \eqref{INTRO1}$_3$ and \eqref{INTRO1}$_4$,  it follows that
            \begin{align}\label{BDY3}
                \|[\BB,\partial_y^2] w\|\lesssim_\alpha \|w\|.
            \end{align}
        \end{enumerate}
    \end{prop}
    \begin{proof}
        \textbf{Step I.} The first assertion is ensured by the fact $[\partial_x^m,\curl^{-1}]=0$ and the definition of $\BB$. The estimate \eqref{BDY2} is established in the following way. By applying Corollary \ref{coroA4}, it follows that
        \begin{align}\label{BDY3-1}
            \|[\BB,\partial_y]w\|&=\|\alpha(1-2y)(\curl^{-1}\partial_y w)_h-\alpha\partial_y[(1-2y)(\curl^{-1}w)_h]\|\notag\\&\lesssim_\alpha\|[\curl^{-1},\partial_y]w\|+\|w\|.
        \end{align}
        Notice that $z:=[\curl^{-1},\partial_y]w $ satisfies the div-curl system
        \begin{align*}
            \begin{cases}
                \dive z=\dive (\curl^{-1}\partial_yw-\partial_y\curl^{-1}w)=0,\\
                \curl z=\curl (\curl^{-1}\partial_yw-\partial_y\curl^{-1}w)=0,
            \end{cases}
        \end{align*}
        supplemented with the boundary conditions
        \begin{align}
            \label{BDY4}
            z_3|_{y=0,1}=-\partial_y(\curl^{-1}w)_3|_{y=0,1}=\dive_h(\curl^{-1}w)_h|_{y=0,1}.
        \end{align}
        Moreover, by applying Lemma \ref{lemmaA4}, it follows
        \begin{align}
            \label{BDY5}
            z=\nabla \phi+(c_1,c_2,0)
        \end{align}
        for some harmonic function $\phi$ and constants $c_1,c_2$. We turn to the estimate of $\nabla\phi$. From \eqref{BDY4} and \eqref{BDY5}, one gets
        \begin{align*}
            \begin{cases}
                \Delta \phi=0,\\
            \partial_y\phi|_{y=0,1}=\dive_h(\curl^{-1}w)_h|_{y=0,1}.
            \end{cases}
        \end{align*}
        Without loss of generality, one may assume $\int_D\phi\ddd x\ddd y=0.$ Then, by the standard energy estimate, it follows
        \begin{align*}
            \|\nabla \phi\|^2\lesssim_\alpha\|\curl^{-1}w\|_{H^{1/2}(\partial D)}\|\phi\|_{H^{1/2}(\partial D)}\lesssim_\alpha\|w\|\|\nabla\phi\|,
        \end{align*}
        which gives 
        \begin{align}
            \label{BDY6}
            \|\nabla\phi\|\lesssim_\alpha\|w\|.
        \end{align}
        It remains to bound $c_1,c_2$. For $j=1,2$, by the mean-zero condition required for the horizontal components in the definition \eqref{A13} of $\curl^{-1}$, it follows that 
        \begin{align*}
            |c_j|&=\left|\int_D z_j \ddd x \ddd y\right|=\left|\int_D \partial_y(\curl^{-1}w)_j \ddd x \ddd y\right|\\&\le \left|\int_{\T^2} (\curl^{-1} w)_j|_{y=0}\ddd x\right|+\left|\int_{\T^2} (\curl^{-1} w)_j|_{y=1}\ddd x\right|\\&\lesssim_\alpha \|\curl^{-1} w\|_{L^{2}(\partial D)}\lesssim_\alpha \|w\|,
        \end{align*}
        which, together with \eqref{BDY5} and \eqref{BDY6}, yields
        \begin{align}
            \label{BDY7}\|[\curl^{-1},\partial_y]w\|=\|z\|\lesssim_\alpha \|w\|.
        \end{align}
        Plugging the above estimate into \eqref{BDY3-1}, we obtain the first estimate in \eqref{BDY2}, which, together with Corollary \ref{coroA4}, ensures the second estimate therein.

        \noindent\textbf{Step I\!I.} We turn to the proof of \eqref{BDY3}. Proceeding as above, one has 
        \begin{align*}
            \|[\BB,\partial_y^2]w\|\lesssim_\alpha \|[\curl^{-1},\partial_y^2] w\|+\|\partial_y\curl^{-1}w\|\lesssim_\alpha  \|[\curl^{-1},\partial_y^2] w\|+\|w\|,
        \end{align*}
        where $\tilde z:=[\curl^{-1},\partial_y^2]w $ satisfies
        \begin{align*}
            \begin{cases}
                \dive \tilde z=\dive (\curl^{-1}\partial_y^2w-\partial_y^2\curl^{-1}w)=0,\\
                \curl \tilde z=\curl (\curl^{-1}\partial_y^2w-\partial_y^2\curl^{-1}w)=0.
            \end{cases}
        \end{align*}
        By using the Navier boundary conditions in the velocity formulation, one has the boundary conditions
        \begin{align*}
            \tilde z_3|_{y=0}=-\partial_y^2(\curl^{-1}w)_3|_{y=0}=\dive_h\partial_y(\curl^{-1}w)_h|_{y=0}=\alpha\dive_h(\curl^{-1}w)_h|_{y=0},
        \end{align*}
        and 
        \begin{align*} 
           \tilde z_3|_{y=1}=-\alpha\dive_h(\curl^{-1}w)_h|_{y=1}.
        \end{align*}
        Therefore, by repeating the derivation of \eqref{BDY7}, one gets
        \[\|[\curl^{-1},\partial^2_y]w\|=\|\tilde z\|\lesssim_\alpha \|w\|,\]
        and thus \eqref{BDY3} holds.
    \end{proof}
    \begin{coro}
        \label{coroB3} The following assertions hold.
        \begin{enumerate}
            \item For any $m\in\N^2$, $[(I-\BB)^{-1},\partial_x^m]=0$. 
            \item For any $\tw\in\TIH_{\curl}\cap C^\infty$ satisfying the mixed Dirichlet--Robin boundary condition \eqref{LS4}$_4$ and $w:=(I-\BB)^{-1}\tw$, it follows that
            \begin{align}\label{BDY8}
            \|[(I-\BB)^{-1},\partial_y] \tw\|+\|[(I-\BB)^{-1},\partial_y^2] \tw\|\lesssim_\alpha \|\tw\|,
            \end{align}
            and
            \begin{align}\label{BDY9}
                \|\partial_y w\|\lesssim_\alpha\|\partial_y\tw\|+\|\tw\|,\qquad \|\partial_y\tw\|\lesssim_{\alpha}\|\partial_y w\|+\|w\|.
            \end{align}
        \end{enumerate}
    \end{coro}
    \begin{proof}
        This follows from the operator identity 
        \[(I-\BB)^{-1}\LL-\LL(I-\BB)^{-1}=(I-\BB)^{-1}[\BB,\LL](I-\BB)^{-1}\]
        and Propositions \ref{propB1} and \ref{propB2}.
    \end{proof}

    \section{Proof of Corollary \ref{NDGWPcoro1}}\label{Duhamel}
    In this section, we prove Corollary~\ref{NDGWPcoro1}. Although the argument is standard, we include the details in order to keep track of the dependence on $\kappa$ in the a priori estimates.

    \noindent\textbf{Step I.} Let $w$ be a solution of \eqref{NDGWP38} with $f=0$, and $\tw:=(I-\BB)w$. Since $\nabla_x$ commutes with the harmonic extension operator $\HH$ defined in \eqref{SL6-1-1}, by applying Proposition \ref{GWPprop3} and the relations 
    \begin{align}
        \label{Duhamel0}\|\tw\|_{m,0}\sim_\alpha \|w\|_{m,0},\qquad \|\nabla_x\tw\|_{m,0}\sim_\alpha \|\nabla_xw\|_{m,0},
    \end{align}
    it follows
    \begin{align}\label{Duhamel1}
        \|\tw(t)\|_{m,0}^2\le C_\alpha e^{-c_\alpha \nu_\kappa t}\|\tw_{\ini}\|_{m,0}^2
    \end{align}
    for any positive integer $m$. Since $\tw$ solves \eqref{NDGWP35}, by applying the energy method used in Step~I\!I of the proof of Proposition \ref{GWPprop3}, it follows
    \begin{align}\label{Duhamel1-1}
        \frac{\ddd}{\ddd t}\|\tw\|_{m,0}^2+\nu_\kappa\|\nabla \tw\|_{m,0}^2\lesssim_\alpha \nu_{\kappa}\|\tw\|_{m,0}^2,
    \end{align}
    which, together with \eqref{Duhamel1}, gives
    \begin{align*}
        \|\tw(t)\|^2_{m,0}+\nu_\kappa\int_0^t\|\nabla \tw\|_{m,0}^2\ddd s\lesssim_\alpha \|w_{\ini}\|^2_{m,0}.
    \end{align*}
    Therefore, by using \eqref{BDY9}, \eqref{Duhamel0}, and \eqref{Duhamel1}, one gets
    \begin{align}
        \label{Duhamel2}\notag\sup_{t\le T}\|w(t)\|^2_{m,0}+\nu_\kappa\int_0^T(\|\nabla w\|_{m,0}^2+\|w\|_{m,0}^2)\ddd s&\lesssim_\alpha \|w_{\ini}\|^2_{m,0}+\nu_\kappa\int_0^T\|\tw\|_{m,0}^2\ddd s\\&\lesssim_\alpha \|w_{\ini}\|^2_{m,0}.
    \end{align}

    \noindent\textbf{Step I\!I.} In this step, we derive a smoothing estimate for $\tw$. Notice that 
    \begin{align*}
        \frac{\ddd}{\ddd t}(t\|\tw\|^2_{1,0})&=\|\tw\|^2_{1,0}+2t\langle \tw,\nu_\kappa \Delta \tw-\nu_\kappa [\BB,\partial_y^2]w+\beta_\kappa(I-\BB)\nabla\HH[n_3w_3]\rangle_{1,0}\\&\le \|\tw\|^2_{1,0}-2t\nu_{\kappa}(\|\nabla \tw\|^2_{1,0}+\alpha\|\tw_3\|^2_{H^1(\partial D)})\\&\quad+t\nu_\kappa C_\alpha\|\tw\|_{1,0}\left(\|\tw\|_{1,0}+\|\nabla \tw\|_{1,0}\right)\\&\le \|\tw\|^2_{1,0}+t\nu_\kappa C_\alpha\|\tw\|_{1,0}^2.
    \end{align*}
    For $t\in(0,\nu_{\kappa}^{-1})$, this implies  
    \begin{align*}
        t\|\tw\|^2_{1,0}\lesssim_\alpha \int_0^t(\|\nabla_x\tw\|^2+\|\tw\|^2)\ddd s\lesssim_\alpha\frac{1}{\nu_\kappa}\|\tw_{\ini}\|^2.
    \end{align*}
    As $\partial_x^k\tw$ solves \eqref{NDGWP35} with the initial datum $\partial_x^k \tw_{\ini}$, for any positive integer $m$, one has 
    \begin{align}\label{Duhamel4}
        \|\tw(t)\|^2_{m,0}\lesssim_\alpha \frac{1}{\nu_{\kappa}t}\|\tw_{\ini}\|^2_{m-1,0},\qquad t\in(0,\nu_\kappa^{-1}).
    \end{align}
    Combining this with \eqref{Duhamel1}, one obtains
    \begin{align}\label{Duhamel5}
         \|\tw(t)\|^2_{m,0}\lesssim_\alpha  e^{-c_\alpha\nu_\kappa (t-\nu_{\kappa}^{-1})}\|\tw(\nu_{\kappa}^{-1})\|_{m,0}^2\lesssim_\alpha   e^{-c_\alpha \nu_\kappa t}\|\tw_{\ini}\|_{m-1,0}^2,\qquad t \ge \nu_\kappa^{-1}.
    \end{align}
    
    \noindent\textbf{Step I\!I\!I.} Let $\hw$ be a solution of \eqref{NDGWP38} with zero initial data. The Duhamel formula gives 
    \[\hw(t)=\int_0^te^{(t-s)\LLL_{\eff}}f(s)\ddd s.\]
    Applying \eqref{Duhamel4} and \eqref{Duhamel5} combined with \eqref{Duhamel0}, one has 
    \[\|e^{(t-s)\LLL_{\eff}}f(s)\|_{m,0}\lesssim_\alpha \left(\frac{\I_{\{t-s<\nu_\kappa^{-1}\}}}{\sqrt{\nu_\kappa (t-s)}}+\I_{\{t-s\ge \nu_\kappa^{-1}\}}e^{-c_\alpha\nu_\kappa(t-s)}\right)\|f(s)\|_{m-1,0},\]
    which, together with the Young inequality, implies
    \begin{align}\label{Duhamel6}
        \int_0^T\|\hw(t)\|^2_{m,0}\ddd t&\lesssim_\alpha \left(\int_0^{\nu_\kappa^{-1}}\frac{1}{\sqrt{\nu_\kappa t}}\ddd t+\int_{\nu_{\kappa}^{-1}}^\infty e^{-c_\alpha\nu_\kappa t}\ddd t\right)^2\int_0^T\|f(t)\|^2_{m-1,0}\ddd t\notag\\&\lesssim_\alpha \frac{1}{\nu_\kappa^2}\int_0^T\|f(t)\|^2_{m-1,0}\ddd t.
    \end{align}
    Now let $\thw:=(I-\BB)\hw$. Since $\thw$ solves \eqref{NDGWP35} with zero initial data and an additional forcing $f$, by repeating the derivation of \eqref{Duhamel1-1}, one has 
    \begin{align*}
        \frac{\ddd}{\ddd t}\|\thw\|_{m,0}^2+\nu_\kappa\|\nabla \thw\|_{m,0}^2\lesssim_\alpha \nu_{\kappa}  \|\thw\|_{m,0}^2+\frac{1}{\nu_\kappa}\|f\|^2_{m-1,0},
    \end{align*}
    which, together with \eqref{Duhamel6}, yields
    \begin{align}\label{Duhamel7}
        \|\thw(t)\|_{m,0}^2+\nu_\kappa\int_0^t\|\nabla \thw\|_{m,0}^2\ddd s\lesssim_\alpha \frac{1}{\nu_\kappa}\int_0^t\|f\|^2_{m-1,0}\ddd s.
    \end{align}
    Therefore, by using \eqref{BDY9}, \eqref{Duhamel0}, and \eqref{Duhamel6}, one gets
    \begin{align*}\sup_{t\le T}\|\hw(t)\|^2_{m,0}&+\nu_\kappa\int_0^T(\|\nabla \hw\|_{m,0}^2+\| \hw\|_{m,0}^2)\ddd s\\&\quad\lesssim_\alpha \frac{1}{\nu_\kappa}\int_0^T\|f\|^2_{m-1,0}\ddd t+\nu_\kappa\int_0^T\| \hw\|_{m,0}^2\ddd s\lesssim_\alpha \frac{1}{\nu_\kappa}\int_0^T\|f\|^2_{m-1,0}\ddd t.
    \end{align*}
    This, together with \eqref{Duhamel2} and the principle of superposition, completes the proof.

\bibliographystyle{alpha}
\bibliography{reference}

@article{FL21,
 author = {Flandoli, Franco and Luo, Dejun},
 title = {High mode transport noise improves vorticity blow-up control in {3D} {Navier}-{Stokes} equations},
 fjournal = {Probability Theory and Related Fields},
 journal = {Probab. Theory Relat. Fields},
 issn = {0178-8051},
 volume = {180},
 number = {1-2},
 pages = {309--363},
 year = {2021}
}

@article{Agr25,
 author = {Agresti, Antonio},
 title = {On anomalous dissipation induced by transport noise},
 fjournal = {Mathematische Annalen},
 journal = {Math. Ann.},
 volume = {393},
 number = {3-4},
 pages = {3141--3190},
 year = {2025},
}

@article{VNVW12,
 author = {Van Neerven, Jan and Veraar, Mark and Weis, Lutz},
 title = {Stochastic maximal {{\(L^{p}\)}}-regularity},
 fjournal = {The Annals of Probability},
 journal = {Ann. Probab.},
 volume = {40},
 number = {2},
 pages = {788--812},
 year = {2012},
}

@article{DtER17,
 author = {Disser, Karoline and ter Elst, A. F. M. and Rehberg, Joachim},
 title = {On maximal parabolic regularity for non-autonomous parabolic operators},
 fjournal = {Journal of Differential Equations},
 journal = {J. Differ. Equations},
 volume = {262},
 number = {3},
 pages = {2039--2072},
 year = {2017},
}

@article{FG95,
 author = {Flandoli, Franco and Gatarek, Dariusz},
 title = {Martingale and stationary solutions for stochastic {Navier}-{Stokes} equations},
 fjournal = {Probability Theory and Related Fields},
 journal = {Probab. Theory Relat. Fields},
 issn = {0178-8051},
 volume = {102},
 number = {3},
 pages = {367--391},
 year = {1995}
}

@article{HSV24,
 author = {Helffer, Bernard and Sj{\"o}strand, Johannes and Viola, Joe},
 title = {Discussing semigroup bounds with resolvent estimates},
 fjournal = {Integral Equations and Operator Theory},
 journal = {Integral Equations Oper. Theory},
 issn = {0378-620X},
 volume = {96},
 number = {1},
 pages = {35},
 note = {Id/No 5},
 year = {2024},
 language = {English},
 doi = {10.1007/s00020-024-02754-x},
 zbMATH = {7812578}
}

@article{Gal20,
 author = {Galeati, Lucio},
 title = {On the convergence of stochastic transport equations to a deterministic parabolic one},
 fjournal = {Stochastics and Partial Differential Equations. Analysis and Computations},
 journal = {Stoch. Partial Differ. Equ., Anal. Comput.},
 issn = {2194-0401},
 volume = {8},
 number = {4},
 pages = {833--868},
 year = {2020},
}

@article{FGL21,
 author = {Flandoli, Franco and Galeati, Lucio and Luo, Dejun},
 title = {Delayed blow-up by transport noise},
 fjournal = {Communications in Partial Differential Equations},
 journal = {Commun. Partial Differ. Equations},
 issn = {0360-5302},
 volume = {46},
 number = {9},
 pages = {1757--1788},
 year = {2021},
}

@article{Agr26,
  author        = {Agresti, Antonio},
  title         = {{Global smooth solutions by transport noise of 3D Navier--Stokes equations with small hyperviscosity}},
  journal       = {Ann. Probab.},
  note          = {To appear},
  year          = {2026},
  eprint        = {2406.09267},
  archivePrefix = {arXiv},
  primaryClass  = {math.PR},
}

@article{FGL22,
  author  = {Flandoli, Franco and Galeati, Lucio and Luo, Dejun},
  title   = {Eddy heat exchange at the boundary under white noise turbulence},
  journal = {Phil. Trans. R. Soc. A},
  fjoural = {Philosophical Transactions of the Royal Society A},
  volume  = {380},
  number  = {2219},
  pages   = {20210096},
  year    = {2022},
}

@article{FL22,
  author  = {Flandoli, Franco and Luongo, Eliseo},
  title   = {Heat diffusion in a channel under white noise modeling of turbulence},
  journal = {Mathematics in Engineering},
  volume  = {4},
  number  = {4},
  pages   = {034},
  year    = {2022},
}

@article{Row24,
 author = {Rowan, Keefer},
 title = {On anomalous diffusion in the {Kraichnan} model and correlated-in-time variants},
 fjournal = {Archive for Rational Mechanics and Analysis},
 journal = {Arch. Ration. Mech. Anal.},
 volume = {248},
 number = {5},
 pages = {47},
 year = {2024},
}

@article{GL25,
 author = {Galeati, Lucio and Luo, Dejun},
 title = {Weak well-posedness by transport noise for a class of {2D} fluid dynamics equations},
 fjournal = {Journal of Functional Analysis},
 journal = {J. Funct. Anal.},
 volume = {289},
 number = {12},
 pages = {59},
 year = {2025},
}

@article{BGM25,
 author = {Bagnara, Marco and Galeati, Lucio and Maurelli, Mario},
 title = {Regularization by rough {Kraichnan} noise for the generalised {SQG} equations},
 fjournal = {Mathematische Annalen},
 journal = {Math. Ann.},
 volume = {392},
 number = {4},
 pages = {4773--4830},
 year = {2025},
}

@misc{CM23,
  author        = {Coghi, Michele and Maurelli, Mario},
  title         = {Existence and uniqueness by {Kraichnan} noise for {2D} {Euler} equations with unbounded vorticity},
  year          = {2023},
  eprint        = {2308.03216},
  archivePrefix = {arXiv},
  note          = {arXiv:2308.03216},
  primaryClass  = {math.PR}
}

@article{GGM26,
  author  = {Galeati, Lucio and Grotto, Francesco and Maurelli, Mario},
  title   = {Anomalous regularization in {Kraichnan}'s passive scalar model},
  fjournal = {Probability Theory and Related Fields},
  journal = {Probab. Theory Relat. Fields},
  year    = {2026},
}

@article{AV24,
 author = {Agresti, Antonio and Veraar, Mark},
 title = {Stochastic {Navier}--{Stokes} equations for turbulent flows in critical spaces},
 fjournal = {Communications in Mathematical Physics},
 journal = {Commun. Math. Phys.},
 volume = {405},
 number = {2},
 pages = {57},
 year = {2024},
}

@misc{AKX25,
  author        = {Aydin, Mehmet Salih and Kukavica, Igor and Xu, Fanhui},
  title         = {Almost global existence for the stochastic {Navier--Stokes} equations with small {$H^{1/2}$} data},
  year          = {2025},
  eprint        = {2501.10331},
  note   = {arXiv:2501.10331},
  archivePrefix = {arXiv},
  primaryClass  = {math.PR}
}

@misc{KX24,
  author        = {Kukavica, Igor and Xu, Fanhui},
  title         = {Global existence of the stochastic {Navier--Stokes} equations in {$L^3$} with small data},
  year          = {2024},
  note   = {arXiv:2410.02919},
  eprint        = {2410.02919},
  archivePrefix = {arXiv},
  primaryClass  = {math.PR}
}

\end{document}